%% file: main_ArXiv_Journal.tex
\documentclass[letter,twocolumn]{autart}

\input{Preamb.tex}

\begin{document}
\begin{frontmatter}
\title{
Regularization in Data-driven Predictive Control: A Convex Relaxation Perspective 
\vspace{-5mm}}
\author[First]{Xu Shang}
\author[First]{Yang Zheng}
\address[First]{Department of Electrical and Computer Engineering, University of California San Diego, \{x3shang, zhengy\}@ucsd.edu}

\begin{abstract}
This paper explores the role of regularization in data-driven predictive control (\method{DDPC}) through the lens of convex relaxation, in line with \cite{dorfler2022bridging}. Using a bi-level optimization framework, we model system identification as an inner problem and predictive control as an outer problem. Within this framework, we show that several regularized \method{DDPC} formulations, including $l_1$-norm penalties, projection-based regularizers, and a newly introduced causality-based regularizer, can be viewed as convex relaxations of their respective bi-level problems. This perspective clarifies the conceptual links between direct and indirect data-driven control and highlights how regularization implicitly enforces system identification. We further propose an approximation-based variant, \method{A-DDPC}, which approximately solves the inner problem with all identification constraints via an iterative algorithm. Numerical experiments demonstrate that \method{A-DDPC} outperforms existing regularized \method{DDPC} by reducing both bias and variance errors. These results indicate that further benefits may be obtained by applying system identification techniques to pre-process the trajectory library in nonlinear settings. Overall, our analysis contributes to a unified convex relaxation view of regularization in \method{DDPC} and sheds light on its strong empirical performance beyond linear time-invariant systems. 
\vspace{-3mm}
\end{abstract}

\thanks{This work is supported by NSF ECCS-2154650, CMMI-2320697, CAREER 2340713, and an Early Career Faculty Development Award from the Jacobs School of Engineering at UC San Diego. The material in this paper was partially presented at the 6th Annual Learning for Dynamics \& Control Conference, University of Oxford, 15-17 July, 2024 \cite{shang2024convex}.}
\end{frontmatter}

\input{1-Introduction}
\input{2-Preliminaries}

\input{3-Predictive-model}

\input{4-Regularization}
\input{5-Hybrid-preprocessing}

\input{6-Numerical-simulation}

\input{7-Conclusion}

{\small
\bibliographystyle{unsrt}
\bibliography{reference.bib}
}

\appendix

\input{Appendix}

\end{document}

%% file: Preamb.tex
\usepackage{multirow}
\usepackage{nicematrix}
\usepackage{mathbbol}
\usepackage{array}
\usepackage{amssymb,arydshln}
\usepackage{float}
\usepackage{bm}
\usepackage{multirow}
\usepackage{makecell}
\usepackage{threeparttable}
\usepackage{booktabs}
\usepackage{balance}
\usepackage{tcolorbox}
\usepackage{xcolor}
\usepackage{graphicx,color,subfigure}
\usepackage{multirow}
\usepackage{wrapfig}
\usepackage{optidef}
\usepackage{cite}
\usepackage[numbers,sort&compress]{natbib} 
\usepackage{algorithm}
\usepackage{algorithmicx}
\usepackage{algpseudocode}
\usepackage{mathrsfs}
\usepackage[colorlinks = true,
            linkcolor = blue,
            urlcolor  = blue,
            citecolor = blue,
            anchorcolor = blue]{hyperref}

\DeclareMathOperator*{\argmin}{argmin}

\newcommand{\ini}{\textnormal{ini}}
\newcommand{\col}{\textnormal{col}}
\newcommand{\f}{\textnormal{F}}
\newcommand{\p}{\textnormal{P}}

\newcommand{\Lo}{\textnormal{L}}
\newcommand{\C}{\textnormal{C}}
\newcommand{\D}{\textnormal{d}}

\newcommand{\h}{\textnormal{H}}

\newcommand{\tr}{{{\mathsf T}}}

\newcommand{\Op}{\textnormal{op}}

\newcommand{\cs}{\textnormal{c}}

\newcommand{\s}{\textnormal{s}}
\newcommand{\sca}{\textnormal{sc}}
\newcommand{\lr}{\textnormal{lr}}
\newcommand{\z}{\textnormal{z}}

\newcommand{\method}[1]{\texttt{#1}}
\newcommand{\row}{\mathrm{row}}
\newcommand{\nsp}{\mathrm{null}}

\newtheorem{theorem}{Theorem}
\newtheorem{proposition}{Proposition}

\newtheorem{lemma}{Lemma}
\newtheorem{corollary}{Corollary}

\newtheorem{definition}{Definition}
\newtheorem{example}{Example}
\newtheorem{remark}{Remark}

\makeatletter
\DeclareRobustCommand{\qed}{%
  \ifmmode % if math mode, assume display: omit penalty etc.
  \else \leavevmode\unskip\penalty9999 \hbox{}\nobreak\hfill
  \fi
  \quad\hbox{\qedsymbol}}
\newcommand{\openbox}{\leavevmode
  \hbox to.77778em{%
  \hfil\vrule
  \vbox to.675em{\hrule width.6em\vfil\hrule}%
  \vrule\hfil}}
\newcommand{\qedsymbol}{\openbox}
\newenvironment{proof}[1][\proofname]{\par
  \normalfont
  \topsep6\p@\@plus6\p@ \trivlist
  \item[\hskip\labelsep\itshape
    #1.]\ignorespaces
}{%
  \qed\endtrivlist
}
\newcommand{\proofname}{Proof}
\makeatother

%% file: 1-Introduction.tex
\section{Introduction}
There has been a surging interest in utilizing data-driven techniques to control systems with unknown dynamics~\cite{hu2023toward, markovsky2021behavioral,talebi2024policy,dorfler2023data}. Existing data-driven methods can be generally categorized into indirect and direct control techniques. The indirect data-driven control approaches typically include the sequential system identification (system ID) and model-based control~\cite{ljung1998system, chiuso2019system,zheng2021sample,kouvaritakis2016model}. This two-stage control pipeline has been widely used, especially for linear systems. 
More recently, the Koopman operator has been leveraged to construct models of unknown nonlinear systems \cite{korda2018linear, haseli2023modeling,mauroy2020koopman}, but the accuracy of such models highly depends on the choice of lifting functions that are non-trivial to select \cite{shang2024willems}. In contrast, direct data-driven control methods bypass system identification altogether and design control strategies directly from input-output data, offering practitioners a potentially more attractive end-to-end alternative~\cite{markovsky2021behavioral, dorfler2022bridging,dorfler2023data}. 

One popular direct approach is the data-driven predictive control (\method{DDPC}) \cite{markovsky2021behavioral}, which utilizes Willems' fundamental lemma~\cite{willems2005note} to construct a data-driven representation of the system and use it in receding horizon control. Early ideas along this line can be traced back to the predictive control approach based on reduced Hankel matrices \cite{yang2015data}. 
Later, one of the most influential formulations, the so-called \method{DeePC}, is proposed in \cite{coulson2019data}. \method{DeePC} is initially established for deterministic linear time-invariant (LTI) systems, and its equivalence with subspace predictive control (\method{SPC}) has been discussed in~\cite{fiedler2021relationship}. Subsequent works \cite{berberich2020data,berberich2021design} have further investigated terminal constraint design for the closed-loop stability in LTI systems. Efficient \method{DDPC} formulations for LTI systems have been developed in \cite{alsalti2024robust,o2022data}, which require less offline data and can handle potentially noisy and irregular data. For systems beyond LTI settings, \method{DDPC} is extended to linear parameter varying systems \cite{verhoek2021data} and the feedback linearizable nonlinear systems \cite{alsalti2023data}. The work \cite{lian2021nonlinear} utilizes reproducing kernel functions to lift the state and obtain a linear model in the reproducing kernel Hilbert space, while the work \cite{lazar2024basis} uses basis functions to synthesize the data-driven representation in \method{DDPC}.

\method{DeePC} and its general \method{DDPC} variants have demonstrated promising experimental results for controlling systems beyond LTI settings in various real-world applications, \emph{e.g.}, robotics \cite{elokda2021data, fawcett2022distributed}, mixed traffic systems \cite{wang2023implementation,Shang2025Decentralized}, power systems \cite{lian2023adaptive,huang2021decentralized}. The recent work~\cite{berberich2022linear} has established some theoretical guarantees for  \method{DDPC} in nonlinear systems. For non-deterministic or nonlinear systems, more offline and online collected data are needed to increase the width (\emph{i.e.}, column number) and depth (\emph{i.e.}, row number) of the trajectory library, so that an accurate enough data-driven representation can be constructed. The benefits of increasing its width are well-recognized in the literature \cite{zhang2023dimension,wang2023deep,lian2023adaptive,wang2023implementation}, and the recent works have further emphasized the importance of enlarging its depth \cite{shang2024willems, lawrence2024deep}.

\begin{figure*}[t]
    \centering
    \includegraphics[width = 0.95\textwidth]{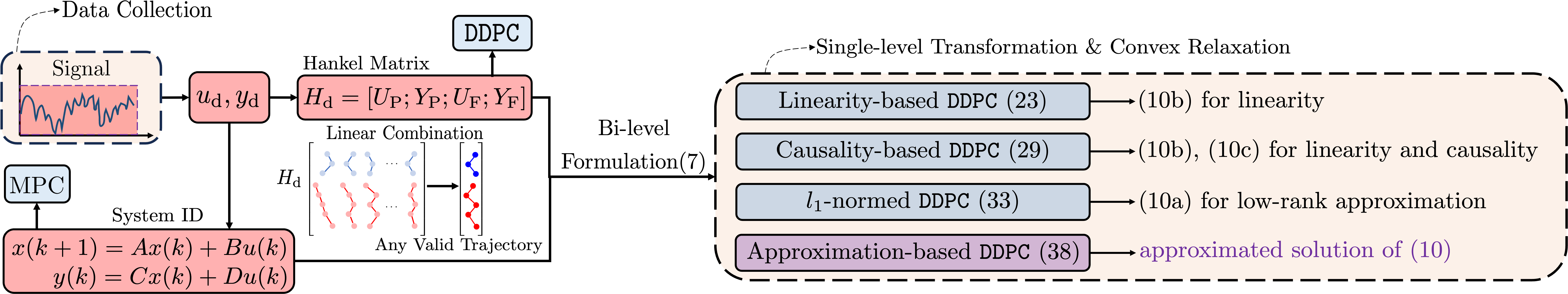}
    \caption{Schematic of data-driven predictive control (\method{DDPC}), which starts by collecting data (usually noisy) from the real system. Indirect methods identify a parametric model, while \method{DDPC} forms a Hankel matrix as the trajectory library for predictive control. The bi-level formulation \eqref{eqn:bi-level} integrates system ID techniques for trajectory library in \method{DDPC}. We introduce a series of convex relaxation \eqref{eqn:p-DeePC-p}, \eqref{eqn:p-DeePC-causal}, \eqref{eqn:l1-exact}  and an approximation \eqref{eqn:DeePC-SVD-Iter} for the bi-level formulation.}
    \label{fig:WorkFlow}
    \vspace{-2mm}
\end{figure*}

Since the original \method{DeePC} in~\cite{coulson2019data}, numerous regularization strategies for \method{DDPC} have been developed~\cite{breschi2023data,dorfler2022bridging,huang2021quadratic,huang2021decentralized,wang2023deep,sader2023causality,yin2021maximum}. Early approaches based on $l_1$- and $l_2$-norm penalties were primarily heuristic, which aim to improve empirical control performance and ensure numerical stability for systems beyond LTI settings~\cite{coulson2019data, huang2021quadratic, huang2021decentralized, wang2023deep}. A novel projection-based regularizer was later proposed in~\cite{dorfler2022bridging}, establishing a formal connection between \method{DeePC} and \method{SPC}; in fact, the two formulations are equivalent when the weighting parameter is sufficiently large. Another line of work introduced $\gamma$-\texttt{DDPC} \cite{breschi2023data}, which reformulates the original data-driven representation via an LQ factorization and introduces a new decision variable~$\gamma$. Regularization is then applied directly to $\gamma$, which offers a potentially transparent interpretation of the regularization effects. Moreover, it is shown in \cite[Theorem 2]{sader2023causality} that $\gamma$-\texttt{DDPC} with $l_2$ regularization and a sufficiently large weight is equivalent to \method{SPC}. This framework was further extended in~\cite{sader2023causality} by incorporating causality constraints and designing an associated regularizer.
More recently, \cite{yin2021maximum} proposed a maximum-likelihood estimator that characterizes future input-output trajectories through an iterative algorithm. 

The advantages and limitations of direct and indirect data-driven methods have been extensively discussed; see the editorial column~\cite{dorfler2023data} for an excellent overview. Recent works have further examined the role of regularization in connecting these two paradigms. A notable contribution is~\cite{dorfler2022bridging}, 
which analyzes the effect of regularization in a principled way via a bi-level optimization framework. In this setting, indirect data-driven control is formulated as a bi-level problem involving both identification and control, and many regularized variants of \method{DDPC} (\emph{e.g.}, $l_1$-norm penalties and projection-based regularizers) can be interpreted as convex relaxations of this formulation. A similar analysis has been carried out for $\gamma$-\texttt{DDPC} in~\cite{sader2023causality}, where the identification task serves as the inner problem. An alternative perspective is established in~\cite{kladtke2023implicit}, which interprets the regularization term as an implicit predictor: the regularizer implicitly selects a model class for the data-driven representation. The recent work~\cite{chiuso2023harnessing} introduces the concept of the final control error (FCE). The proposed FCE-\method{DDPC} minimizes this measure and includes certain regularized \method{DDPC} and $\gamma$-\texttt{DDPC} as suboptimal instances~\cite{chiuso2023harnessing}. 

In this paper, we adopt a bi-level optimization framework, similar to~\cite{dorfler2022bridging}. We aim to study the interplay between direct and indirect data-driven control for systems beyond LTI settings, especially for the effect of regularizers. In this formulation, data preprocessing is modeled as an inner optimization problem (identification), while predictive control is formulated as an outer optimization problem (online control). Figure~\ref{fig:WorkFlow} illustrates the overall process. Our main technical contributions are as follows. 

First, we highlight the role of the Hankel trajectory library $H$ as a non-parametric model within the bi-level framework. This choice not only facilitates the incorporation of common system ID constraints (\emph{e.g.}, SPC~\cite{favoreel1999spc}, low-rank approximation~\cite{markovsky2016missing}, and causal models~\cite{qin2005novel}) into the inner problem, but also enables the use of penalty methods to reduce the bi-level formulation to a single-level problem (Theorem~\ref{them:x-exact-main}). Second, under this bi-level framework, we establish three regularized \method{DDPC} formulations as convex relaxations of their corresponding bi-level optimization problems (Theorems \ref{cor:relax-spc}, \ref{cor:relax-causal}, \ref{cor:relax-lr}). In each case, the explicit projection of data onto LTI identification constraints is replaced with suitable regularizations that account for implicit identification. While this perspective aligns with~\cite{dorfler2022bridging}, we refine the analysis by clarifying conceptual inaccuracies, providing elementary self-contained proofs, and introducing a new causality-based regularizer (Theorem~\ref{cor:relax-causal}). Finally, all three regularized \method{DDPC} formulations only implicitly enforce a subset of system ID constraints while neglecting others. To address this limitation, we propose an approximation-based variant of \method{DDPC} (\method{A-DDPC}), which approximately solves the inner identification problem with all system ID constraints via an iterative algorithm. This leads to a refined data-driven representation that approximately captures linearity, causality, and the dominant trajectory space. Numerical experiments show that \method{A-DDPC} outperforms existing regularized \method{DDPC} approaches in reducing both bias and variance errors. 

Some preliminary results appeared in \cite{shang2024convex}.  In this work, we make several substantial new contributions: 1) we prove that the regularized \method{DDPC} formulations are convex relaxations of specific bi-level optimization problems and the associated regularizers are exact penalties, rather than merely convex approximations as discussed in \cite{shang2024convex}; 2) we explicitly introduce the causality requirement and systematically integrate it into the proposed bi-level optimization framework; and 3) we enhance numerical experiments for systems beyond the LTI setting. Finally, we have included the theoretical proofs omitted in \cite{shang2024convex}.

The rest of this paper is structured as follows. Section~\ref{sec:Prob-Meth} reviews the preliminaries. Section~\ref{section:bi-level-formulation} discusses the use of the trajectory matrix as a predictive model in indirect and direct approaches. In Section~\ref{sec:methods}, we present three regularized \method{DDPC} formulations as convex relaxations of suitable bi-level problems. Section~\ref{sec:SVD-Iter} develops an approximation-based \method{DDPC}. Numerical results are reported in Section~\ref{sec:results}, and Section~\ref{sec:conclusions} concludes the paper. Some technical proofs are provided in the appendix.

\textbf{Notation:} For a series of vectors $x_1, \ldots, x_n$ and matrices $A_1, \ldots, A_n$ with the same column size, we denote $\col(x_1, \ldots, x_n):= [x_1^\tr, \ldots, x_n^\tr]^\tr$ and $\col(A_1, \ldots, A_n):= [A_1^\tr, \ldots, A_n^\tr]^\tr$. We represent the $p$-norm of a vector $x$ as $\|x\|_p$ and the Frobenius norm of a matrix $A$ as $\|A\|_F$, and we denote $\|x\|_X^2 := x^\tr X x$ with a square matrix $X$. The null space and range space of $A$ are denoted by $\mathrm{Null}(A)$ and $\mathrm{Im}(A)$, respectively. We use $\mathbb{0}$ and $I$ to denote the zero matrix and identity matrix with compatible dimensions. We represent the open $p$-norm ball centered at $a \in \mathbb{R}^n$ with radius $r$ as $\mathcal{B}_p(a;r) := \{x \in \mathbb{R}^n \ | \ \|x-a\|_p < r\}$. We use $\mathcal{A}^\circ$ to denote the interior of the set $\mathcal{A}$.

%% file: 2-Preliminaries.tex
\section{Preliminaries}
\label{sec:Prob-Meth}
This section reviews model-based predictive control, the fundamental Lemma \cite{willems2005note} from a behavioral~perspective, and a basic data-driven predictive control~in~\cite{coulson2019data}. 

\subsection{LTI systems and model predictive control} \label{subsection:LTI}
Consider a discrete-time LTI system:
\vspace{-2mm}
\begin{equation}
\label{eqn:LTI-sys}
    \begin{aligned}
    \bar{x}(k+1) &= A \bar{x}(k) + B \bar{u}(k), \\
    \bar{y}(k) &= C \bar{x}(k) + D \bar{u}(k),
    \end{aligned} 
\vspace{-2mm}
\end{equation}
where the state, input, output at time $k$ are $\bar{x}(k) \in \mathbb{R}^n$, $\bar{u}(k) \in \mathbb{R}^m$, and $\bar{y}(k) \in \mathbb{R}^p$, respectively. Throughout this paper, we assume that $(A, B)$ is controllable and $(A, C)$ is observable. The lag of the LTI system \eqref{eqn:LTI-sys} is defined as the minimum integer $l \in \mathbb{N}$ such that its observability matrix $\col(C, CA, \ldots, CA^{l-1})$ has full column rank $n$. It is known that $l \leq n$ when $(A,C)$ is observable. 
 
Given a desired reference trajectory $y_\textrm{r} \in \mathbb{R}^{pN}$ with horizon $N >0$, input constraint set $\mathcal{U} \subseteq \mathbb{R}^m$, output constraint set $
\mathcal{Y} \subseteq \mathbb{R}^p$, we aim to design control inputs such that the system output tracks the reference trajectory. In model predictive control (MPC), this is achieved, at time $t$, by solving the receding horizon predictive control
\vspace{-2mm}
\begin{subequations}
\label{eqn:predictive-control}
\begin{align}
    \min_{x, u, y } \;\, & \sum_{k=t}^{t+N-1}\big(\|y(k) - y_\textrm{r}(k)\|_Q^2 + \|u(k)\|_R^2\big) \nonumber \\ 
    \mathrm{subject~to} \;\, & x(k+1) = A x(k)+ B u(k),  \label{eqn:predictive-control-a} \\
    & y(k) = {C} x(k) + {D} u(k), \label{eqn:predictive-control-b} \\
    & x(t) = x_t, \label{eqn:predictive-control-c}\\
    & u(k) \!\in\! \mathcal{U}, y(k) \!\in\! \mathcal{Y},  k = t, \ldots, t\!+\!N \!-\!1, \label{eqn:predictive-control-d}
\end{align} 
\end{subequations}
\noindent where $x_t \! \in \! \mathbb{R}^n$ is the initial state at time $t$, $x(k), u(k)$ and $y(k)$ denote the predicted state, input, and output at time $k$, and $R$ and $Q$ are positive definite cost matrices. According to \eqref{eqn:predictive-control-a}-\eqref{eqn:predictive-control-c}, $\{x(k),u(k),y(k)\}_{k=t}^{t+N-1}$ is a \textit{predicted} trajectory of the system \eqref{eqn:LTI-sys} with an initial state $x_t$ at time $t$ using the system model \eqref{eqn:LTI-sys}. Equation \eqref{eqn:predictive-control-d} enforces the constraints on the predicted trajectory, and we assume $\mathcal{U}$ and $\mathcal{Y}$ are convex sets. To simplify notation and without loss of generality, we consider a regulation problem (\emph{i.e.}, $y_\textrm{r} = \mathbb{0}_{pN}$) for the rest of the discussions. It is clear that \eqref{eqn:predictive-control} is a convex optimization problem (it is a quadratic program when $\mathcal{U}$ and $\mathcal{Y}$ are polytopes), which admits an efficient solution when the model for \eqref{eqn:LTI-sys} is known, \emph{i.e.},
matrices $A$, $B$, $C$ and $D$ are known.  

In this work, we focus on the case where both the system model \eqref{eqn:LTI-sys} and the initial condition $x_t$ are unknown. Instead, we have access to 
\begin{itemize}
    \item \textit{offline data}, \emph{i.e.}, a length-$T$ pre-collected input and output trajectory $u_\textnormal{d} = \textnormal{col}(u_\textnormal{d}(1),\ldots, u_\textnormal{d}(T)) \in \mathbb{R}^{mT}$, $
    y_\textnormal{d} = \textnormal{col}(y_\textnormal{d}(1)$, $ \ldots, y_\textnormal{d}(T)) \in \mathbb{R}^{pT}$ from \eqref{eqn:LTI-sys}; 
    \item  \textit{online data}, \emph{i.e.}, the most recent past input and output sequence of length-$T_\ini$. 
\end{itemize}

Then, problem \eqref{eqn:predictive-control} can be implemented by either indirect system identification and model-based control~\cite{kouvaritakis2016model} 
or direct data-driven predictive control, such as the so-called \method{DeePC} \cite{coulson2019data} and its related approaches~\cite{dorfler2022bridging,markovsky2021behavioral, breschi2023data}. The advantages and limitations of these two classes of methods have been extensively discussed in the literature \cite{dorfler2023data}.

\subsection{Direct data-driven predictive control} \label{subsection:DeePC}
\label{subsec:DeePC}
We here review the notion of \textit{persistent excitation} (PE), which ensures the offline data is sufficiently rich. 
\begin{definition}[Persistently Exciting]
\label{def:Hankel-struct}
    A sequence of data points $\omega  =  \textnormal{col}(\omega(1),\omega(2), \ldots, \omega(T))$ of the length $T$ is persistently exciting (PE) of order $L$ ($L < T$) if its associated Hankel matrix with depth $L$, 
    \vspace{-4mm} 
    \[
    \mathcal{H}_L(\omega) = \begin{bmatrix}
        \omega(1) & \omega(2) & \cdots & \omega(T-L+1) \\
        \omega(2) & \omega(3) & \cdots & \omega(T-L+2) \\
        \vdots    & \vdots    & \ddots & \vdots \\
        \omega(L) & \omega(L+1) & \cdots & \omega(T)
    \end{bmatrix}, \vspace{-3.5mm}
    \] 
    has full row rank. 
\end{definition}

The following result from \cite{{willems2005note}}, commonly known as \textit{the fundamental lemma}, forms the foundation of many recent results of direct data-driven predictive control. 

\begin{lemma}[Fundamental Lemma~\cite{willems2005note}] \label{lemma:Fundamental}
    Suppose~\eqref{eqn:LTI-sys} is controllable. Given a length-$T$ input/output trajectory $u_\textnormal{d} \in \mathbb{R}^{mT}$ and $
    y_\textnormal{d} \in \mathbb{R}^{pT}$ where $u_\textnormal{d}$ is PE of order $L+n$, then a length-$L$ input/output sequence $\{u_\textnormal{s}(k), y_\textnormal{s}(k)\}_{k=0}^{L-1}$ is a valid trajectory of~\eqref{eqn:LTI-sys} if and only if there exists a $g \in \mathbb{R}^{T-L+1}$ such that \vspace{-2mm}
    \begin{equation} 
    \label{eqn:FundaLemma}
    \begin{bmatrix}
    \mathcal{H}_L(u_\textnormal{d})\\
    \mathcal{H}_L(y_\textnormal{d})
    \end{bmatrix} g
    =
    \begin{bmatrix}
    u_\textnormal{s}\\
    y_\textnormal{s}
    \end{bmatrix}. \vspace{-2mm}
    \end{equation}
    If $L$ is not smaller than the lag of the system~\eqref{eqn:LTI-sys}, matrix $\textnormal{col}(\mathcal{H}_L(u_\textnormal{d}), \mathcal{H}_L(y_\textnormal{d}))$ has rank $mL+n$.
    %\markend
\end{lemma}

The fundamental lemma gives a parameterization of all finite-dimensional input/output trajectories of \eqref{eqn:LTI-sys} using only one offline trajectory $\{u_\textnormal{d},y_\textnormal{d}\}$. In particular, the image of the Hankel matrix in \eqref{eqn:FundaLemma} is the same as the set of all system trajectories of length $L$. 

With the fundamental lemma, we can use \eqref{eqn:FundaLemma} to build a predictive model to replace \eqref{eqn:predictive-control-a}-\eqref{eqn:predictive-control-c}. This is the basic idea in \method{DeePC}~\cite{coulson2019data}. In particular, we consider the input/output sequence with length $L$ satisfying $L = T_\ini  +  N$. The Hankel matrix formed by the offline data in \eqref{eqn:FundaLemma} is partitioned as
\vspace{-2mm}
 \begin{equation}
\label{eqn:Hankel_partition}
  \begin{bmatrix}
    U_{\textrm{P}} \\
    U_{\textrm{F}} 
\end{bmatrix} := \mathcal{H}_L(u_\textrm{d}), \quad 
\begin{bmatrix}
    Y_{\textrm{P}} \\
    Y_{\textrm{F}} 
\end{bmatrix} := \mathcal{H}_L(y_\textrm{d}),
\vspace{-2mm}
\end{equation}
where $U_\p$ and $U_\f$ consist of the first $T_\ini$ rows and the last $N$ rows of $\mathcal{H}_L(u_\D)$, respectively (similarly for $Y_\p$ and $Y_\f$).  We denote the most recent \textit{measured} past input trajectory of length $T_\ini$ and the \textit{designed} future input trajectory of length $N$, respectively, as 
$$
\begin{aligned}
u_\ini &= \col(u_{t-T_\textrm{ini}},u_{t-T_\textrm{ini}+1},\ldots, u_{t-1}), %\text{and} 
\\ u &= \col(u(t), u(t+1),\ldots, u(t+N-1)), 
\end{aligned}
$$
(similarly for $y_\textrm{ini}, y$). Then, Lemma \ref{lemma:Fundamental} ensures that the sequence $\col(u_\textrm{ini}, y_\textrm{ini}, u, y)$ is a valid trajectory of~\eqref{eqn:LTI-sys} if and only if there exists $g \in \mathbb{R}^{T-T_\textrm{ini} -N+1}$ such that \vspace{-2mm}
\begin{subequations}\label{eqn:predictor}
\begin{equation}
\label{eqn:predictor-a}
H_{\mathrm{d}} g= \col(u_\textrm{ini}, y_\textrm{ini}, u, y), 
\end{equation}
where, for notational simplicity, we denote  
\vspace{-2mm}
\begin{equation} \label{eqn:predictor-data-matrix}
    H_{\mathrm{d}} := \col(U_\p, Y_\p, U_\f, Y_\f), 
    \vspace{-2mm}
\end{equation}
as the Hankel matrix associated with pre-collected data $\{u_\mathrm{d}, y_\mathrm{d}\}$; see \eqref{eqn:Hankel_partition}. 
\end{subequations}
If $T_{\mathrm{ini}}$ is larger or equal to the lag of~\eqref{eqn:LTI-sys}, $y$ is unique given any $u_\textrm{ini}$, $y_\textrm{ini}$ and $u$ in \eqref{eqn:predictor}. 

The basic \method{DeePC}~\cite{coulson2019data} utilizes \eqref{eqn:predictor} as the data-driven representation of \eqref{eqn:predictive-control-a}-\eqref{eqn:predictive-control-c} and reformulate problem~\eqref{eqn:predictive-control}~as
\vspace{-2mm}
\begin{equation}
\label{eqn:DeePC}
\begin{aligned}
\min_{g,u,y} \quad & \sum_{k=t}^{t+N-1}\big(\|y(k)\|_Q^2 + \|u(k)\|_R^2\big) \\
\mathrm{subject~to} \quad & \eqref{eqn:predictor}, u \in \mathcal{U}, y \in \mathcal{Y},
\end{aligned}
\end{equation} 
where we slightly abuse the notation and use $u \in \mathcal{U}, y \in \mathcal{Y}$ to denote input/output constraints \eqref{eqn:predictive-control-d}.

\subsection{Direct vs. indirect data-driven control} 

For LTI systems with noise-free data, model-based control  \eqref{eqn:predictive-control} and \method{DeePC} \eqref{eqn:DeePC} are equivalent (see~\cite[Theorem 5.1]{coulson2019data}), thanks to the fundamental lemma. In this case, the Hankel matrix $H_{\mathrm{d}}$ in \eqref{eqn:predictor-data-matrix}, also referred to as the \textit{trajectory matrix} (since each of its columns represents a valid system trajectory), serves as a non-parametric model. 

However, data collected from practical systems is rarely noise-free. In particular, the offline data $u_{\mathrm{d}}, y_{\mathrm{d}}$ and the resulting trajectory library $H_{\mathrm{d}}$ in \eqref{eqn:predictor-data-matrix} may be corrupted by 1) ``variance'' noises that enter the process dynamics and output measurement,  2) and ``bias'' errors induced by nonlinear dynamics beyond LTI \cite{dorfler2022bridging}. To address these issues, various regularization and data preprocessing techniques have been proposed to extend the basic \method{DeePC} \eqref{eqn:DeePC}. These include $l_1$/$l_2$ regularization \cite{dorfler2022bridging}, $\gamma$-\texttt{DDPC} \cite{sader2023causality}, low-rank approximation~\cite{markovsky2016missing}, singular-value decomposition~\cite{zhang2023dimension}. While some works \cite{zhang2023dimension,dorfler2022bridging,mattsson2024equivalence,sader2023causality} have explored the relationship among different schemes, most of them are carried out on a case-by-case basis. One notable exception~is~\cite{dorfler2022bridging}, which employs a principled bi-level optimization framework to investigate the interplay between direct and indirect data-driven control. 

We here adopt the same bi-level optimization principle as~\cite{dorfler2022bridging} to systematically explore the effects of regularization %and dimension reduction 
in direct data-driven predictive control (\method{DDPC}). In particular, we consider a nested bi-level \method{DDPC}, \vspace{-2mm}
\begin{subequations} 
\label{eqn:bi-level}
\begin{align}
    \min_{\substack{g,\sigma_y \in \Gamma, 
    \\ u \in \mathcal{U} ,y \in \mathcal{Y}}} \quad & \|y\|_Q^2 + \|u\|_R^2  + \lambda_y \|\sigma_y\|_2^2  \nonumber 
    \\
    \mathrm{subject~to} \quad & {H}^* g=\col(u_\ini, y_\ini+\sigma_y, u, y ), \label{eqn:bi-level-pred} \\
    &\textrm{where} \quad  {H}^* \in \arg \min_{{H} \in \mathcal{S}} \; J_{\mathrm{id}}({H},H_{\mathrm{d}}). \label{eqn:bi-level-inner} 
\end{align} 
\vspace{-4mm}
\end{subequations} \\
The inner problem \eqref{eqn:bi-level-inner} processes the offline data $H_{\mathrm{d}}$ (see \eqref{eqn:predictor}), which may be collected from non-deterministic LTI systems or nonlinear systems, and then the resulting predictor is used in the outer online control \eqref{eqn:bi-level-pred}. This bi-level formulation, often referred to as indirect \method{DDPC}, is modular and consists of two well-separated components: an inner identification layer and an outer model-based predictive control layer. Before discussing its connections and distinctions with existing schemes, let us further clarify the notation in \eqref{eqn:bi-level}. 
In the outer problem, the term $\|y\|_Q^2+\|u\|_R^2$ denotes the usual one $\sum_{k=t}^{t+N-1}\big(\|y(k)\|_Q^2 + \|u(k)\|_R^2\big)$, and we have introduced a slack variable $\sigma_y$ with constraint $\Gamma$ and a regularization term $\|\sigma_y\|_2^2$ with $\lambda_y > 0$. This ensures constraint feasibility despite noise, as used in~\cite{markovsky2021behavioral,berberich2020data}.  
In the inner problem, $J_{\mathrm{id}}(\cdot,\cdot)$ denotes a suitable loss function, and the constraint ${H} \in\mathcal{S}$ enforces prior data structures. We will clarify our choice of ${H} \in\mathcal{S}$ in Section \ref{section:bi-level-formulation}. 

\begin{remark}[Comparison with \cite{dorfler2022bridging}]  \label{remark:comparison}
    The~bi-level~\method{DDPC} \eqref{eqn:bi-level} is inspired by the work \cite{dorfler2022bridging}, which argues that ``direct and regularized data-driven control can be derived as a convex relaxation of the indirect approach''. While adopting the same general framework, our work introduces two key distinctions. Conceptually, the~bi-level \method{DDPC} \eqref{eqn:bi-level} explicitly emphasizes the use of a \textit{non-parametric} trajectory matrix as the predictive model, whereas \cite{dorfler2022bridging} primarily considers a more general and abstract behavioral LTI setting in the inner problem. Technically, the analysis in \cite{dorfler2022bridging} 
    relies on arguments that require additional assumptions or clarifications, and some claims may not hold without these refinements. Also, the exact penalty arguments in \cite{dorfler2022bridging} rely on \cite[Proposition 2.4.3]{clarke1990optimization}, which requires adaptation in the \method{DDPC} context. Inspired by the notion of partial calmness, our proof is transparent and self-contained. We will clarify these points in more detail in Sections \ref{section:bi-level-formulation}~and~\ref{sec:methods}. 
\end{remark}

%% file: 3-Predictive-model.tex
\section{Trajectory matrix as a predictive model: indirect and direct approaches}  \label{section:bi-level-formulation}

In this section, we first detail the inner identification~\eqref{eqn:bi-level-inner} and then discuss how to bridge the bi-level \method{DDPC}~\eqref{eqn:bi-level} with a single-level optimization via penalty methods. 

\subsection{Behavior-based identification for predictive control} 
From the offline noisy data $H_{\mathrm{d}}$, we aim to get a new trajectory library $H^*$ for predictive control
\vspace{-1mm} 
\[
{H}^*:=\col({U}_\p^*, {Y}_\p^*, {U}_\f^*, {Y}_\f^*),  
\vspace{-1mm}
\]
where each column of $H^*$ is an approximated trajectory of the system ``purified'' from the data $H_{\mathrm{d}}$. Let $\mathcal{W}_L$ (where $L = T_\ini + N$) denotes the space of all possible length-$L$ trajectories  of an LTI system,~\emph{i.e.}, 
\vspace{-2mm}
\[
\mathcal{W}_L = \left\{\!\begin{bmatrix}
            \bar{u}\\\bar{y}
        \end{bmatrix}  \mid   \exists x_0 \in \mathbb{R}^{n}, \; \eqref{eqn:LTI-sys} \; \text{holds with} \ \bar{x}(0) = x_0 \!\right\}\!.
\vspace{-1mm}
\]
Ideally, we may consider the following problem 
\vspace{-2mm}
\begin{equation}
    \label{eqn:identification-traj-space}
    \begin{aligned}
       H^* = \argmin_{H} \quad & J_{\mathrm{id}}({H},H_{\mathrm{d}}) := \|H - H_{\mathrm{d}}\|_F^2 \\
        \mathrm{subject~to} \quad &  \mathrm{Im}(H) \subseteq \mathcal{W}_L.
    \end{aligned}
\vspace{-1mm} 
\end{equation}
If the offline data $H_{\mathrm{d}}$ is noise-free and satisfies the persistent excitation, the fundamental lemma ensures that $\mathrm{Im}(H_\mathrm{d}) = \mathcal{W}_L$. In this case, the optimal solution to \eqref{eqn:identification-traj-space} is $H^* = {H}_\mathrm{d}$. 
In general, the constraint $\mathrm{Im}(H) \subseteq \mathcal{W}_L$ is too abstract and not tractable. For an explicit expression, we have several necessary conditions. 
\begin{itemize}
    \item First, the trajectory matrix $H$ should be of \textit{low rank}, \emph{i.e.}, we have $\mathrm{rank}\, H^* \leq mL + n$. 
    \item Second, the trajectory matrix $H$ should satisfy \textit{linearity}, \emph{i.e.}, the future output is a linear function of the past data and future input, and we have 
    \vspace{-2mm} 
    \begin{equation}  \label{eq:linearity}
    {Y}_\f^* = K \ \col(U_\p, Y_\p, U_\f),
    \vspace{-2mm} 
    \end{equation}
    for some coefficient matrix $K$.  
    \item Third, the trajectory matrix $H$ should satisfy \textit{causality}, \emph{i.e.}, the coefficient matrix $K$ in \eqref{eq:linearity} should have a block partition  $K = \begin{bmatrix}
        K_{\mathrm{p}} \; K_{\mathrm{f}}
    \end{bmatrix}$, and $K_{\mathrm{f}} \in \mathcal{L}$ encodes a block-lower triangle structure for causality. 
    \item Finally, we may also impose a Hankel structure on $H \in \mathcal{H}$ if $H_\D$ has a Hankel structure. That ensures the approximated trajectory library can be formulated, using the same sliding-window construction as $H_{\mathrm d}$, by a single long trajectory from some finite-dimensional LTI system. Representing the trajectory library with a Hankel structure is widely used in both theory \cite{berberich2020data,berberich2022linear} and practical applications \cite{elokda2021data,wang2023implementation} of \method{DDPC}, although these works do not incorporate other conditions.  
\end{itemize}

We then relax \eqref{eqn:identification-traj-space} with the following  problem
\vspace{-2mm} 
\begin{subequations} 
\label{eqn:identification}
\begin{align}
     ({H}^*, K^*) \in \argmin_{\tilde{H}, K} \quad 
 &\|\tilde{H} - H_{\mathrm{d}}\|_F^2 \nonumber  \\
    \mathrm{subject~to} \quad & \mathrm{rank}(\tilde{H}) = mL+n \label{eqn:id-rank}, \\
    &  \tilde{Y}_\f = K \ \col(U_\p, Y_\p, U_\f), \label{eqn:id-rowsp} \\
    & K = \begin{bmatrix}
        K_{\mathrm{p}} \quad K_{\mathrm{f}}
    \end{bmatrix}, \ K_{\mathrm{f}} \in \mathcal{L},  \label{eqn:id-causal} \\ % \quad \text{(Causality)} \\
    & \tilde{H} \in \mathcal{H}. \label{eqn:id-Hankel}
\end{align}
\vspace{-5mm}
\end{subequations} \\ 
We thus consider \eqref{eqn:identification} as our inner identification in \eqref{eqn:bi-level-inner}. Note that the requirements of linearity \eqref{eqn:id-rowsp} and causality \eqref{eqn:id-causal} are typically used in causal SPC \cite{favoreel1999spc,fiedler2021relationship}, and the requirement of low rank \eqref{eqn:id-rank} and Hankel structure \eqref{eqn:id-Hankel} are often employed in low-rank approximation \cite{markovsky2016missing}. When the offline data $H_\mathrm{d}$ in \eqref{eqn:predictor} is collected from an LTI system \eqref{eqn:LTI-sys} with noise-free data, the unique optimal solution to \eqref{eqn:identification} is trivially ${H}^* = H_\mathrm{d}$. If $H_\mathrm{d}$ contains ``variance'' noises and/or ``bias'' errors, enforcing all constraints simultaneously makes the identification  \eqref{eqn:identification} difficult to solve due to its nonconvexity. 
We will demonstrate two strategies to address this challenge: 1) simplifying \eqref{eqn:identification} via only considering a subset of the constraints (Section \ref{sec:methods}) and 2) approximately solving \eqref{eqn:identification} via addressing each constraint sequentially (Section \ref{sec:SVD-Iter}). 

\begin{remark}
\label{rem: diff-bi-form}
    The inner problem \eqref{eqn:identification} aims to identify a trajectory matrix as a predictive model, which is slightly different from \cite{dorfler2022bridging}. Only constraints $\eqref{eqn:id-rank}, \eqref{eqn:id-rowsp}$ are discussed individually in \cite{dorfler2022bridging}. Problem \eqref{eqn:identification} is also related to the classical SPC. However, SPC focuses on an explicit multi-step predictor as $y = [K_p, K_f]\col(u_{\ini}, y_{\ini}, u)$. In contrast, our approach leverages a Hankel matrix as a non-parametric multi-step predictor. This shift enables the incorporation of additional constraints into \eqref{eqn:identification} and allows us to leverage closed-loop guarantees, either by including terminal ingredients \cite{berberich2024overview} or by ensuring a sufficiently long prediction horizon \cite{bongard2022robust} in \eqref{eqn:bi-level}. When $H_\D$ comes from nonlinear systems, it generally contains the ``bias" error. In this case, we can further replace the system order $n$ in \eqref{eqn:id-rank} with a tunable parameter $n_\z$ (\emph{i.e.}, the right-hand side of \eqref{eqn:id-rank} becomes $mL+n_\z$ and $n_\z \ge n$). Then, ${H}^*$ may implicitly correspond to a high-dimensional linear system, which can be obtained from Koopman lifting techniques \cite{shang2024willems,shang2025dictionary}.
\end{remark}

\subsection{A direct approach via penalty methods}
\label{subsec:bi-to-single}

With the inner identification \eqref{eqn:identification}, the bi-level formulation \eqref{eqn:bi-level} is an indirect \method{DDPC}. Using penalty methods, the bi-level structure can be transformed into a single-level problem. This is a standard idea in bi-level optimization~\cite{ye1997exact} and has been recently used to analyze direct data-driven control in \cite{dorfler2022bridging}. We adopt a similar strategy here, but also highlight a subtle point that leads to some conceptual inaccuracies in \cite{dorfler2022bridging}.  

A key step involves replacing the inner optimization problem \eqref{eqn:identification} with equivalent (but implicit) constraints, which are then incorporated into the objective function via penalty methods. One conceptually simple way is to assume the existence of a set of optimality conditions  
$
H \in \mathcal{C}_{\mathrm{opt}},
$
where $\mathcal{C}_{\mathrm{opt}}$ denotes the optimality constraints for the inner problem~\eqref{eqn:identification}. Then, we can write problem~\eqref{eqn:bi-level} equivalently~as 
\vspace{-1mm}
\begin{subequations}
\label{eqn:bi-to-single}
\begin{align}
    \min_{\substack{g,\sigma_y \in \Gamma, {H}, \\ u \in \mathcal{U} ,y \in \mathcal{Y}}} \quad & \|y\|_Q^2 + \|u\|_R^2  + \lambda_y \|\sigma_y\|_2^2 \nonumber
    \\
    \mathrm{subject~to} \quad & {H} g=\col(u_\ini, y_\ini+\sigma_y, u, y ), \label{eqn:single-predict} \\
    &  {H} \in  \mathcal{C}_{\mathrm{opt}}.  \label{eqn:inner-opt}
\end{align}
\vspace{-5mm}
\end{subequations}\\
We note that \eqref{eqn:bi-to-single} is only of a conceptual use at this stage, as \eqref{eqn:inner-opt} may not be known explicitly. We next consider a continuous penalty function satisfying 
\vspace{-1mm}
\begin{equation}
\label{eqn:penalty}
\left\{ 
\begin{aligned}
p(H) = 0, & \quad \textrm{if} \ {H} \in  \mathcal{C}_{\mathrm{opt}}, \\
p(H) >0 , & \quad \textrm{if} \ {H} \notin  \mathcal{C}_{\mathrm{opt}}.
\end{aligned} \right. 
\vspace{-1mm}
\end{equation}
Now, we can write a penalized problem 
\begin{equation} \label{eq:penalized-problem}
    \begin{aligned}
    \min_{\substack{g,\sigma_y \in \Gamma, {H}, \\ u \in \mathcal{U} ,y \in \mathcal{Y}}} \quad & \|y\|_Q^2 + \|u\|_R^2  + \lambda_y \|\sigma_y\|_2^2 + \lambda p(H)
    \\
    \mathrm{subject~to} \quad & {H} g=\col(u_\ini, y_\ini+\sigma_y, u, y ). 
\end{aligned}
\end{equation}
Under a mild assumption (see \emph{e.g.}, \cite[Theorem 6.6]{ruszczynski2011nonlinear}), we can ensure that as the penalty parameter $\lambda \to \infty$, the optimal solution of \eqref{eq:penalized-problem} converges to an optimal solution of \eqref{eqn:bi-to-single}, \emph{i.e.}, the original bi-level problem \eqref{eqn:bi-level}. If we use a small penalty parameter $\lambda$, the single-level formulation \eqref{eq:penalized-problem} becomes a relaxation of \eqref{eqn:bi-level}, \emph{i.e.}, the optimal value of \eqref{eq:penalized-problem} is smaller than that in \eqref{eqn:bi-level}. 

The reasoning above critically relies on the conceptual optimality constraint $H \in \mathcal{C}_{\mathrm{opt}}$, associated with the inner problem \eqref{eqn:bi-level-inner}. {Dropping any constraint in \eqref{eqn:bi-level-inner} results in a different optimality set, meaning that the corresponding penalized single-level problem may no longer offer a valid relaxation. This subtle but important point may be easily overlooked; for example, some relaxation statements in \cite{dorfler2022bridging} are not fully accurate without further clarification. We consider a simple example below.}
\begin{example} \label{example:bi-level}
    Consider a simple bi-level formulation 
    \vspace{-2mm}
   \begin{subequations}
\label{eqn:bi-level-example}
\begin{align}
    \min_{x \in \mathbb{R}} \quad & x^2  \nonumber 
    \\
    \mathrm{subject~to} \quad & y^* x=1, \label{eqn:bi-level-outer-ex}  \\
    &\mathrm{where} \;  y^* \in \arg \min_{y \in \mathbb{R}} \; y^2 \label{eqn:bi-level-inner-ex} \\
    &\qquad \qquad \mathrm{subject~to} \;\;  2 \leq y \leq 4. \label{eqn:bi-level-inner-constraint}
\end{align} 
\vspace{-5mm}
\end{subequations} \\
This inner problem has an optimality condition as $y^* \in \mathcal{C}_{\mathrm{opt}} := \{2\}$. We consider a penalty function $p(y) = | y-2 |$, which satisfies \eqref{eqn:penalty}. Then, a penalized single-level optimization is
\vspace{-2mm}
\begin{equation} \label{eq:single-level-ex}
    \begin{aligned}
         \min_{x \in \mathbb{R}, y \in \mathbb{R}} \quad & x^2 + \lambda |y - 2|  
    \\
    \mathrm{subject~to} \quad & y x=1.   \\
    \end{aligned}
\vspace{-2mm}
\end{equation}
It can be verified that when $\lambda \ge \frac{1}{4}$ (see Fig. \ref{fig:pen_optimal} for illustration), \eqref{eq:single-level-ex} has the same optimal solution as \eqref{eqn:bi-level-example}. If $\lambda < \frac{1}{4}$, the optimal value of \eqref{eq:single-level-ex} is a lower bound of~\eqref{eqn:bi-level-example}. If we relax the constraint \eqref{eqn:bi-level-inner-constraint} as $y\leq 4$, then the inner problem has a different optimal solution $y^* = 0$. In this case, the outer problem in \eqref{eqn:bi-level-outer-ex} becomes infeasible. This illustrates that dropping constraints from the inner problem does not necessarily reduce the optimal value of the outer problem. This nuance was overlooked in \cite{dorfler2022bridging}.
\end{example}

\subsection{Shifting regularization from $H$ to $g$} \label{subsection:regularization-on-g}

The direct data-driven method \eqref{eq:penalized-problem} is conceptually useful but not practically implementable for two reasons: 1) the regularization is imposed on the Hankel matrix $H$, which is implicit; and 2) the constraint in \eqref{eq:penalized-problem} is bilinear in $H$ and $g$, which is computationally intractable. We here discuss how to shift the regularization from $H$ to $g$.

\begin{figure}[t]
\setlength{\abovedisplayskip}{0pt}
\centering
{\includegraphics[width=0.45\textwidth]{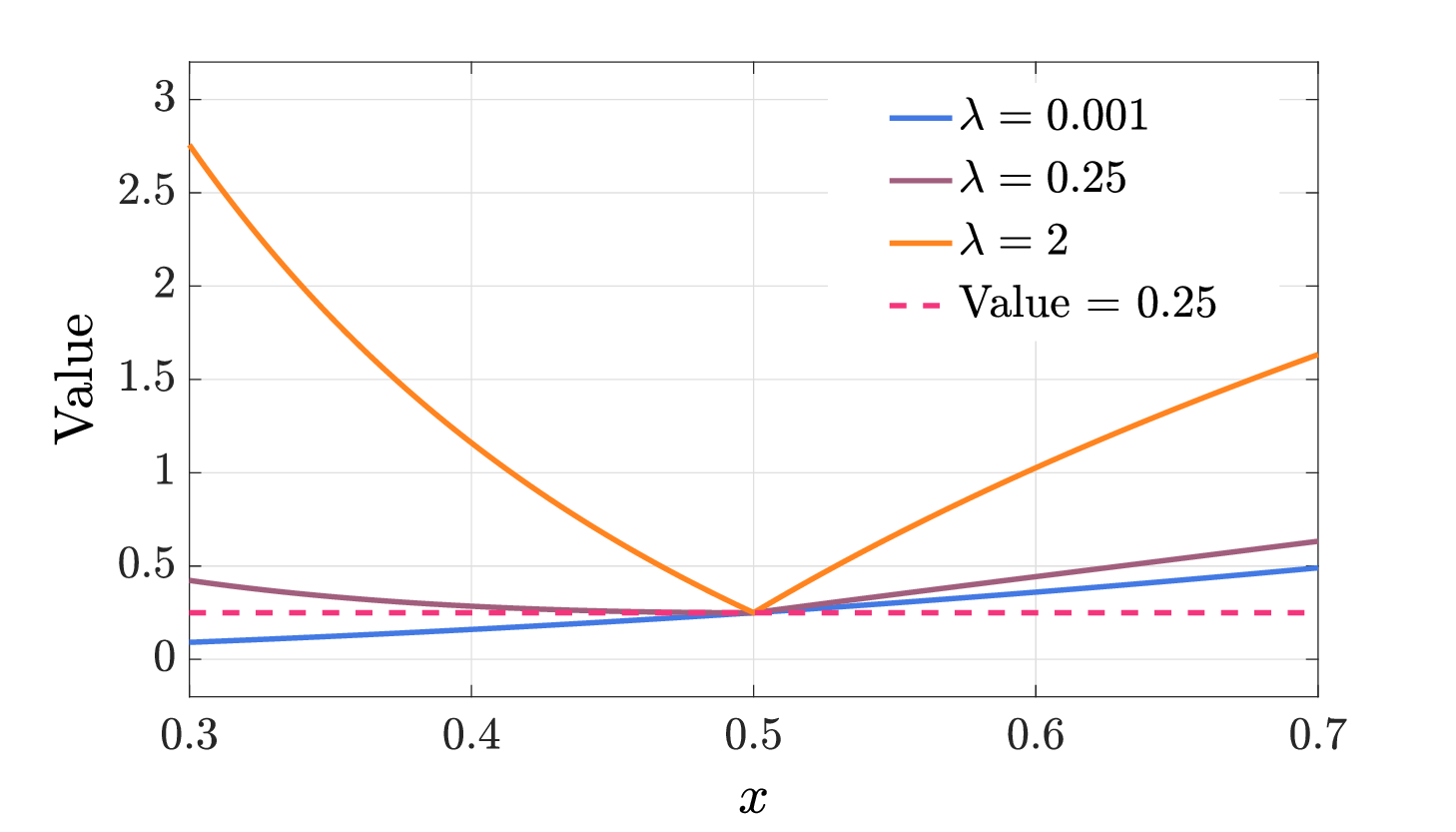}}
\vspace{-2mm}
\caption{{Influence of the penalty parameter $\lambda$ in  \eqref{eq:single-level-ex}, where we have replaced $y$ with $1/x$.}}
\label{fig:pen_optimal}
\end{figure}

Our general idea is to select some identification constraints from \eqref{eqn:id-rank}-\eqref{eqn:id-Hankel} such that the inner identification problem admits a simpler and potentially explicit optimality condition $\hat{\mathcal{C}}_{\mathrm{opt}}$. Then, we can impose some constraints $h:\mathbb{R}^{\bar{n}} \rightarrow \mathbb{R}$ on $g$ such that the set
\vspace{-1mm}
$$
\{(g,\sigma_y, u, y) \mid Hg=\col(u_\ini, y_\ini+\sigma_y, u, y ),  {H} \in  \hat{\mathcal{C}}_{\mathrm{opt}}\} 
\vspace{-1mm}
$$
is the same as 
\vspace{-1mm}
$$
    \{(g,\sigma_y, u, y) \mid \bar{H}g=\col(u_\ini, y_\ini+\sigma_y, u, y ), h(g) = 0 \}, 
\vspace{-1mm}
$$
where $\bar{H}$ is a fixed trajectory matrix. We next move the constraint on $g$ to the cost via regularization, leading to a direct \method{DDPC} 
\vspace{-1mm}
\begin{equation} \label{eq:penalized-problem-on-g}
    \begin{aligned}
    \min_{\substack{g,\sigma_y \in \Gamma, \\ u \in \mathcal{U} ,y \in \mathcal{Y}}} \quad & \|y\|_Q^2 + \|u\|_R^2  + \lambda_y \|\sigma_y\|_2^2 + \lambda_w |h(g)| 
    \\
    \mathrm{subject~to} \quad & \bar{H} g=\col(u_\ini, y_\ini+\sigma_y, u, y ). 
\end{aligned}
\vspace{-1mm}
\end{equation}
Section \ref{sec:methods} will detail three ways to derive~an explicit form~\eqref{eq:penalized-problem-on-g}, including regularization~using~projection-based norm, causality-based norm, and $l_1$ norm. 

Similar to \eqref{eq:penalized-problem}, to ensure the equivalence of the optimal solution, one may require the regularization parameter $\lambda_w$ in \eqref{eq:penalized-problem-on-g} to approach infinity. Fortunately, in our context, we can establish an exact penalty with a finite regularization parameter $\lambda_w$. In particular, consider a quadratic optimization problem of the form 
\vspace{-2mm}
\begin{equation}
\label{eqn:qp-c-main}
\begin{aligned}
 \min_{x_1 \in \mathcal{X}, x_2}  \ &  x_1^\tr M x_1  \\
\mathrm{subject~to} \ & A_1 x_1 + A_2 x_2 = b, \\
& \|Dx_2\|_p = 0, 
\end{aligned}
\vspace{-2mm}
\end{equation}
where $M$ is positive semidefinite, $A_1$, $A_2$, $b$, and $D$ are problem data of compatible dimensions, $\mathcal{X}$ is a convex set, and $\|\cdot\|_p$ is any $p$-norm. We note that $x_2$ plays the role of $g$ in \eqref{eq:penalized-problem-on-g}. 
We then consider a regularized version 
\vspace{-2mm}
\begin{equation}
\label{eqn:qp-pn-main}
\begin{aligned}
 \min_{x_1 \in \mathcal{X}, x_2}  \ &  x_1^\tr M x_1 + \lambda_w \|Dx_2\|_p  \\
\mathrm{subject~to} \ & A_1 x_1 + A_2 x_2 = b. \\
\end{aligned}
\vspace{-2mm}
\end{equation}

We next show that under mild assumptions, there exists a finite $\lambda_w >0$ such that \eqref{eqn:qp-c-main} and \eqref{eqn:qp-pn-main} are equivalent.

\begin{theorem}
\label{them:x-exact-main}
    Consider \eqref{eqn:qp-c-main} and \eqref{eqn:qp-pn-main}, where $\mathcal{X}$ is a convex set and $M$ is positive semidefinite. Then, for any fixed $\lambda_w \geq 0$, \eqref{eqn:qp-pn-main} is a convex relaxation of \eqref{eqn:qp-c-main}, \emph{i.e.}, the optimal value of \eqref{eqn:qp-pn-main} is no larger than that of \eqref{eqn:qp-c-main}. Further, if there exists an optimal solution $(x_1^*, x_2^*)$ for \eqref{eqn:qp-c-main} such that $x_1^* \in \mathcal{X}^\circ$, then there exists a finite $\lambda_w^* > 0$, such that \eqref{eqn:qp-c-main} and \eqref{eqn:qp-pn-main} are equivalent for all $\lambda_w > \lambda_w^*$, \emph{i.e.}, they have the same optimal solutions.
\end{theorem}

It is clear that \eqref{eqn:qp-pn-main} is a convex relaxation of \eqref{eqn:qp-c-main} since every optimal solution of \eqref{eqn:qp-c-main} is feasible for \eqref{eqn:qp-pn-main} with the same objective value. One key fact in this theorem is that we can find a finite penalty parameter such that \eqref{eqn:qp-c-main} and \eqref{eqn:qp-pn-main} are equivalent. Theorem \ref{them:x-exact-main} confirms that $\|Dx_2\|_p$ is an exact penalty function for the quadratic problem~\eqref{eqn:qp-c-main}. Our proof adapts the notion of partial calmness \cite{ye1997exact}. Let $f(x) = x^\tr M x$. The key idea is to show that any feasible point $(x_1, x_2)$ of the penalized problem around a local region of $(x_1^*, x_2^*)$ can be ``repaired'' to a nearby feasible point $(\tilde{x}_1, \tilde{x}_2)$ satisfying $Dx_2=0$. Then, we illustrate the difference between $f(x_1)$ and $f(\tilde{x}_1)$ is controlled proportionally to $\|Dx_2\|_p$. Combining it with the fact that $f(\tilde{x}_1) \ge f(x_1^*)$ yields a local lower bound of the form $f(x_1) + \lambda_w \|Dx_2\|_p\ge f(x_1^*)+(\lambda_w-\lambda_w^*)\|Dx_2\|_p$, implying that for any $\lambda_w>\lambda_w^*$ it is never optimal to keep a nonzero constraint violation. Finally, the convexity of \eqref{eqn:qp-pn-main} promotes this local argument to a global equivalence. We present the full proof in Appendix \ref{appendix:exact-penalty-proof}.  

To illustrate Theorem~\ref{them:x-exact-main}, we randomly generate problem instances of \eqref{eqn:qp-c-main} and \eqref{eqn:qp-pn-main}. As shown in Figure \ref{fig:pen_p_norm}, all $p$-norms render $\|Dx_2\|_p$ an exact penalty. Note that smaller values of $p$ lead to faster convergence of the optimal value of \eqref{eqn:qp-pn-main} toward that of \eqref{eqn:qp-c-main} as $\lambda_w$ increases. As $\|Dx_2\|_p$ decreases monotonically with $p \in [1, \infty]$, larger $p$ values impose weaker penalties and require larger $\lambda_w$ to guarantee equivalence between the two problems.

%% file: 4-Regularization.tex
\section{Regularization in direct data-driven control via convex relaxations}
\label{sec:methods}

In this section, we apply Theorem \ref{them:x-exact-main} and derive three direct \method{DDPC} strategies via convex relaxations, including regularization using projection-based norm, causality-based norm, and $l_1$ norm. The~projection-based norm and $l_1$ norm are widely used in the literature (see \cite{dorfler2022bridging} and its references), and the causality-based norm is~new.

As indicated in Section \ref{subsection:regularization-on-g}, our general idea is to select some identification constraints from \eqref{eqn:id-rank}-\eqref{eqn:id-Hankel} such that the inner identification problem admits a simpler and potentially explicit optimality condition.

\subsection{Projection-based norm and a Hankel-form SPC} 
\label{subsec:SPC}

We first consider the following bi-level problem 
\vspace{-2mm}
\begin{align}
    \min_{\substack{g,\sigma_y \in \Gamma, %{\color{red}\tilde{H} \in \mathcal{S}_\D}
    \\ u \in \mathcal{U} ,y \in \mathcal{Y}}} \quad & \|y\|_Q^2 + \|u\|_R^2  + \lambda_y \|\sigma_y\|_2^2  \nonumber 
    \\
    \mathrm{subject~to} \quad & {H}^* g=\col(u_\ini, y_\ini+\sigma_y, u, y ),  \label{eqn:bi-level-SPC} \\
    &\textrm{where} \;\;  {H}^* \in \argmin_{\tilde{H}, K} \;\; 
 \|\tilde{H} - H_{\mathrm{d}}\|_F^2 \nonumber  \\
   & \qquad \qquad \mathrm{subject~to} \;\;  \eqref{eqn:id-rowsp}, \nonumber
   \vspace{-7mm}
\end{align} 
which only considers the SPC constraint \eqref{eqn:id-rowsp} for the inner identification. This bi-level \method{DDPC} \eqref{eqn:bi-level-SPC} is much simpler than the general one \eqref{eqn:bi-level}. However, for a fixed $\lambda_y > 0$, the optimal value of \eqref{eqn:bi-level-SPC} is not necessarily lower than that of \eqref{eqn:bi-level}. This is because dropping constraints in the inner problem \eqref{eqn:identification} does not necessarily enlarge its optimal solution set, which constitutes the feasible region of the outer problem. Therefore, similar to Example~\ref{example:bi-level}, \eqref{eqn:bi-level-SPC} is not a relaxation of \eqref{eqn:bi-level}.

\begin{figure}[t]
\setlength{\abovedisplayskip}{0pt}
\centering
{\includegraphics[width=0.45\textwidth]{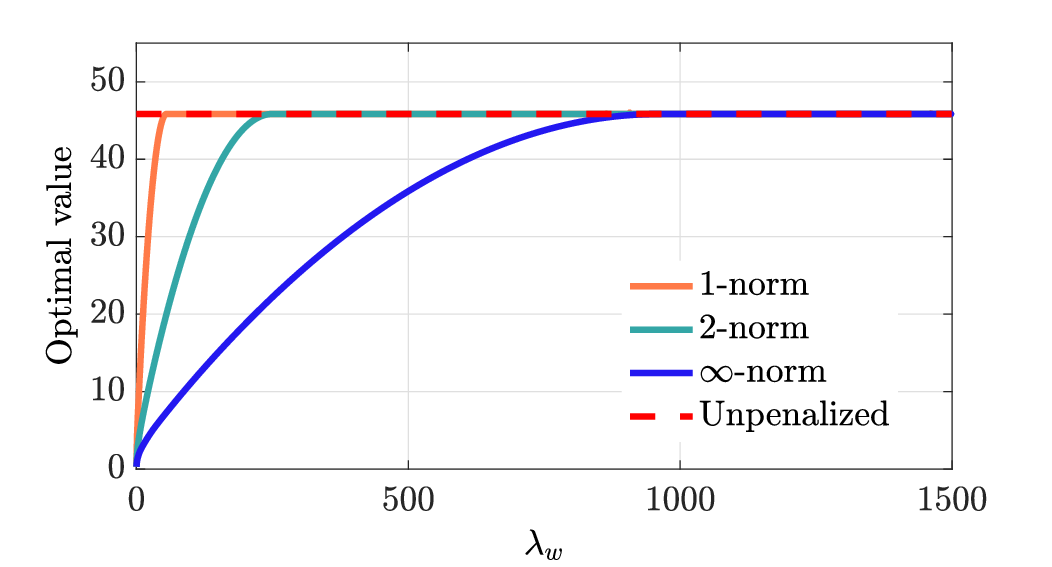}}
\vspace{-2mm}
\caption{Numerical illustration of Theorem \ref{them:x-exact-main}. The optimal value of \eqref{eqn:qp-pn-main} increases with larger weight $\lambda_w$ across different $p$-norms ($p=1,2,\infty$). In all cases, once $\lambda_w$ exceeds a threshold, the optimal value of \eqref{eqn:qp-pn-main} coincides with that of \eqref{eqn:qp-c-main}.
}
\label{fig:pen_p_norm}
\end{figure}

It is not difficult to derive the unique optimal solution 
$H^*_\s$ for the inner problem in \eqref{eqn:bi-level-SPC}, which can be obtained via Moore–Penrose inverse as  
${H}^*_\s = \col(U_\p, Y_\p, U_\f, M)$  
where $M = Y_\f \Pi_1$ with  
\vspace{-1mm}
\begin{equation} \label{eq:projection-matrix}
    \Pi_1 := H_1^\dag H_1, \qquad H_1 := \col(U_\p, Y_\p, U_\f). 
\vspace{-1mm}
\end{equation} 
Then, the bi-level problem \eqref{eqn:bi-level-SPC}  is equivalent to 
\vspace{-1mm}
\begin{equation}
\label{eqn:DD-SPC}
\begin{aligned}
\min_{\substack{g, \sigma_y \in \Gamma, \\ u \in \mathcal{U}, y \in \mathcal{Y}}}  \;\, & \|u\|_R^2 + \|y\|_Q^2 + \lambda_y \|\sigma_y\|_2^2  \\
\mathrm{subject~to} \;\, & 
{H}^*_\s g
= 
\col(
u_{\ini},
y_{\ini}+\sigma_y,
u,
y
).
\end{aligned}
\vspace{-1mm}
\end{equation}
Note that \eqref{eqn:DD-SPC} differs slightly from the classical SPC formulation, and we call it Equality-form \texttt{SPC}. In particular, the classical SPC does not involve the use of an identified trajectory library $H_\s^*$; instead, it employs an explicit multi-step predictor (also see Remark \ref{rem: diff-bi-form}). We provide further details in Appendix \ref{appendix:equal-spc-ddspc}. 

Following the discussion in Section \ref{subsection:regularization-on-g}, we here fix $H = H_{\mathrm{d}}$ and impose a constraint on $g$ to get an equivalent reformulation for~\eqref{eqn:DD-SPC}. In this case, we use a projection-based norm  $\|(I-\Pi_1)g\|_2$ as in \cite{dorfler2022bridging}, and consider 
\vspace{-1mm}
\begin{subequations}
\label{eqn:reg-proj}
\begin{align}
    \min_{\substack{g,\sigma_y \in \Gamma, \\ u \in \mathcal{U} ,y \in \mathcal{Y}}} \quad & \|y\|_Q^2 + \|u\|_R^2  + \lambda_y \|\sigma_y\|_2^2 \nonumber\\
    \mathrm{subject~to} \quad & H_{\mathrm{d}} g=\col(u_\ini, y_\ini+\sigma_y, u, y ), \label{eqn:reg-proj-1}  \\
    & \|(I-\Pi_1)g\|_2 = 0, \label{eqn:reg-proj-2}
\end{align}
\vspace{-5mm}
\end{subequations} \\
where the matrix $\Pi_1$ is defined in \eqref{eq:projection-matrix}. 

\begin{proposition}
\label{prop:spc-equival}
    Suppose $Q \succ 0, R \succ 0, \lambda_y > 0$ and that $\Gamma, \mathcal{U},$ and $\mathcal{Y}$ are convex sets. Fix any data matrix $H_\D$ in~\eqref{eqn:predictor-data-matrix}. Then, the bi-level problem \eqref{eqn:bi-level-SPC} has the same optimal value as the single-level problems \eqref{eqn:DD-SPC} and \eqref{eqn:reg-proj}. Moreover, all three problems admit the same unique optimal solution $u^*, y^*$ and $\sigma_y^*$. 
\end{proposition}
The proof follows a similar idea in \cite{shang2024convex}, and we postpone it to Appendix \ref{appendix:spc-equival}. The key idea is to show that feasible regions for decision variables $u, y, \sigma_y$ are the same for \eqref{eqn:bi-level-SPC}, \eqref{eqn:DD-SPC}, and \eqref{eqn:reg-proj}. Note that the data matrix $H_\D$ in Proposition \ref{prop:spc-equival} can be arbitrary and may come from a nonlinear system. 

It is now clear that \eqref{eqn:reg-proj} is in the form of \eqref{eqn:qp-c-main} (we give a detailed construction in Appendix \ref{appendix:exact-penalty}). Thus, applying Theorem \ref{them:x-exact-main} leads to the following result. 
\begin{corollary}
\label{them:DeeP-exact}
    Suppose $Q \succ 0, R \succ 0, \lambda_y > 0$ and $\Gamma, \mathcal{U},$ and $\mathcal{Y}$ are convex sets. Consider the regularized problem 
        \begin{equation}
\label{eqn:p-DeePC-p}
\begin{aligned}
\min_{\substack{g, \sigma_y \in \Gamma, \\ u \in \mathcal{U}, y \in \mathcal{Y}}}  \quad &  \|u\|_R^2 + \|y\|_Q^2 + \lambda_y \|\sigma_y\|_2^2 \\
& + \lambda_g \|(I-\Pi_1)g\|_2   \\
\mathrm{subject~to} \quad & H_\D g =\col(u_\ini, y_\ini+\sigma_y, u, y).
\end{aligned}
\end{equation}
For any $\lambda_g \geq 0$, it is a convex relaxation of \eqref{eqn:reg-proj}. If the optimal solution $(\sigma_y^*, u^*, y^*)$ to \eqref{eqn:reg-proj} is an interior point of $\Gamma \times \mathcal{U} \times \mathcal{Y}$, then there exists $\lambda_g^* >0$ such that \eqref{eqn:reg-proj} and \eqref{eqn:p-DeePC-p} are equivalent for all $\lambda_g > \lambda_g^*$. In other words, they have the same optimal cost value and optimal solution. 
\end{corollary}
We call \eqref{eqn:p-DeePC-p} Linearity-based \method{DDPC} (\method{L-DDPC}). Combining Corollary~\ref{them:DeeP-exact} with Proposition~\ref{prop:spc-equival} leads to the result below.
\begin{theorem}
\label{cor:relax-spc}
Consider the bi-level \method{DDPC} \eqref{eqn:bi-level-SPC} and the single-level \method{L-DDPC} \eqref{eqn:p-DeePC-p}. Suppose $Q \succ 0, R \succ 0, \lambda_y > 0$ and $\Gamma, \mathcal{U}, \mathcal{Y}$ are convex sets. Then, for any nonnegative $\lambda_g\geq0$, \eqref{eqn:p-DeePC-p} is a convex relaxation of \eqref{eqn:bi-level-SPC}, that is:
\begin{enumerate}
    \item Problem \eqref{eqn:p-DeePC-p} is convex; 
    \item Any feasible $\sigma_y, u, y$ for \eqref{eqn:bi-level-SPC} is feasible to \eqref{eqn:p-DeePC-p}; 
    \item The optimal value of \eqref{eqn:p-DeePC-p} is smaller than or equal to that of \eqref{eqn:bi-level-SPC}.
\end{enumerate}
\end{theorem}
Our Theorem \ref{cor:relax-spc} is closely related to \cite[Theorem 4.6]{dorfler2022bridging}, but there are two key differences. Conceptually, as illustrated in Example \ref{example:bi-level}, the \method{L-DDPC} \eqref{eqn:p-DeePC-p} is only a convex relaxation of \eqref{eqn:bi-level-SPC}, and not necessarily of the original bi-level problem \eqref{eqn:bi-level}. This subtle distinction was overlooked in \cite[Theorem 4.6]{dorfler2022bridging}. Technically, we establish that \eqref{eqn:p-DeePC-p} serves as a convex relaxation of \eqref{eqn:bi-level-SPC} for any nonnegative parameter $\lambda_g \geq 0$, rather than only for sufficiently small $\lambda_g$ as in \cite[Theorem 4.6]{dorfler2022bridging}.
The threshold $\lambda_g^*$ that makes \eqref{eqn:bi-level-SPC} and \eqref{eqn:p-DeePC-p} equivalent depends on both the Lipschitz constant of the cost function over a specific region and problem data (\emph{e.g.}, $H_\D$) rather than only on the Lipschitz constant as in \cite[Theorem 4.6]{dorfler2022bridging}. Moreover, our proof is elementary, relying on Theorem \ref{them:x-exact-main} and Proposition \ref{prop:spc-equival}, which are inspired by the notion of partial calmness. We avoid utilizing \cite[Proposition 2.4.3]{clarke1990optimization} as in \cite[Theorem 4.6]{dorfler2022bridging}, whose application in the \method{DDPC} setting may require additional adaptations. 

\subsection{Causality-based norm and a Causal SPC}
\label{subsec:causal-DeePC}

We next consider the following bi-level problem
\vspace{-2mm}
\begin{align}
    \min_{\substack{g,\sigma_y \in \Gamma, 
    \\ u \in \mathcal{U} ,y \in \mathcal{Y}}} \quad & \|y\|_Q^2 + \|u\|_R^2  + \lambda_y \|\sigma_y\|_2^2  \nonumber 
    \\
    \mathrm{subject~to} \quad & {H}^* g=\col(u_\ini, y_\ini+\sigma_y, u, y ),  \label{eqn:bi-level-SPC-causal} \\
    &\textrm{where} \;\;  {H}^* \in \argmin_{\tilde{H}, K} \;\; 
 \|\tilde{H} - H_{\mathrm{d}}\|_F^2 \nonumber  \\
   & \qquad \qquad \mathrm{subject~to} \;\;  \eqref{eqn:id-rowsp}, \eqref{eqn:id-causal}, \nonumber
\vspace{-2mm}
\end{align} 
which considers the SPC constraint \eqref{eqn:id-rowsp} and the causality constraint \eqref{eqn:id-causal} for the inner identification. If $H_\D$ has full row rank, the optimal solution $H^*_\sca$ for the inner problem in \eqref{eqn:bi-level-SPC-causal} is unique \cite[Lemma 2]{sader2023causality}, which can be obtained via its optimal solution $K^*_\sca$, derived as
\vspace{-1mm}
\[
K_\sca^* = \begin{bmatrix}
    L_{31} & L_{32}^*
\end{bmatrix} \begin{bmatrix}
    L_{11} & \mathbb{0} \\ L_{21} & L_{22}
\end{bmatrix}^{-1},
\vspace{-1mm}
\]
where $L_{ij}$ comes from the LQ factorization of $H_\D$, \emph{i.e.}, 
\vspace{-1mm}
\begin{equation*}
%\label{eqn:LQ-fact}
H_\D = \begin{bmatrix}
    L_{11} & \mathbb{0} & \mathbb{0} \\
    L_{21} & L_{22} & \mathbb{0} \\
    L_{31} & L_{32} & L_{33} 
\end{bmatrix}
\begin{bmatrix}
    Q_1 \\ Q_2 \\ Q_3
\end{bmatrix}
\vspace{-1mm}
\end{equation*}
and $L_{32}^*$ is the lower-block triangular part of $L_{32} \in \mathbb{R}^{pL \times mL}$. Then, the optimal $H_\sca^*$ for the inner problem in \eqref{eqn:bi-level-SPC-causal} can be written as
\vspace{-1mm}
\begin{equation}
\label{eqn:opt-causal}
H_\sca^* = \begin{bmatrix}
    \col(U_\p,\!  Y_\p) \\ U_\f \\ K^* \! \col(U_\p,\! Y_\p,\! U_\f)
    \end{bmatrix}= \begin{bmatrix}
        L_{11} & \mathbb{0} \\
        L_{21} & L_{22} \\
        L_{31} & L_{32}^*
    \end{bmatrix} 
    \begin{bmatrix}
        Q_1 \\ Q_2
    \end{bmatrix}.
    \vspace{-2mm}
\end{equation}
We can equivalently formulate \eqref{eqn:bi-level-SPC-causal} as
\vspace{-2mm}
\begin{equation}
\label{eqn:c-gamma-DeePC}
\begin{aligned}
\min_{\substack{g, \sigma_y \in \Gamma, \\ u \in \mathcal{U}, y \in \mathcal{Y}}}  \quad & \|u\|_R^2 + \|y\|_Q^2 + \lambda_y \|\sigma_y\|_2^2  \\
\mathrm{subject~to} \quad & 
H_\sca^*g =
\col(u_{\ini}, y_{\ini}+\sigma_y,u,y).  
\end{aligned}
\vspace{-2mm}
\end{equation}

Following the discussion in Section \ref{subsection:regularization-on-g}, we next fix $H$ to a new trajectory library and set a suitable constraint on $g$ to obtain an equivalent reformulation for \eqref{eqn:c-gamma-DeePC}. For this, we denote the difference between $L_{32}$ and its lower-block triangular part $L_{32}^*$ as $L_{32}' := L_{32} - L_{32}^*$ and the new trajectory library $\hat{H}$ is 
\vspace{-2mm}
\begin{equation} \label{eq:hat-H}
\hat{H} = \begin{bmatrix}
        L_{11} & \mathbb{0} & \mathbb{0} & \mathbb{0}\\
        L_{21} & L_{22} & \mathbb{0} & \mathbb{0}\\
        L_{31} & L_{32}^* & L_{33} & L_{32}'
\end{bmatrix}\! 
    \begin{bmatrix}
        Q_1 \\ Q_2 \\ Q_3 \\ Q^*
    \end{bmatrix}
    \vspace{-2mm}
\end{equation}
where $Q^*$ has orthonormal rows and $Q^* Q_i^\tr = \mathbb{0}$ for $i = 1,2,3$. We further consider the norm $\|Q_\cs g\|_2$ where $Q_\cs = \col(Q_3, Q^*)$. Then, we have
\vspace{-2mm}
\begin{subequations}
\label{eqn:reg-causal}
\begin{align}
    \min_{\substack{g,\sigma_y \in \Gamma, \\ u \in \mathcal{U} ,y \in \mathcal{Y}}} \quad & \|y\|_Q^2 + \|u\|_R^2  + \lambda_y \|\sigma_y\|_2^2 \nonumber\\
    \mathrm{subject~to} \quad & \hat{H} g=\col(u_\ini, y_\ini+\sigma_y, u, y ),  \label{eqn:reg-causal-1}\\
    & \|Q_\cs g\|_2 = 0 \label{eqn:reg-causal-2}.
\end{align}
\vspace{-6mm}
\end{subequations} 

Our next result shows that the optimal solutions of \eqref{eqn:bi-level-SPC-causal}, \eqref{eqn:c-gamma-DeePC} and \eqref{eqn:reg-causal} are all the same. 
\begin{proposition}
\label{them:c-spc-equival}
    Suppose $Q \succ 0, R \succ 0, \lambda_y > 0$ and $\Gamma, \mathcal{U}, \mathcal{Y}$ are convex sets. Fix any data matrix $H_\D$ in \eqref{eqn:predictor-data-matrix} with full row rank and sufficiently large column number. The bi-level problem \eqref{eqn:bi-level-SPC-causal} has the same optimal value as the single-level problems \eqref{eqn:c-gamma-DeePC} and \eqref{eqn:reg-causal}. Furthermore, all three problems have the same unique optimal solution $u^*, y^*$, and $\sigma_y^*$.
\end{proposition}

The idea is to show \eqref{eqn:bi-level-SPC-causal}, \eqref{eqn:c-gamma-DeePC} and \eqref{eqn:reg-causal} share the same feasible region; see Appendix \ref{appendix:c-spc-equival}.  It is clear that~\eqref{eqn:reg-causal} is in the form of \eqref{eqn:qp-c-main}. 
Thus, applying Theorem \ref{them:x-exact-main} leads to the following result.  
\begin{corollary}
\label{them:c-DeeP-exact}
    Suppose $Q \succ 0, R \succ 0, \lambda_y > 0$ and $\Gamma, \mathcal{U}$, and $\mathcal{Y}$ are convex sets. Consider the regularized problem
\vspace{-2mm}
\begin{equation}
\label{eqn:p-DeePC-causal}
\begin{aligned}
\min_{\substack{g, \sigma_y \in \Gamma, \\ u \in \mathcal{U}, y \in \mathcal{Y}}}  \quad &  \|u\|_R^2 + \|y\|_Q^2 + \lambda_y \|\sigma_y\|_2^2 \\
& + \lambda_g \|Q_\cs g\|_2   \\
\mathrm{subject~to} \quad & \hat{H}g =\col(u_\ini, y_\ini+\sigma_y, u, y),
\end{aligned}
\vspace{-2mm}
\end{equation}
with $\hat{H}$ given in \eqref{eq:hat-H}. For any $\lambda_g \geq 0$, it is a convex relaxation of \eqref{eqn:reg-causal}. Further, if the optimal solution $(\sigma_y^*, u^*, y^*)$ for \eqref{eqn:reg-causal} is an interior point of $\Gamma \times \mathcal{U} \times \mathcal{Y}$, then there exists $\lambda_g^* >0$, such that, problem \eqref{eqn:p-DeePC-causal} has the same optimal value and optimal solutions as \eqref{eqn:reg-causal} for all $\lambda_{g} > \lambda_g^*$.
\end{corollary}
We call \eqref{eqn:p-DeePC-causal} as Causality-based \method{DDPC} (\method{C-DDPC}). The new regularizer $\|Q_\cs g\|_2$ penalizes both the violation of the row space constraint (\emph{i.e.}, $Q_3g$) and the usage of noncasual information (\emph{i.e.}, $Q^*g$). 

Combining Proposition \ref{them:c-spc-equival} and Corollary \ref{them:c-DeeP-exact}, we confirm that \eqref{eqn:p-DeePC-causal} is a relaxation of \eqref{eqn:bi-level-SPC-causal} under mild conditions. 
\begin{theorem}
\label{cor:relax-causal}
Consider the bi-level \method{DDPC}  \eqref{eqn:bi-level-SPC-causal} and the single-level \method{C-DDPC}  \eqref{eqn:p-DeePC-causal}. Suppose $Q \succ 0, R \succ 0, \lambda_y > 0$ and $H_\D$ has full row rank with sufficiently large column number. Let $\Gamma, \mathcal{U}$ and $\mathcal{Y}$ be convex sets. For any nonnegative $\lambda_g  \ge 0$, \eqref{eqn:p-DeePC-causal} is a convex relaxation of~\eqref{eqn:bi-level-SPC-causal}, that is:
\begin{enumerate}
    \item \eqref{eqn:p-DeePC-causal} is a convex optimization problem; 
    \item any feasible $\sigma_y, u, y$ for \eqref{eqn:bi-level-SPC-causal} is feasible to \eqref{eqn:p-DeePC-causal}; 
    \item the optimal value of \eqref{eqn:p-DeePC-causal} is smaller than or equal to that of \eqref{eqn:bi-level-SPC-causal}.
\end{enumerate}
\end{theorem}

Theorems \ref{cor:relax-spc} and \ref{cor:relax-causal} show the direct \method{L-DDPC} and \method{C-DDPC} are convex relaxations of the corresponding bi-level \method{DDPC} formulations. Since hard identification constraints are replaced by regularization terms, these direct approaches admit flexibility in terms of soft identification. Corollaries \ref{them:DeeP-exact} and \ref{them:c-DeeP-exact} developed from Theorem~\ref{them:x-exact-main} provide theoretical conditions that characterize when the relaxed problems become equivalent to the original formulations.

\begin{remark}[Comparison with \cite{sader2023causality}]
    To our knowledge, the single-level \method{C-DDPC} in \eqref{eqn:p-DeePC-causal} is new, and its relationship with the bi-level  \method{DDPC} \eqref{eqn:bi-level-SPC-causal} has not been studied previously. The closest related work is \cite{sader2023causality}, which introduces a causality-based regularizer within the $\gamma$-\method{DDPC} framework. The decision variable is reformulated and the regularization is imposed on $\gamma$, thereby losing the explicit trajectory library interpretation. Instead, our method~\eqref{eqn:p-DeePC-causal} directly regularizes $g$, which retains the structure of the trajectory library. In particular, \method{C-DDPC} can then be naturally combined with the $\|g\|_1$ term, which implicitly performs trajectory selection, as we will discuss in Section \ref{subsection:low-rank}.
\end{remark}

\vspace{2mm}
\subsection{Low-rank approximation and $l_1$ norm} \label{subsection:low-rank}
We finally consider the following bi-level problem
\vspace{-2mm}
\begin{align}
    \min_{\substack{g,\sigma_y \in \Gamma, 
    \\ u \in \mathcal{U} ,y \in \mathcal{Y}}} \quad & \|y\|_Q^2 + \|u\|_R^2  + \lambda_y \|\sigma_y\|_2^2  \nonumber 
    \\
    \mathrm{subject~to} \quad & {H}^* g=\col(u_\ini, y_\ini+\sigma_y, u, y ),  \label{eqn:bi-level-lr} \\
    &\textrm{where} \;\;  {H}^* \in \argmin_{\tilde{H}, K} \;\; 
 \|\tilde{H} - H_{\mathrm{d}}\|_F^2 \nonumber  \\
   & \qquad \qquad \mathrm{subject~to} \;\;  \eqref{eqn:id-rank}. \nonumber
\end{align}
\vspace{-6mm}\\
This bi-level problem aims to approximate the data~$H_\D$ with a low-rank matrix in~\eqref{eqn:id-rank}. 
It is clear that~the~unique optimal solution $H_{\lr}^*$ of the inner problem in~\eqref{eqn:bi-level-lr}~is 
\vspace{-2mm}
$$
H_\lr^* = \sum_{i = 1}^{mL+n} \bar{\sigma}_i \bar{u}_i \bar{v}_i^\tr
\vspace{-2mm}
$$  
where $\bar{\sigma}_i, \bar{u}_i$ and $\bar{v}_i$ are obtained via the standard SVD decomposition $H_\D =  \sum_{i = 1}^{(m+p)L} \bar{\sigma}_i \bar{u}_i \bar{v}_i^\tr$. Problem  \eqref{eqn:bi-level-lr} can be equivalently formulated as 
\vspace{-2mm}
\begin{equation}
\label{eqn:lr}
\begin{aligned}
\min_{g, \sigma_y \in \Gamma, u \in \mathcal{U}, y \in \mathcal{Y}}  \;\, & \|y\|_Q^2+\|u\|_R^2 + \lambda_y \|\sigma_y\|_2^2  \\
\mathrm{subject~to} \;\, & 
{H}_{\lr}^* g
= 
\col(
u_{\ini},
y_{\ini}+\sigma_y,
u,
y
).
\end{aligned}
\vspace{-1mm}
\end{equation}

We now fix $H= H_\D$ and add a $l_1$-norm constraint  of $g$, leading to a single-level problem 
\vspace{-1mm}
\begin{equation}
\label{eqn:reg-l1}
\begin{aligned}
    \min_{g,\sigma_y \in \Gamma, u \in \mathcal{U} ,y \in \mathcal{Y}} \quad & \|y\|_Q^2 + \|u\|_R^2  + \lambda_y \|\sigma_y\|_2^2\\
    \mathrm{subject~to} \quad & H_\D g=\col(u_\ini, y_\ini+\sigma_y, u, y ),  \\
    & \|g\|_1 \le \alpha_\cs.
\end{aligned}
\end{equation}

We have the following result. 
\begin{proposition}
\label{prop:l1-equival}
    Fix any data matrix $H_\D$, and suppose $\alpha_\cs$ is sufficiently large. The optimal value of \eqref{eqn:reg-l1} is smaller than or equal to that of \eqref{eqn:lr}.
\end{proposition}
\vspace{-2mm} 
\begin{proof}
Let $(g^*,\sigma_y^*,u^*,y^*)$ be an optimal solution of \eqref{eqn:lr}, and define 
$
b^* := %\col(u_\ini,y_\ini+\sigma_y^*,u^*,y^*)
 H_\lr^* g^*.
$ 
Since $\mathrm{Im}(H_\lr^*)\subseteq \mathrm{Im}(H_\D)$, there exists some $\hat g$ such that
$
H_\D \hat g = b^*.
$ 
Hence $(\hat g,\sigma_y^*,u^*,y^*)$ is feasible for \eqref{eqn:reg-l1}. Since the objective does not depend on $g$, it has the same objective value. Therefore, if $\alpha_\cs \ge \|\hat g\|_1$, then $(\hat g,\sigma_y^*,u^*,y^*)$ is feasible for \eqref{eqn:reg-l1}, and we thus complete the proof. 
\end{proof}

The constraint $\|g\|_1 \le \alpha_\cs$ can be moved to the objective function without changing the optimal solution of \eqref{eqn:reg-l1} (but not the optimal value). We have the following result.

\begin{proposition}
\label{them:l1-exact}
    Let $Q \succ 0, R \succ 0, \lambda_y > 0$, and $\Gamma \times \mathcal{U} \times \mathcal{Y}$ being convex. Suppose problem \eqref{eqn:reg-l1} satisfies Slater's condition. Consider the regularized problem
    \vspace{-2mm}
       \begin{equation}
\label{eqn:l1-exact}
\begin{aligned}
    \min_{g,\sigma_y \in \Gamma, u \in \mathcal{U} ,y \in \mathcal{Y}} \quad & \|y\|_Q^2 + \|u\|_R^2  + \lambda_y \|\sigma_y\|_2^2 \\
    & + \lambda_g \|g\|_1 \\
    \mathrm{subject~to} \quad & H_\D g=\col(u_\ini, y_\ini+\sigma_y, u, y).
\end{aligned}
\vspace{-2mm}
\end{equation}
There exists a penalty parameter $\lambda_g^* \ge 0$, such that any optimal solution $(\sigma_y^*, u^*, y^*, g^*)$ of \eqref{eqn:reg-l1} is also an optimal solution of \eqref{eqn:l1-exact} for $\lambda_g = \lambda_g^*$.

\vspace{-2mm}
\end{proposition}

    \begin{proof}
To simplify notation, we let
$
x:=\col(\sigma_y,u,y), 
\mathcal{C}:=\Gamma\times\mathcal U\times\mathcal Y,
$
and define
$
f(x):=\|y\|_Q^2+\|u\|_R^2+\lambda_y\|\sigma_y\|_2^2.
$ 
The equality constraint in \eqref{eqn:reg-l1} can be written as
$
H_\D g-Ax=b
$ 
for some matrix $A$ and vector $b$. 

Then we rewrite \eqref{eqn:reg-l1} as 
\vspace{-2mm}
\[ 
\min_{x\in \mathcal{C},g}\ f(x)
\quad \text{subject to}\quad
H_\D g-Ax=b,\quad \|g\|_1\le \alpha_\cs. \vspace{-2mm}
\]
Strong duality holds for \eqref{eqn:reg-l1} by Slater's condition. Let $(\lambda_g^*,\mu^*)$ be a dual optimal solution, with $\lambda_g^*\ge 0$.
Take any primal optimal solution $(x^*,g^*)$. By strong duality of \eqref{eqn:reg-l1}, we know that $(x^*,g^*)$ minimizes the Lagrangian \vspace{-2mm}
\[ 
f(x)+\lambda_g^*(\|g\|_1-\alpha_\cs)+(\mu^*)^\tr(H_\D g-Ax-b). \vspace{-2mm}
\]
Dropping the constant term $-\lambda_g^*\alpha_\cs$, we conclude that $(x^*,g^*)$ also minimizes
$
f(x)+\lambda_g^*\|g\|_1+(\mu^*)^\tr(H_\D g-Ax-b).
$ 
Since $(x^*,g^*)$ is feasible for \eqref{eqn:l1-exact}, it follows that $(x^*,g^*)$ is optimal for \eqref{eqn:l1-exact}. As $(x^*,g^*)$ is an arbitrary primal optimal solution of \eqref{eqn:reg-l1}, the result follows. 
\end{proof}

We note that the converse implication generally does not hold in Proposition \ref{them:l1-exact}: an optimal solution of 
\eqref{eqn:l1-exact} need not be feasible, and hence need not be optimal, for \eqref{eqn:reg-l1}. Propositions \ref{prop:l1-equival} and \ref{them:l1-exact}  imply that \eqref{eqn:l1-exact} is a relaxation of~\eqref{eqn:bi-level-lr} with respect to the function $\|y\|_Q^2+\|u\|_R^2 + \lambda_y \|\sigma_y\|_2^2$ under mild conditions. 
\begin{theorem}
\label{cor:relax-lr}

Consider the bi-level problem \eqref{eqn:bi-level-lr} and the single-level problem \eqref{eqn:l1-exact}. Suppose $Q \succ 0$, $R \succ 0$, $\lambda_y > 0$, and $\Gamma$, $\mathcal U$, and $\mathcal Y$ are convex sets. Let $\lambda_g^* \ge 0$ be the penalty parameter given by Proposition \ref{them:l1-exact}. Then, for any $0 \le \lambda_g \le \lambda_g^*$, problem \eqref{eqn:l1-exact} is a convex relaxation of \eqref{eqn:bi-level-lr} with respect to the function
$
f(\sigma_y,u,y):=\|y\|_Q^2+\|u\|_R^2+\lambda_y\|\sigma_y\|_2^2,
$
in the following sense: 
\begin{enumerate}
    \item \eqref{eqn:l1-exact} is a convex optimization problem;
    \item any feasible $(\sigma_y, u, y)$ for \eqref{eqn:bi-level-lr} is feasible to  \eqref{eqn:l1-exact};
    \item for any optimal solution $(g^* ,\sigma_y^* ,u^* ,y^* )$ of \eqref{eqn:l1-exact},
    $
    f(\sigma_y^* ,u^* ,y^* ) 
    $ is smaller than the optimal value~of~\eqref{eqn:bi-level-lr}.
\end{enumerate}
\end{theorem}
     A similar statement as Theorem \ref{cor:relax-lr} appears in \cite[Theorem 4.8]{dorfler2022bridging}. However, we emphasize that the relaxation is only true with respect to the function $\|y\|_Q^2 + \|u\|_R^2 + \lambda_y \|\sigma_y\|_2^2$ under an upper bound for the regularizing parameter $\lambda_g$, which is treated differently in \cite{dorfler2022bridging}. The reason is that while the optimal solution of \eqref{eqn:l1-exact} is the same as~\eqref{eqn:reg-l1}, their objective functions are different, which generally leads to different optimal values. 

%% file: 5-Hybrid-preprocessing.tex
\section{Hybrid preprocessing and regularization of data for predictive control}
\label{sec:SVD-Iter}
The \method{DDPC} variants in Section \ref{sec:methods} have modified the inner identification problem \eqref{eqn:identification}, instead of directly solving it. By simplifying the inner problem, we obtain single-level formulations that serve as convex relaxations for the modified bi-level \method{DDPC} \eqref{eqn:bi-level}; see Theorems \ref{cor:relax-spc}, \ref{cor:relax-causal}, and \ref{cor:relax-lr}.

In this section, we propose an iterative algorithm that approximates the solution of the original inner problem \eqref{eqn:identification} and leverages this approximate solution for the outer predictive control. In particular, given the offline input and output data \eqref{eqn:Hankel_partition}, we consider that the input data $H_u := \col(U_\p, U_\f)$ is accurate and contains no noise. Therefore, we focus on denoising the output trajectory, and problem \eqref{eqn:identification} becomes  \vspace{-3mm}
\begin{subequations} \label{eq:SVD-Dom}
 {  \begin{align} 
\min_{\tilde{H}_y, K} \quad & \|H_y - \tilde{H}_y\|_F \nonumber\\
\mathrm{subject~to} \quad & \textrm{rank}(\tilde{H}) = mL+n,  \label{eq:SVD-Dom-a}\\
& \tilde{Y}_\f = K \ \col(U_\p, \tilde{Y}_\p, U_\f), \label{eq:SVD-Dom-c}\\
& K = \begin{bmatrix}
    K_p & K_f
\end{bmatrix}, \ K_f \in \mathcal{L} \label{eq:SVD-Dom-d}, \\
& \tilde{H}_y \in \mathcal{H} \label{eq:SVD-Dom-b}, 
\end{align}}
\end{subequations} \vspace{-5mm} \\
where $\tilde{H}_y, \tilde{H}$ denote $\col(\tilde{Y}_\p, \tilde{Y}_\f)$, and $\col(U_\p, \tilde{Y}_\p, U_\f, \tilde{Y}_\f)$, respectively. Compared with \eqref{eqn:identification}, the decision variables in \eqref{eq:SVD-Dom} become $\tilde{H}_y$ and $K$, as we only denoise the output trajectory. Still, problem \eqref{eq:SVD-Dom} is difficult to solve due to the interplay between \eqref{eq:SVD-Dom-a} to \eqref{eq:SVD-Dom-d}. Without \eqref{eq:SVD-Dom-c}, \eqref{eq:SVD-Dom-d}, it is a structured low-rank approximation (SLRA) problem~\cite{markovsky2008structured,yin2021low}. We here adapt an iterative SLRA algorithm in~\cite{yin2021low} to get the approximation solution to \eqref{eq:SVD-Dom}. 

\vspace{1.5mm}
\subsection{Sequential optimization}
Our key idea is to adopt a sequential optimization strategy that addresses the constraints in \eqref{eq:SVD-Dom} one at a time. The detailed procedure is described below.

\textbf{Step 1: Low-rank approximation.} Since the input data $u_\D$ has no noise and satisfies the persistent excitation, we have $\textrm{rank}(H_u) = mL$. In contrast, the output data $y_\D$ may contain ``variance'' noise and ``bias'' error. Thus, the rank of the raw data $H$ (\emph{i.e.}, $\col(H_u, H_y)$) is larger than $mL+n$. We first consider the constraint \eqref{eq:SVD-Dom-a} that approximates $H$ with a low-rank matrix, which is  
\vspace{-1mm}
\begin{equation}
\label{eqn:low-rank-approx}
    \begin{aligned} 
\min_{\tilde{H}_y} \quad & \|H_y - \tilde{H}_y\|_F \\
\mathrm{subject~to} \quad & \textrm{rank}(\tilde{H}) = mL+n.
\end{aligned}
\vspace{-1mm}
\end{equation}
This problem is slightly different from the standard low-rank approximation, as we do not change the input data $U_\p$ and $U_\f$.  Still, we can get an analytical solution of~\eqref{eqn:low-rank-approx} via singular value decomposition.
\begin{proposition}
\label{prop:opt-low-rank}
    Consider problem \eqref{eqn:low-rank-approx}, where the input data $u_\D$ has no noise and satisfies the persistent excitation. Let $\Pi_2 = H_u^\dag H_u$ be the orthogonal projector onto the row space of $H_u$, and denote SVD for the component of $H_y$ in the null space of $H_u$ as $H_y (I - \Pi_2) = \sum_{i = 1}^{pL} \bar{\sigma}_i \bar{u}_i \bar{v}_i^\tr$.
    Then, the optimal solution $H_{y_1}$ to \eqref{eqn:low-rank-approx} is given by 
    \vspace{-2mm}
    \begin{equation}
    \label{eqn:opt-low-rank}
    H_{y_1} := H_y \Pi_2 + \sum_{i=1}^{n} \bar{\sigma}_i \bar{u}_i \bar{v}_i^\tr.
    \vspace{-2mm}
    \end{equation}
\end{proposition}
Since we have $\mathrm{rank}(H_u) = mL$, the key insight for Proposition \ref{prop:opt-low-rank} is to ensure the part of $H_y$ in the null space of $H_u$ has rank $n$. Specifically, we first divide $H_y$ into two parts that are $H_y \Pi_2$ in the row space of $H_u$ and $H_y(I-\Pi_2)$ in the null space of $H_u$. We then perform an SVD of $H_y(I-\Pi_2)$ to estimate its rank-$n$ approximation, and finally combine it with the component of $H_y$ in the row space of $H_u$ as in \eqref{eqn:opt-low-rank}. A detailed proof is provided in Appendix \ref{appendix:opt-low-rank}.

For notational simplicity, we denote the mapping from the data $H_y$ to the optimal solution $H_{y_1}$ of \eqref{eqn:low-rank-approx} as $\Pi_\Lo$.

\textbf{Step 2: Hankel structure:} We then project  $H_{y_1}$ to a Hankel matrix set via averaging skew-diagonal elements~\cite{yin2021low} and represent the projector and the resulting Hankel matrix as $\Pi_\h$ and $H_{y_2}$.

\textbf{Step 3: Causality guarantee:} We finally use the Hankel approximation $H_{y_2}$ to form the problem 
\vspace{-2mm}
\begin{equation}
\label{eqn:row-causal}
    \begin{aligned} 
\min_{\tilde{H}_y, K} \quad & \|H_{y_2} - \tilde{H}_y\|_F \\
\mathrm{subject~to} \quad & \tilde{Y}_\f = K \ \col(U_\p, Y_{\p_2}, U_\f), \\
& K = \begin{bmatrix}
    K_p & K_f
\end{bmatrix}, \ K_f \in \mathcal{L},
\end{aligned}
\vspace{-2mm}
\end{equation}
which tackles constraints \eqref{eq:SVD-Dom-c}, \eqref{eq:SVD-Dom-d} and $\col(Y_{\p_2}, Y_{\f_2})\! :=\! H_{y_2}$. Problem~\eqref{eqn:row-causal} also has an analytical solution as shown in \eqref{eqn:opt-causal}, and we denote the mapping from $H_{y_2}$ to the optimal solution $H_{y_3}$ as $\Pi_\C$. We note that the analytical solution of \eqref{eqn:row-causal} is derived based on the fact that $\col(U_\p, Y_{\p_2}, U_\f, Y_{\f_2})$ has full row rank, which is generally true after the Hankel-structure approximation. 

We repeat Steps $1\!-\!3$ iteratively as listed in Algorithm~\ref{alg:iter-SLRA}. The resulting matrix $H_y^*$ from Algorithm~\ref{alg:iter-SLRA} is partitioned as $\col(Y_\p^*, Y_\f^*)$, and we form a new Hankel matrix $H_{\Op}^* = \col(U_\p, Y_\p^*, U_\f, Y_\f^*)$. We call $H_{\Op}^*$ an \textit{approximated optimal linear representation},~as~it directly tackles the inner identification problem without relaxation.
\begin{algorithm}[t]
\caption{Iterative Construction for Approximately Optimal Trajectory Library}
    \label{alg:iter-SLRA}
  \begin{algorithmic}[1]
    \Require 
    $U_\p$,  $U_\f$, $Y_\p$, $Y_\f$, $\epsilon$
    \State $H_{y} \gets \col(Y_\p, Y_\f)$, $H_{y_3} \gets H_{y}$;
    \Repeat
    \State $H_{y_1} \gets \Pi_\Lo(H_{y_3})$ \quad (\texttt{Low-rank approx});
    \State $H_{y_2} \!\gets\!\! \Pi_\h (H_{y_1})$ \quad (\texttt{Hankel~proj});
    \State $H_{y_3} \!\gets\!\! \Pi_\C (H_{y_2})$ \quad (\texttt{Causality~proj});
    \Until{$\|H_{y_2} -H_{y_3}\|_F \le \epsilon \|H_{y_3}\|_F$}
    \Ensure $H_y^* = H_{y_3}$
\end{algorithmic}
\end{algorithm}

\vspace{1.5mm}
\subsection{Data-driven predictive control with approximated linear data-driven representation} 

We use the approximately optimal linear model $H_{\Op}^*$ from  Algorithm \ref{alg:iter-SLRA} as the predictor in data-driven predictive control, which leads to  
\begin{equation}
\label{eqn:DeePC-SVD-Iter}
\begin{aligned}
\min_{\substack{g, \sigma_y \in \Gamma, \\ u \in \mathcal{U}, y \in \mathcal{Y}}}  \quad & \|u\|_R^2 + \|y\|_Q^2  + \lambda_y \|\sigma_y\|_2^2 + \lambda_1 \|g\|_1\\
\mathrm{subject~to} \quad & 
H_{\Op}^* g \!=\! 
\col( u_{\ini}, y_{\ini}+\sigma_y, u, y),\\
\end{aligned}
\end{equation}
where the $l_1$-norm regularizer is used to implicitly select the trajectory in $H_{\Op}^*$. We call this formulation \eqref{eqn:DeePC-SVD-Iter} as Approximation-based \method{DDPC} (\method{A-DDPC}). We note that $H_{\Op}^*$ only needs to be computed once via the offline collected data, and the initial trajectories $u_\ini, y_\ini$ are updated online at each step.

The SLRA \cite{markovsky2008structured,yin2021low}, which utilizes SVD and projection onto the set of Hankel matrices iteratively, has been well-studied. Sufficient conditions for its convergence guarantee are discussed in~\cite{andersson2013alternating}. However, these results cannot be directly applied to Algorithm~\ref{alg:iter-SLRA} as we require extra separation steps to divide $H_y$ into the row and null spaces of $H_u$ and the causality projection. In numerical implementation, Algorithm \ref{alg:iter-SLRA} is applied one time offline and was observed to converge in all our experiments with various pre-collected trajectories. Establishing a formal theoretical proof and analyzing its convergence rate remain open, which we leave for future work.

\begin{remark}[LTI systems and beyond]
\label{rem:LTI-character}
  Constraints \eqref{eq:SVD-Dom-a}-\eqref{eq:SVD-Dom-d} are necessary for the Hankel matrix $H$ to come from an LTI system, but they are not sufficient. As discussed in \cite{dorfler2022bridging}, the parameter $K$ needs to satisfy additional structure requirements to ensure the existence of the corresponding matrix parameters $A, B, C, D$ in \eqref{eqn:LTI-sys}. For nonlinear systems, making the data matrix $H$ more structured, as in the LTI setting, may introduce a larger bias error. Our proposed \method{A-DDPC} in \eqref{eqn:DeePC-SVD-Iter} preserves the linear, low-rank structure, and causality without requiring the predictor $H_{\Op}^*$ to exactly correspond to an LTI system. This hybrid preprocessing and regularization imposes certain system structure but also gives flexibility to allow \eqref{eqn:DeePC-SVD-Iter} to select model complexity implicitly, thus improving empirical control performance. 
\end{remark}

\begin{remark}[Complexity of \method{DDPC} variants]
\method{L-DDPC} \eqref{eqn:p-DeePC-p}, \method{C-DDPC} \eqref{eqn:p-DeePC-causal}, and \method{A-DDPC} \eqref{eqn:DeePC-SVD-Iter} have the same~optimization form and thus share  similar online computational complexity. Their main difference lies in the construction of trajectory libraries and the choice of regularizers. Using the equality constraints, we can eliminate $u, y, \sigma_y$, so that the only decision variable is $g$, whose dimension is $T-T_{\ini}- N+1$. For standard LTI systems, the length of the pre-collected trajectory to ensure the persistent exciting of the input is $(m+1)(T_\ini +N+n)-1$, and we can choose $T_\ini = n$. In this case, $g$ has dimension $m(N+2n)+n$, which is $2mn +n$ larger than that in classical MPC and is thus of comparable order. For systems beyond LTI settings, a longer pre-collected trajectory is generally needed to capture richer system behavior. In our experiments, the computation time for systems beyond LTI settings is on the order of $10^{-2} \, \mathrm{s}$, which remains suitable for real-time implementation. We note that scalability can become a limitation when applying \method{DDPC} or its variants for high-dimensional systems. They generally exhibit richer system behavior, which requires trajectory libraries to have larger depth and width. That increases both the number of decision variables and constraints of the problem, leading to increased complexity. In this case, we may need to apply truncated SVD to reduce the number of columns for the trajectory library \cite{zhang2023dimension} or select the most relevant trajectories for the current state to formulate the trajectory library \cite{naf2025choose}.
\end{remark}

%% file: 6-Numerical-simulation.tex
\section{Numerical experiments}
\label{sec:results}
We compare the numerical performance\footnote{Our code is available at \url{https://github.com/soc-ucsd/Convex-Approximation-for-DeePC}.} for the \method{DDPC} variants in Sections \ref{sec:methods} and \ref{sec:SVD-Iter}, including: 1) \method{L-DDPC} \eqref{eqn:p-DeePC-p}, 2) \method{C-DDPC} \eqref{eqn:p-DeePC-causal} and 3) \method{A-DDPC}  \eqref{eqn:DeePC-SVD-Iter}. We also include results from the Equality-form \method{SPC} (\method{SPC}) in \eqref{eqn:DD-SPC} and the indirect \method{DDPC} using system identification for comparison. In particular, we used the N4SID \cite{van1994n4sid} for the system ID. We do not consider \eqref{eqn:l1-exact} individually, as an $l_1$-norm regularizer is added for all \method{DDPC} variants. 

We perform two sets of experiments in this section: 1) an LTI system with noisy measurement \cite{fiedler2021relationship} (``variance" error) and 2) another nonlinear system with noise-free measurement \cite{dorfler2022bridging} (``bias" error). An additional example of controlling a two-wheeled robot using \method{DDPC} variants is provided in Appendix \ref{appendix:two-wheeled-robot}. We see that the indirect method performs better on the non-deterministic LTI system under the ``variance" error, while the direct \method{DDPC} variants have superior performance on the nonlinear system under the ``bias" error. This finding is consistent with \cite{dorfler2022bridging}. Furthermore, our results show the superior performance of \method{A-DDPC} in both systems. For hyperparameter selection, we perform a grid search for $\lambda_g$ and $\lambda_1$ over a broad range for the non-deterministic LTI system: we use the realized cost as the evaluation metric and choose a region that yields relatively good performance. The selected parameters are further tuned in the nonlinear example to enhance numerical stability.

\subsection{Non-deterministic LTI system} 
\label{subsec:res-LTI}
\noindent \textbf{Experiment setup.}  We first consider an LTI system from~\cite{fiedler2021relationship}, which is a triple-mass-spring system with $n=8$ states, $m=2$ inputs (two stepper motors), and $p=3$ outputs (disc angles). We consider data collection with additive Gaussian measurement noises $\omega \sim \mathcal{N}(0, \sigma I)$. We utilize the data sets with various sizes and noise levels that are $T=400, 600, 800$ and $\sigma$ from $0.02$ to $0.1$ in increments of $0.02$. The prediction horizon and the initial sequence are chosen as $N = 40$ and $T_{\ini} = 4$, respectively. We choose $Q = I$, $R = 0.1I, \mathcal{U} = [-0.7, 0.7]$ and $\Gamma, \mathcal{Y} = \mathbb{R}^3$ unless otherwise specified.

\textbf{Numerical results.} 
We here compare the realized control cost for different \method{DDPC} variants with various sizes of data sets and noise levels. Since the performance of \method{DDPC} variants depends on the pre-collected data set, all realized control costs for different \method{DDPC} variants are averaged over 100 pre-collected data sets.

\begin{figure}[t]
\setlength{\abovedisplayskip}{0pt}
\centering
\subfigure[$T = 400$]{\includegraphics[width=0.46\textwidth]{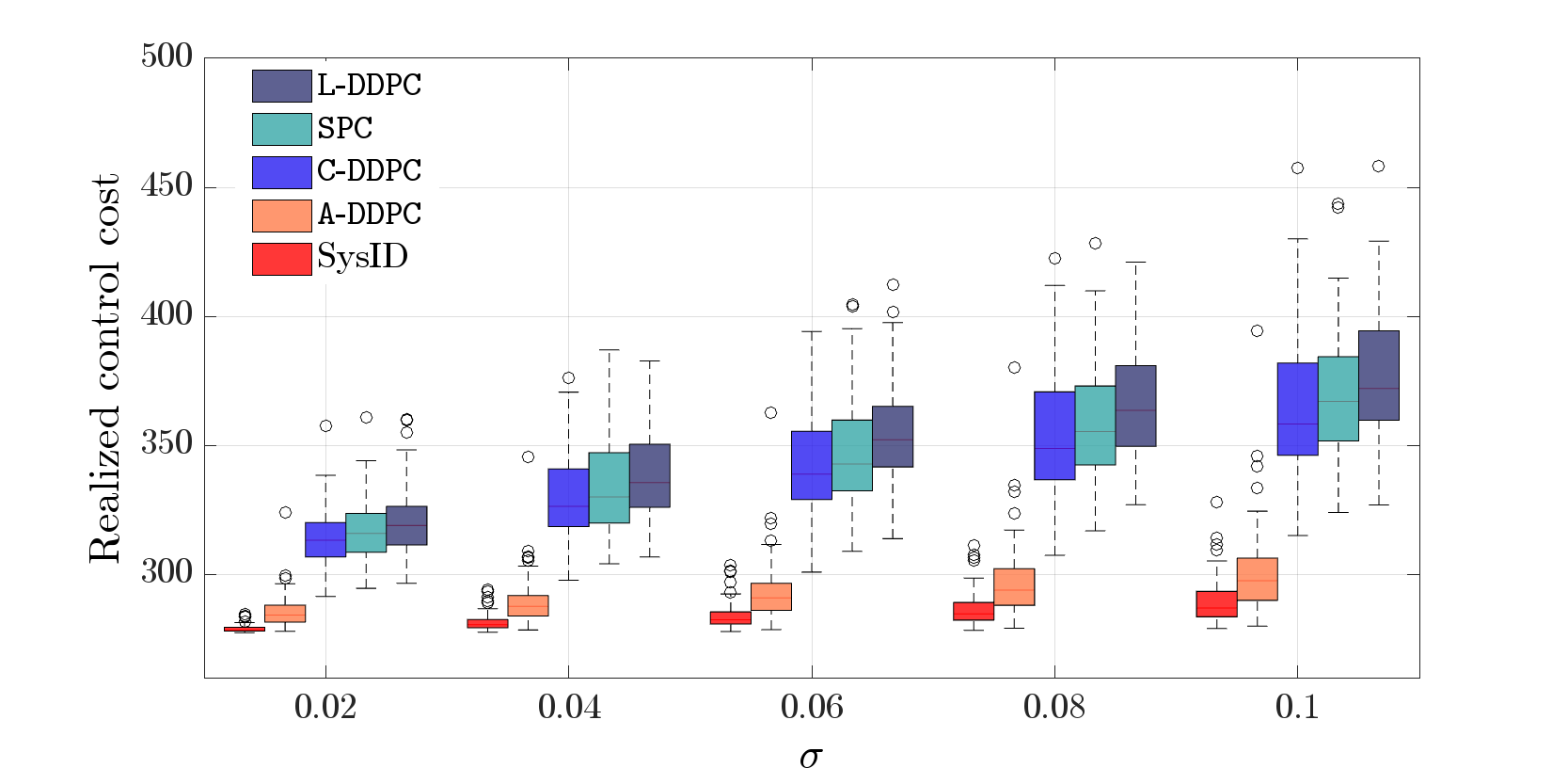} \label{fig:traj-hybrid}} \\
\vspace{-4mm}
\subfigure[$T = 600$]{\includegraphics[width=0.46\textwidth]{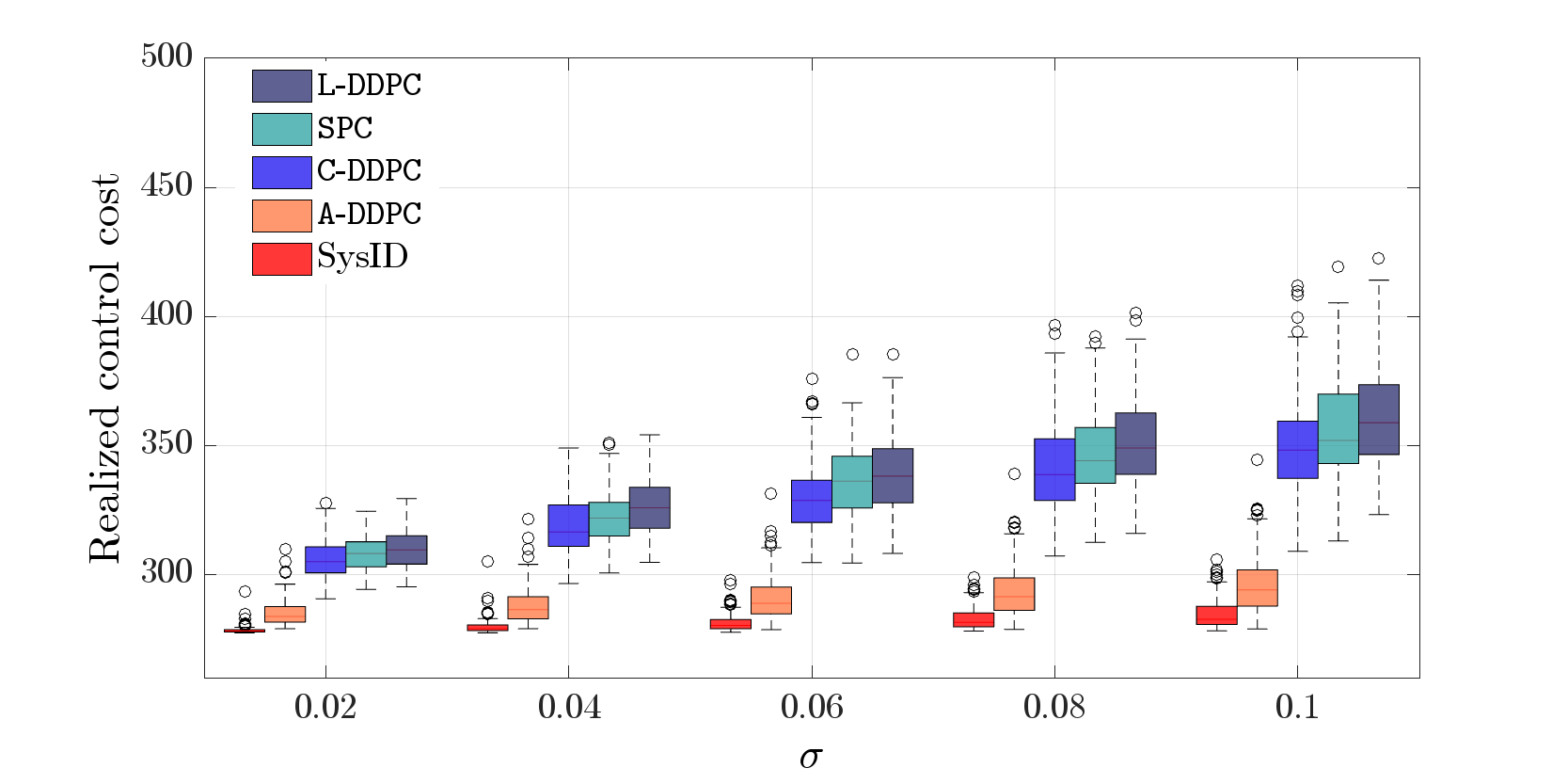}} 
\\
\vspace{-4mm}
\subfigure[$T = 800$]{\includegraphics[width=0.46\textwidth]{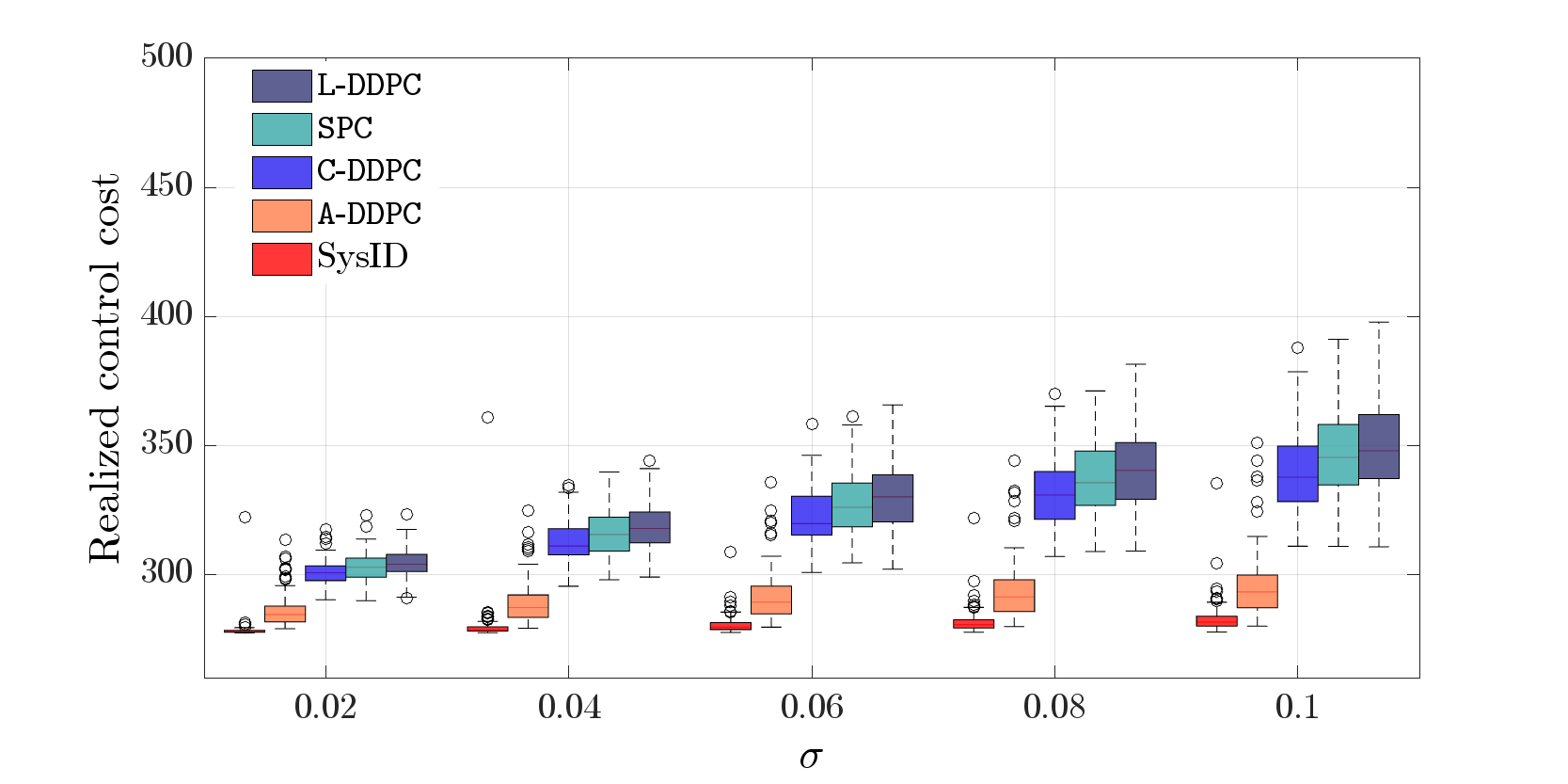}}\\
\vspace{-3mm}
\caption{Comparison of realized control cost of different variants of \method{DDPC} with varying measurement noises. (a), (b) and (c) illustrate the control performance of data-driven controllers with different pre-collected trajectory lengths. Across all experiments, the proposed \method{A-DDPC} \eqref{eqn:DeePC-SVD-Iter} is comparable to the system ID approach, which provides the best performance among all \method{DDPC} methods. 
}
\label{fig:traj-nond-LTI}

\vspace{-.5mm}
\end{figure}

The numerical results are shown in Fig.~\ref{fig:traj-nond-LTI}. As expected, the control performance of all data-driven controllers deteriorates with increasing noise levels. Among the \method{DDPC}-based methods, the proposed \method{A-DDPC} \eqref{eqn:DeePC-SVD-Iter} achieves the best realized control costs and exhibits the strongest noise robustness. This improvement may come from its iterative low-rank approximation, causality projection, and Hankel structure projection (see Algorithm~\ref{alg:iter-SLRA}), which together act as an effective noise filter. This reduces the influence of variance noises in LTI systems.

\begin{table}[h]
\caption{Realized Control Cost and with different pre-collected trajectory lengths with $\sigma = 0.1$; The ground truth value is $277.2487$, and the increase ratio is shown in parentheses.}
\centering
\begin{threeparttable}
\fontsize{8pt}{7pt}\selectfont
\begin{tabular}{c@{\hskip 5pt}c@{\hskip 5pt}c@{\hskip 5pt}c}
\toprule
& $T=400$ & $T = 600$ & $T=800$ \\
\midrule
\method{L-DDPC} & 377.51 ($36.16\%$)  & 361.36 ($30.34\%$) & 349.10 ($25.92\%$) \\
\method{SPC} & 368.42 ($32.88\%$) & 356.60 ($28.62\%$) & 346.56 ($25.00\%$) \\
\method{C-DDPC} & 364.51 ($ 31.48\%$) & 351.43 ($26.76\%$) & 340.39 ($22.77\%$)\\
\method{A-DDPC} & 300.77 ($\mathbf{8.48\%}$) & 296.59 ($\mathbf{6.98\%}$) & 296.31 ($\mathbf{6.88\%}$) \\
SysID & 289.25 ($4.33\%$)  & 284.64 ($2.67\%$) & 283.02 ($2.08\%$) \\
\bottomrule
\end{tabular}
\end{threeparttable}
\label{table:cost-time}
\end{table}

Table~\ref{table:cost-time} lists the realized control costs at noise level $\sigma = 0.1$ for different pre-collected trajectory lengths and control approaches. We also report their percentage increases relative to the ground-truth cost computed with noise-free data from \eqref{eqn:DeePC}. The results show a clear performance ordering: $\method{L-DDPC} > \method{SPC} > \method{C-DDPC} > \method{A-DDPC} > \textrm{System ID}$. For the LTI system with noisy data, the inner problem in \eqref{eqn:bi-level} enforces LTI structure in the data-driven representation, which in turn improves noise rejection in the outer predictive control problem \eqref{eqn:bi-level}. Consequently, methods that impose more constraints on the inner problem tend to yield more structured representations and better performance. Notably, the increase in realized control cost for \method{A-DDPC} is approximately $7\%$, significantly outperforming the other \method{DDPC} variants. 

Fig. \ref{fig:trajectories-LTI-noise} shows typical trajectories for all methods under different noise levels. In this case, the open-loop trajectories for all approaches remain close to the optimal trajectory obtained by MPC with the accurate system model up to $1\, \mathrm{s}$. After that, $\method{L-DDPC}, \method{SPC}$ and $\method{C-DDPC}$ exhibit larger oscillations and deviations from the optimal trajectory, whereas the proposed \method{A-DDPC} and the sequential system identification and control are closer to the optimal one with reduced oscillations. Furthermore, across different noise levels, \method{L-DDPC}, \method{SPC} and \method{C-DDPC} show relatively large variability while \method{A-DDPC} and the system identification approach are less sensitive to noise.
\begin{figure*}[t]
\setlength{\abovedisplayskip}{0pt}
\centering
\subfigure[\method{L-DDPC}]{\includegraphics[width=0.2\textwidth,trim={3mm 0mm 8mm 0},clip]{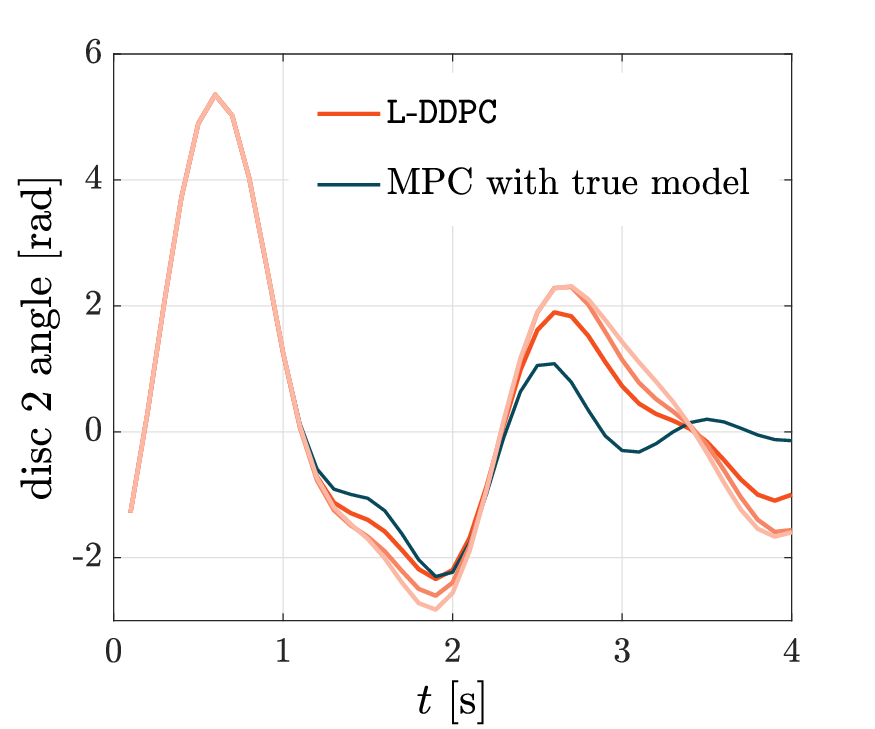}}
\hspace{-2mm}
\subfigure[\method{SPC}]{\includegraphics[width=0.2\textwidth,trim={3mm 0mm 8mm 0},clip]{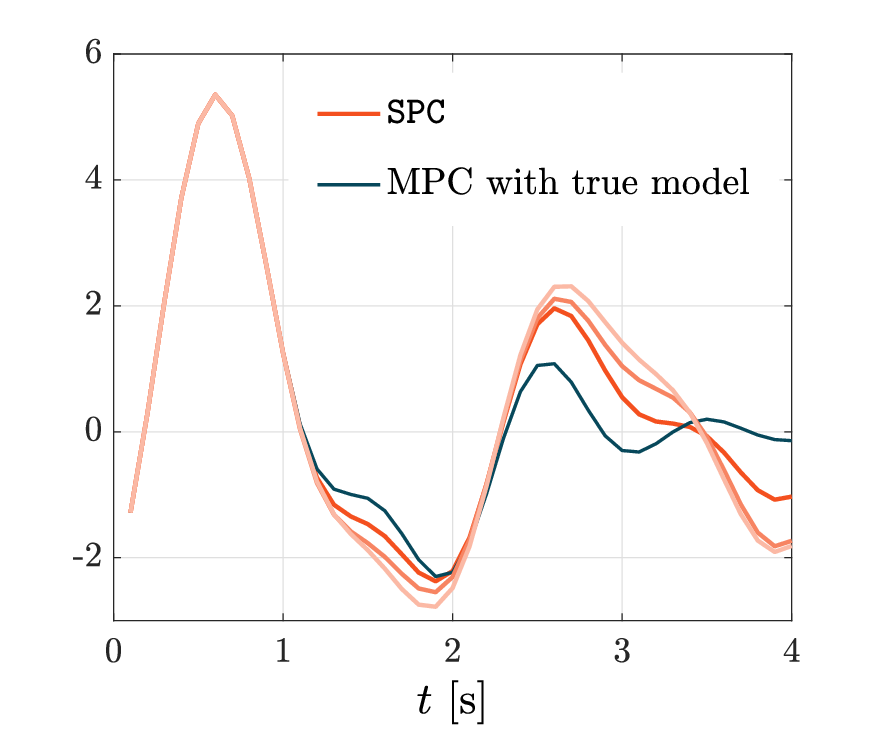}}
\hspace{-2mm}
\subfigure[\method{C-DDPC}]{\includegraphics[width=0.2\textwidth,trim={3mm 0mm 8mm 0},clip]{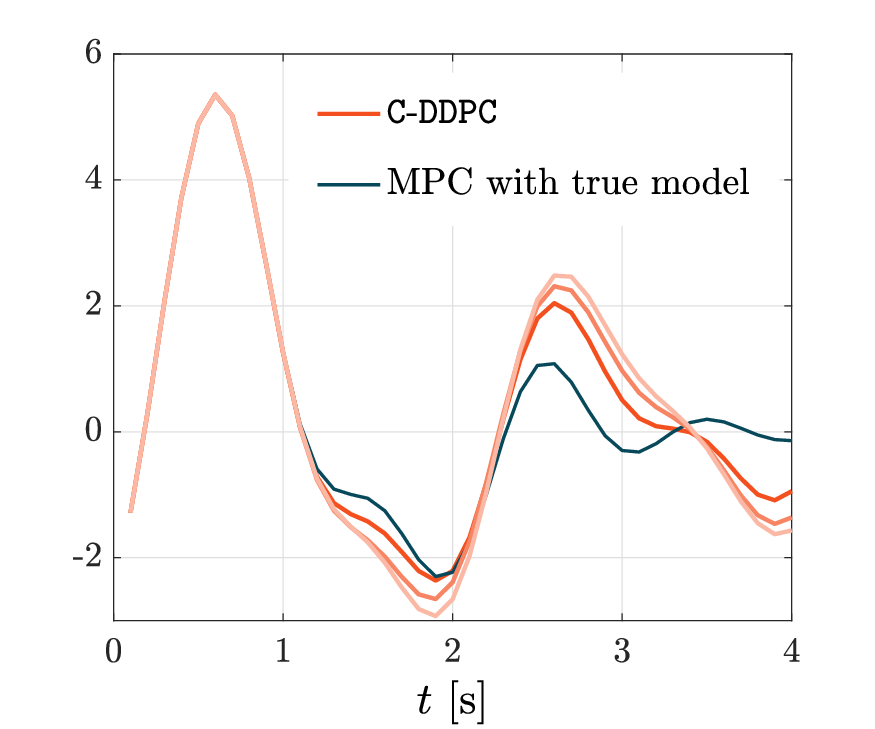}}
\hspace{-2mm}
\subfigure[\method{A-DDPC}]{\includegraphics[width=0.2\textwidth,trim={3mm 0mm 8mm 0},clip]{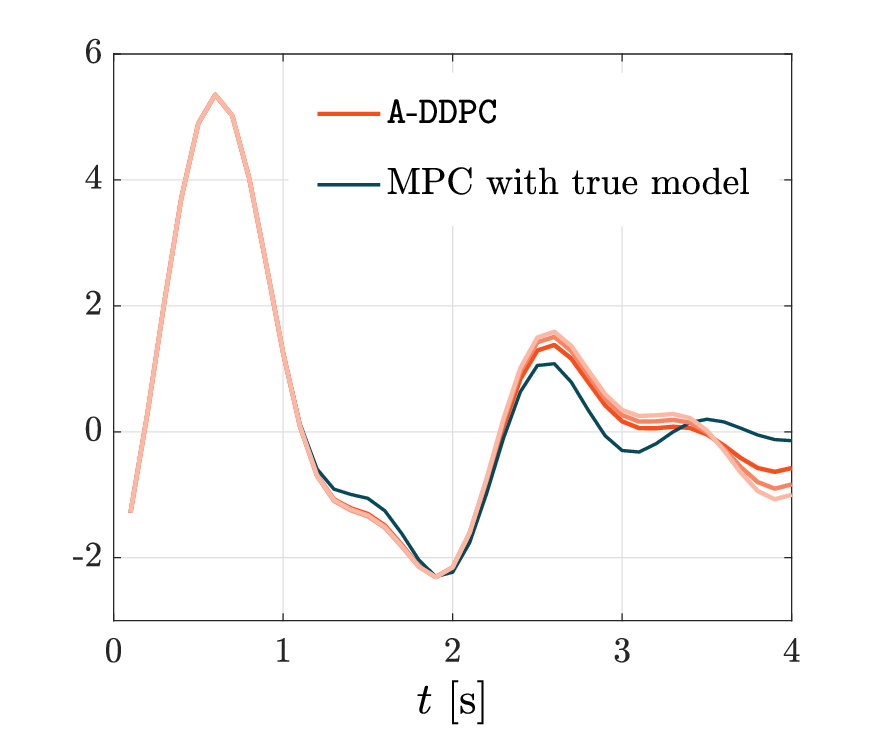}}
\hspace{-2mm}
\subfigure[System ID]{\includegraphics[width=0.2\textwidth,trim={3mm 0mm 8mm 0},clip]{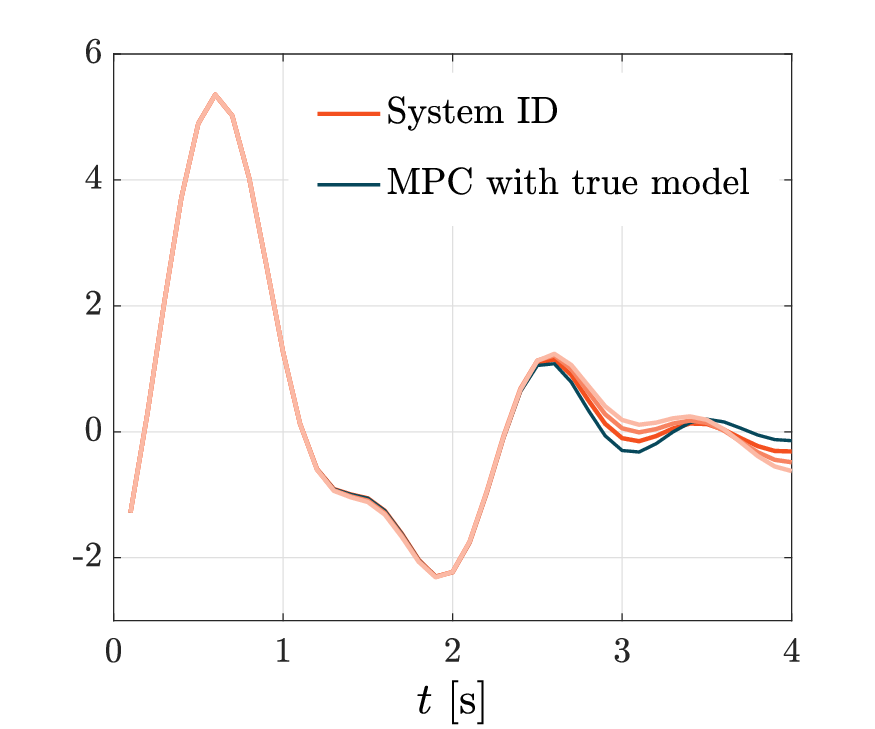}}
\caption{Output trajectories (the angle of disc 2) of different approaches. The blue trajectory denotes the optimal trajectory using MPC with the true system model. The orange trajectories from dark to light represent different \method{DDPC} variants under noise levels of $0.02$, $0.06$ and $0.1$, respectively.}
\label{fig:trajectories-LTI-noise}
\end{figure*}

\vspace{1.5mm}
\subsection{Nonlinear system} 
\label{subsec:numerical-results}
We here consider a nonlinear Lotka-Volterra system~\cite{dorfler2022bridging}
\vspace{-2mm}
\begin{equation}
\label{eqn:nonlinear-sys}
\dot{x} = 
    \begin{bmatrix}
        \dot{x}_1 \\
        \dot{x}_2
    \end{bmatrix} = \begin{bmatrix}
        a x_1 - b x_1 x_2 \\
        d x_1 x_2 - c x_2 + u
    \end{bmatrix},
    \vspace{-2mm}
\end{equation}
where $x_1, x_2$ denote prey and predator populations and $u$ is the control input. We use $a =c=0.5, b = 0.025, d=0.005$ and the discretization time step $\Delta t = 0.1$ in our experiments. We denote the discretized form of the nonlinear system \eqref{eqn:nonlinear-sys} and its linearization around the equilibrium $(\bar{u}, \bar{x}_1, \bar{x}_2) = (0, \frac{c}{d}, \frac{a}{b})$ in the error state~space as $\hat{x}(k+1) = f_{\text{nonlinear}}(\hat{x}(k), \hat{u}(k))$ and $\hat{x}(k+1) = f_{\text{linear}}(\hat{x}(k), \hat{u}(k))$, respectively.

We then construct systems with various nonlinearity, which gradually depart from the linear regime and the violation of the LTI assumption becomes more significant. We aim to evaluate how the performance of different \method{DDPC} approaches evolves as the level of model mismatch increases. Specifically, we interpolate between $f_{\textrm{nonlinear}}$ and $f_{\textrm{linear}}$ that is 
\vspace{-1mm}
\[
\begin{aligned}
\hat{x}(k+1) =  \ & \epsilon \cdot f_{\text{linear}}(\hat{x}(k), \hat{u}(k)) \\
& \quad + (1-\epsilon) \cdot f_{\text{nonlinear}}(\hat{x}(k), \hat{u}(k)),
\end{aligned}
\vspace{-1mm}
\]
and further collect the input-output trajectories of these systems with different $\epsilon$ ranging from $0$ to $1$ to formulate the associated trajectory libraries. The length of the pre-collected trajectory is $T=400$. The prediction horizon and initial sequence are set as $N = 60$ and $T_\ini = 4$. We choose $Q=I, R=0.5I$ and $\mathcal{U} = [-20, 20]$. 

\noindent \textbf{Comparison of direct/indirect methods.} We compare the realized control costs for the indirect system ID approach and different \method{DDPC} variants on systems with varying degrees of nonlinearity. Model orders are chosen to be 8 and 4 for \method{A-DDPC} in Algorithm \ref{alg:iter-SLRA} and N4SID, respectively, as they yield relatively satisfactory performance among various orders from testing. Similar to Section \ref{subsec:res-LTI}, we average the realized control costs over 100 pre-collected trajectories. We note that the identified model from N4SID is often ill-conditioned when the nonlinearity is high, which caused numerical issues in solving \eqref{eqn:predictive-control} in some of our experiments. We discard these ill-conditioned solutions when computing the average performance for the indirect system ID approach.

Both direct and indirect approaches perform well when the nonlinearity is low (\emph{i.e.}, $\epsilon \in [0.8, 1])$, as shown in Fig. \ref{fig:cost-nonlinear-All}. However, the cost for the indirect method significantly increases with higher nonlinearity, while the performance of direct methods remains relatively consistent. The superior performance of direct data-driven methods is consistent with experimental observations from~\cite{dorfler2022bridging}. The indirect system ID method projects the noisy data on a structured linear model parameterized by the matrices $(A, B,C, D)$ (see Remark \ref{rem:LTI-character}), which induces ``bias'' error due to selecting a wrong model class. On the other hand, the complexity of the LTI system is regularized but not specified in direct methods, such as \method{L-DDPC} and \method{C-DDPC}, where suitable choices of the weighting parameters provide more flexibility. For the \method{A-DDPC}, the model order $n$ in \eqref{eq:SVD-Dom-a} is tunable and can be chosen larger than the state dimension, and it does not require identification of an explicit parametric LTI model, which gives an improved capacity to capture the richer behavior of the nonlinear system. These features lead to superior performance of \method{L-DDPC}, \method{C-DDPC}, and \method{A-DDPC} for controlling nonlinear systems in our experiments.

\begin{figure}
\centering
\setlength{\abovedisplayskip}{0pt}
    \includegraphics[width=0.5\textwidth]{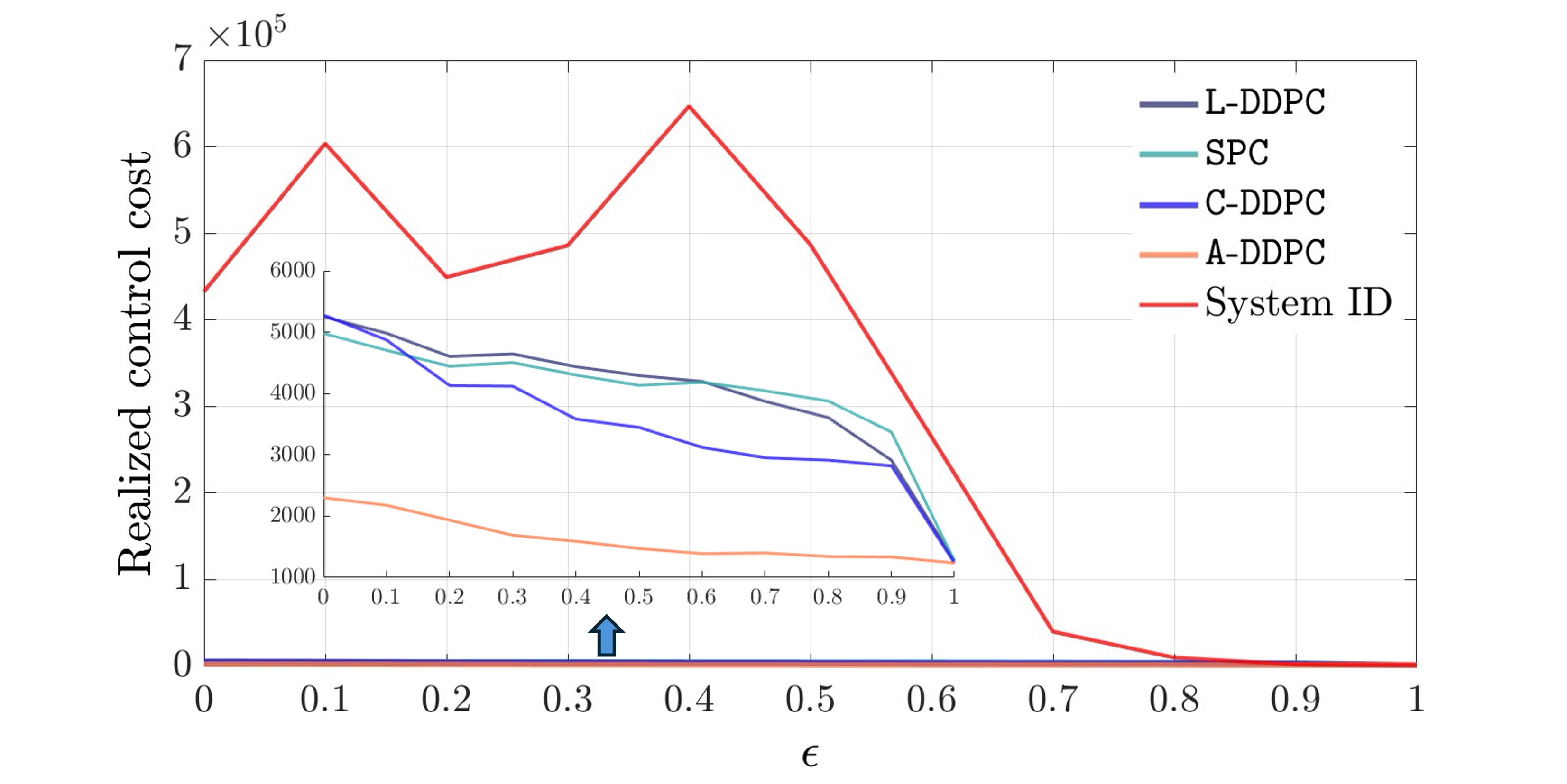} \\
    \vspace{-3.5mm}
    \caption{Comparison of realized cost for system ID and direct \method{DDPC} variants for the nonlinear system in Section \ref{subsec:numerical-results}. The System ID approach has the worst performance due to a wrong model class.}
    \label{fig:cost-nonlinear-All}
\end{figure}

\noindent \textbf{Comparison of \method{DDPC} variants.} We then compare the performance of different \method{DDPC} variants. The results are shown in Fig.~\ref{fig:cost-nonlinear}. It is obvious that realization costs for all approaches become higher with the increase of nonlinearity. The \method{C-DDPC} and \method{A-DDPC} with more structured data-driven representation outperform \method{L-DDPC} and \method{SPC} which only relax or tackle the linearity requirement without considering the causality and the Hankel structure. Furthermore, among all direct data-driven approaches, \method{A-DDPC} performs the best for both the nonlinear system (bias error) and the LTI system with measurement noise (variance noise). These numerical results suggest that we might obtain additional benefits when employing appropriate techniques from system ID to pre-process the trajectory library of nonlinear systems. Quantitatively analyzing the effect of nonlinearity is an interesting direction for future work, and the distance to the class of nonlinear systems that admit an accurate Koopman linear model \cite{shang2026existence} may serve as a useful metric.
\begin{figure}
    \centering
    \setlength{\abovedisplayskip}{0pt}
\includegraphics[width=0.5\textwidth]{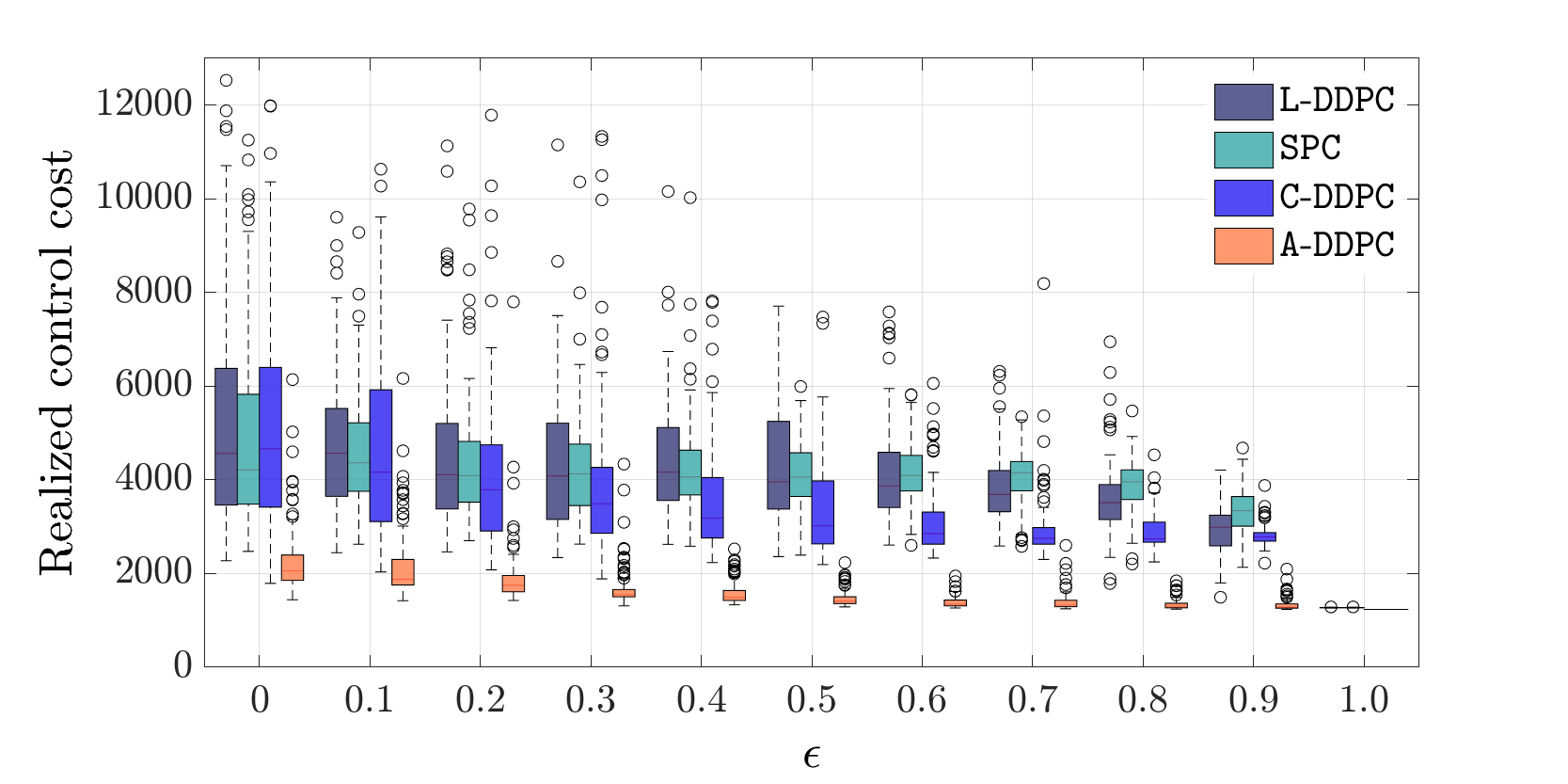}\\ 
\vspace{-3.5mm}
    \caption{ 
    Comparison of realized control cost of different \method{DDPC} variants for systems with varying nonlinearity. Across all experiments, the proposed \method{A-DDPC} \eqref{eqn:DeePC-SVD-Iter} provides the best control performance among all \method{DDPC} methods.}
    \label{fig:cost-nonlinear}
\end{figure}

%% file: 7-Conclusion.tex
\section{Conclusion}
\label{sec:conclusions}
In this paper, we have analyzed the role of regularizations in direct and indirect data-driven control via a bi-level optimization framework. We prove that, after dropping some inner constraints, the bi-level optimization problem can be transformed into a single-level convex optimization problem. Moreover, regularizers developed from those single-level problems are exact penalty functions under certain conditions. Thus, using regularizers is a further convex relaxation with respect to the corresponding bi-level problem (after dropping constraints). Moreover, we have proposed a new variant, called \method{A-DDPC}, which uses an iterative algorithm to obtain a new data-driven control approach. Numerical simulations have demonstrated the superior performance of \method{A-DDPC} \eqref{eqn:DeePC-SVD-Iter} with the more structured predictor. 

Interesting future directions include: 1) investigating the closed-loop performance of the \method{DDPC} schemes with various regularizers, 2) characterizing the relationship between the performance of different \method{DDPC} variants and the degree of system nonlinearity, and 3) developing online trajectory library update algorithms with scalable \method{DDPC} variants, which ensures control performance for nonlinear systems while preserving real-time efficiency. Another interesting direction for future work is to systematically decide the model order in \eqref{eq:SVD-Dom-a} for the proposed \method{A-DDPC} and validate it through real-world experiments, following settings similar to those in \cite{elokda2021data, fawcett2022distributed, wang2023implementation, Shang2025Decentralized, lian2023adaptive, huang2021decentralized}.

%% file: Appendix.tex
\section{Exact penalty for quadratic optimization} 
\label{appendix:exact-penalty-proof}
\subsection{A technical lemma} 
\label{appendix:penalty-proof}
We here introduce Hoffman’s error bound for a system of linear equalities. 
\begin{theorem}[Hoffman's error bound \cite{hoffman1952approximate}]
    \label{them:Hoffman}
        Let the system of linear equalities $Ax = \bar{b}$ with $x \in \mathbb{R}^{n_x}$ be consistent and $\bar{\mathcal{F}}$ be its feasible solution set. Fix any $p \in [1, \infty]$. Then, there exists a constant $\bar{L}$, such that for any $x \in \mathbb{R}^n$ there exists a point  $x^* \in \bar{\mathcal{F}}$ satisfying  
        $
        \|x -x^*\|_p \le \bar{L} \|Ax-\bar{b}\|_p.
        $
\end{theorem}
This theorem shows that if a vector does not satisfy the linear equality constraints, then its distance to the feasible set can be bounded by the size of the constraint violation. We now state a key technical lemma that will be used in the proof of Theorem~\ref{them:x-exact-main}. 
\begin{lemma}[Linear perturbation]
\label{lemma:perturbation}
Let $A_1 \in \mathbb{R}^{m \times n_1}$, $A_2 \in \mathbb{R}^{m \times n_2}$, $D \in \mathbb{R}^{s \times n_2}$,
and $b \in \mathbb{R}^m$. Suppose the set
\[
\mathcal{F}=\{(x_1,x_2)\mid A_1 x_1 + A_2 x_2 = b,\ \ Dx_2  = 0\}
\]
is non-empty. Fix any $p\in[1,\infty]$. Then there exists a finite constant $L>0$ such that for any pair
$(x_1,x_2)$ satisfying $A_1 x_1 + A_2 x_2 = b$, there exists a pair $(\tilde{x}_1,\tilde{x}_2)$ satisfying
$
(\tilde{x}_1,\tilde{x}_2) \in \mathcal{F}, \, \|x_1 - \tilde{x}_1\|_p \leq L \|Dx_2\|_p.
$
\end{lemma}
\begin{proof}
    We can construct $A:= \begin{bmatrix}
        A_1 & A_2 \\ 0 & D
    \end{bmatrix}, \bar{b} := \begin{bmatrix}
        b \\ 0
    \end{bmatrix}$ and then $\mathcal{F}$ is the set of feasible solutions for constraint $Ax = \bar{b}$. From Theorem~\ref{them:Hoffman}, there exists $L$ such that, for any $x := \col(x_1, x_2)$ there exists $\tilde{x} := (\tilde{x}_1, \tilde{x}_2) \in \mathcal{F}$ with
    \[
    \begin{aligned}
    \|x_1 - \tilde{x}_1\|_p &\le \|x-\tilde{x}\|_p \\
    &\le L \|\begin{bmatrix}
        A_1 x_1 + A_2 x_2 -b \\ Dx_2
    \end{bmatrix}\|_p = L \|Dx_2\|_p,
     \end{aligned}
    \]
    where the last equality comes from $A_1x_1 + A_2 x_2 = b$. This completes the proof.
\end{proof}

\subsection{Proof of Theorem~\ref{them:x-exact-main}}

Our proof below adapts the idea of partial calmness \cite{ye1997exact} into the quadratic optimization problems in Theorem~\ref{them:x-exact-main}. 

We let $f(x)\!:=\!x^\tr Mx$ for simplicity. Since $\|\!\cdot \! \|_p$ is a norm, the constraint $\|Dx_2\|_p\!=\!0$ is equivalent to $Dx_2\!=\!0$.
It suffices to prove that there exists a finite $\lambda_w^*>0$ such that for every $\lambda_w>\lambda_w^*$,
every global minimizer $(x_1,x_2)$ of \eqref{eqn:qp-pn-main} must satisfy $Dx_2=0$.
Indeed, if this holds, then all minimizers of \eqref{eqn:qp-pn-main} are feasible for \eqref{eqn:qp-c-main}.
On the set $\{(x_1,x_2): Dx_2=0\}$, the penalty term vanishes, so both problems have the same
objective values and constraints. Therefore, the two problems have the same optimal value and the same set of optimal solutions.

\textbf{Step 1: Every minimizer $(x_1,x_2)$ of \eqref{eqn:qp-pn-main} must satisfy $Dx_2=0$}.  
We claim that there exist constants $\delta>0$ and $\lambda_w^*>0$ such that for any $\lambda_w>\lambda_w^*$,
the inequality
\begin{equation}\label{eq:calm-ineq-1}
f(x_1)+\lambda_w\|Dx_2\|_p
\ \ge\
f(x_1^*)+(\lambda_w-\lambda_w^*)\|Dx_2\|_p
\end{equation}
holds for all feasible $(x_1,x_2)$ of \eqref{eqn:qp-pn-main} satisfying
\begin{equation}\label{eq:perturb-x}
\|x_1-x_1^*\|_p \le \delta,\qquad \|Dx_2\|_p \le \delta,
\end{equation}
where $(x_1^*,x_2^*)$ is the given optimal solution of \eqref{eqn:qp-c-main}. We note that, the existence of such feasible solution is guaranteed as we can trivially choose $x_1 = x_1^*$ and $x_2 = x_2^*$.
%We postpone the proof of \eqref{eq:calm-ineq-1} to Step~2.

Assuming \eqref{eq:calm-ineq-1} for now, fix any $\lambda_w>\lambda_w^*$.
Since $(x_1^*,x_2^*)$ is feasible for \eqref{eqn:qp-pn-main} and satisfies $Dx_2^*=0$, it attains
objective value $f(x_1^*)$ in \eqref{eqn:qp-pn-main}.
Moreover, \eqref{eq:calm-ineq-1} implies that any feasible point in the neighborhood \eqref{eq:perturb-x} has
objective value at least $f(x_1^*)$, with strict inequality whenever $Dx_2\neq 0$.
Hence $(x_1^*,x_2^*)$ is a local minimizer of \eqref{eqn:qp-pn-main}.
Since \eqref{eqn:qp-pn-main} is convex, every local minimizer is global, so $(x_1^*,x_2^*)$ is a global
minimizer of \eqref{eqn:qp-pn-main}.

Now let $(\bar x_1,\bar x_2)$ be any global minimizer of \eqref{eqn:qp-pn-main}. If $D\bar x_2\neq 0$, consider the
convex combination
\[
(x_{1,t},x_{2,t}) := (1-t)(x_1^*,x_2^*)+t(\bar x_1,\bar x_2),\qquad t\in(0,1),
\]
which is feasible to \eqref{eqn:qp-pn-main} for all $t\in[0,1]$. 
Since the objective is convex and both endpoints are global minimizers, each $(x_{1,t},x_{2,t})$ is also a global
minimizer. % and attains the same optimal value as $(x_1^*,x_2^*)$.
Choose $t>0$ small enough so that \eqref{eq:perturb-x} holds, \emph{i.e.},
$
t\|\bar x_1-x_1^*\|_p \le \delta,\, 
t\|D\bar x_2\|_p \le \delta.
$ 
Then $Dx_{2,t}=tD\bar x_2\neq 0$, and applying \eqref{eq:calm-ineq-1} gives
$
f(x_{1,t})+\lambda_w\|Dx_{2,t}\|_p
\ \ge\
f(x_1^*)+(\lambda_w-\lambda_w^*)\|Dx_{2,t}\|_p
\ >\ f(x_1^*),
$ 
contradicting that $(x_{1,t},x_{2,t})$ is a global minimizer.
Therefore, every global minimizer $(\bar x_1,\bar x_2)$ of \eqref{eqn:qp-pn-main} must satisfy $D\bar x_2=0$.
This completes the proof once we establish \eqref{eq:calm-ineq-1}.

\textbf{Step 2: Proof of the key inequality \eqref{eq:calm-ineq-1}.}
Since $x_1^*$ is an interior point of $\mathcal X$, there exists $r>0$ such that the $p$-norm ball
$\mathcal{B}_p(x_1^*;r)\subseteq \mathcal X$.
Because $f$ is quadratic and $\mathcal{B}_p(x_1^*;r)$ is bounded, $f$ is Lipschitz on this ball, \emph{i.e.}, there exists $\kappa_r > 0$ such that 
\begin{equation}\label{eq:Lip-local}
|f(u)-f(v)|\le \kappa_r\|u-v\|_p,\quad \forall u,v\in\mathcal{B}_p(x_1^*;r).
\end{equation}

Since $(x_1^*,x_2^*)$ is feasible for \eqref{eqn:qp-c-main}, the set
$\mathcal F:=\{(x_1,x_2):A_1x_1+A_2x_2=b,\ Dx_2=0\}$ is nonempty. Hence Lemma~\ref{lemma:perturbation} yields a
constant $L>0$ such that for any $(x_1,x_2)$ satisfying $A_1x_1+A_2x_2=b$, there exists
$(\tilde x_1,\tilde x_2)\in\mathcal F$ with
\begin{equation}\label{eq:perturbation-bound}
\|x_1-\tilde x_1\|_p \le L\|Dx_2\|_p.
\end{equation}
Now take any feasible $(x_1,x_2)$ of \eqref{eqn:qp-pn-main} satisfying \eqref{eq:perturb-x} and let
$(\tilde x_1,\tilde x_2)\in\mathcal F$ be the corresponding perturbed point from Lemma~\ref{lemma:perturbation}.
Then
\[
\|\tilde x_1-x_1^*\|_p
\le \|x_1-x_1^*\|_p + \|x_1-\tilde x_1\|_p
\le \delta + L\delta
= r,
\]
so $\tilde x_1\in\mathcal{B}_p(x_1^*;r)\subseteq \mathcal X$ via choosing $\delta := r/(1+L)$. That illustrates $(\tilde x_1,\tilde x_2)$ is feasible for
\eqref{eqn:qp-c-main}. By optimality of $(x_1^*,x_2^*)$ for~\eqref{eqn:qp-c-main},
\[
f(\tilde{x}_1) \ge f(x_1^*).
\]
Moreover, since $x_1\in\mathcal X$, $\tilde x_1\in\mathcal{B}_p(x_1^*;r)\subseteq\mathcal X$, and both are
in $\mathcal{B}_p(x_1^*;r)$ whenever \eqref{eq:perturb-x} holds (note that $\delta < r$ from construction), the local Lipschitz bound
\eqref{eq:Lip-local} gives
\[
f(x_1)
\ge f(\tilde x_1) - \kappa_r\|x_1-\tilde x_1\|_p
\ge f(x_1^*) - \kappa_r L\|Dx_2\|_p,
\]
where we used \eqref{eq:perturbation-bound} in the last step. Therefore, for any $\lambda_w\ge 0$,
\[
f(x_1)+\lambda_w\|Dx_2\|_p
\ge f(x_1^*) + (\lambda_w-\kappa_rL)\|Dx_2\|_p.
\]
This is exactly \eqref{eq:calm-ineq-1} with $\lambda_w^* := \kappa_rL$. This completes the proof.

\section{Technical proofs}
\subsection{Relation with the classcial \method{SPC}}
\label{appendix:equal-spc-ddspc}

The  classical SPC is of the following form 
\begin{equation}
\begin{aligned}
\min_{\sigma_y \in \Gamma, u \in \mathcal{U}, y \in \mathcal{Y}} \;\, &  \|y\|_Q^2 + \|u\|_R^2 + \lambda_y \|\sigma_y\|_2^2 \\
\mathrm{subject~to} \;\, & 
y
= Y_\f 
\begin{bmatrix}
    U_\p \\
    Y_p \\
    U_\f
\end{bmatrix}^\dag
\begin{bmatrix}
u_{\ini}\\
y_{\ini}+\sigma_y\\
u\\
\end{bmatrix}, \label{eqn:SPC-main}
\end{aligned}
\end{equation}
which does not have the variable $g$. We can establish the following equivalence. 

\begin{proposition}
\label{proposition:data-driven-SPC}
    If $Q\! \succ \! 0, R \! \succ \! 0$ and $H_1\!=\! \col(U_\p, Y_\p, U_\f)$ in \eqref{eq:projection-matrix} has full row rank, then \eqref{eqn:DD-SPC} and \eqref{eqn:SPC-main} have the same optimal solution $u^*, y^*$ and $\sigma_y^*$, $\forall \lambda_y >0$.
\end{proposition} 

Our proof is divided into two main parts:
\begin{enumerate}
    \item When $H_1= \col(U_\p, Y_\p, U_\f)$ has full row rank, we show that \eqref{eqn:DD-SPC} and \eqref{eqn:SPC-main} have the same feasible region: if $\sigma_y, u, y, g$ is feasible to \eqref{eqn:DD-SPC}, then the same $\sigma_y, u, y$ is also feasible for \eqref{eqn:SPC-main}. Conversely, given any feasible solution $\sigma_y, u, y$ to \eqref{eqn:SPC-main}, there exists a vector $g$ such that $\sigma_y, u, y, g$ is feasible to \eqref{eqn:DD-SPC}.
    \item The \eqref{eqn:DD-SPC} and \eqref{eqn:SPC-main} have the same cost function in terms of $u, y, \sigma_y$. 
\end{enumerate}
Combining the two properties above with the fact that the cost function in \eqref{eqn:SPC-main} is strongly convex, we conclude \eqref{eqn:DD-SPC} and \eqref{eqn:SPC-main} have the same unique optimal solution for decision variables $\sigma_y, u$ and $y$.

The property 2 above is obvious. We prove the property~1 below. Let us first decompose 
     \begin{equation} \label{eq:ortho-decomposition-Yf}
     Y_\f = M+M^\bot,
     \end{equation}
     where $M$ (\emph{i.e.}, $Y_\f \Pi_1$) is the (row space) orthogonal projection of $Y_\f$ on the row space of $H_1$ and $M^\bot$ (\emph{i.e.}, $Y_\f(I - \Pi_1)$) is the rest part of $Y_\f$ in the null space of $H_1$. Since $H_1$ has full row rank, thanks to the property of Moore–Penrose inverse, we have $H_1 H_1^\dag = I$ and the range space of $H_1^\dag$ is the same as the row space of $H_1$, which means $M^\bot H_1^\dag = 0$. 
    
    We assume $u_1, y_1, \sigma_{y_1}, g_1$ is a feasible solution for \eqref{eqn:DD-SPC}. Then, without loss of generality, the vector $g_1$ can be represented as 
    \[
    g_1 = H_1^\dag \col(
        u_\ini, y_\ini + \sigma_{y_1}, u_1
    ) + \hat{g}
    \]
    where $\hat{g}$ is a vector in the null space of $H_1$. We have $M \hat{g} = 0$ since $H_1\hat{g} = 0$, and $Y_\f H_1^\dag = (M+M^\bot)H_1^\dag = M H_1^\dag$ because $M^\bot H_1^\dag = 0$. Thus, from the equality constrain in \eqref{eqn:DD-SPC}, the vector $y_1$ satisfies
    \begin{equation}
    \label{eqn:relate-y-YF}
    \begin{aligned}
    y_1 & = M g_1 = M H_1^\dag \begin{bmatrix}
        u_\ini \\ y_\ini + \sigma_{y_1} \\ u_1
    \end{bmatrix} + M \hat{g} \\
    & = M H_1^\dag \begin{bmatrix}
        u_\ini \\ y_\ini + \sigma_{y_1} \\ u_1
    \end{bmatrix} 
    = Y_\f H_1^\dag \begin{bmatrix}
        u_\ini \\ y_\ini + \sigma_{y_1} \\ u_1
    \end{bmatrix},
    \end{aligned}
    \end{equation}
    which means $u_1, y_1, \sigma_{y_1}$ is also a feasible solution of $\eqref{eqn:SPC-main}$.
    
    We next assume $u_1, y_1, \sigma_{y_1}$ is a feasible solution for \eqref{eqn:SPC-main}. Substituting the orthonormal decomposition \eqref{eq:ortho-decomposition-Yf} into the equality constraint of \eqref{eqn:SPC-main}, we have
    \begin{equation}
    \label{eqn:Yf_derive}
    \begin{aligned}
    y_1 & = Y_\f H_1^\dag \begin{bmatrix}
        u_\ini \\ y_\ini + \sigma_{y_1} \\ u_1
    \end{bmatrix} \\
    & = (M +M^\bot) H_1^\dag \begin{bmatrix}
        u_\ini \\ y_\ini + \sigma_{y_1} \\ u_1
    \end{bmatrix} \\
     & = M H_1^\dag \begin{bmatrix}
        u_\ini \\ y_\ini + \sigma_{y_1} \\ u_1
    \end{bmatrix}.
    \end{aligned}
    \end{equation}
    Upon defining $g_1 = H_1^\dag \col(
        u_\ini, y_\ini + \sigma_{y_1}, u_1
    )$, we have $y_1 = M g_1$ from \eqref{eqn:Yf_derive}. We then substitute $g_1$ into the equality constraint of \eqref{eqn:DD-SPC}, leading to
    \[
    \begin{bmatrix}
        H_1 \\ M
    \end{bmatrix} g_1 = 
    \begin{bmatrix}
    H_1 H_1^\dag \begin{bmatrix}
        u_\ini \\ y_\ini + \sigma_{y_1} \\ u_1
    \end{bmatrix}  \\
    M g_1
    \end{bmatrix}
    = \begin{bmatrix}
        u_\ini \\ y_\ini + \sigma_{y_1} \\ u_1 \\
        y_1
    \end{bmatrix},
    \]
where we have used the fact $H_1 H_1^\dag = I$ since $H_1$ has full row rank. This means that $u_1, y_1, \sigma_{y_1}, g_1$ is a feasible solution for \eqref{eqn:DD-SPC}. This completes our proof.  

\subsection{Proof of Proposition \ref{prop:spc-equival}}
\label{appendix:spc-equival}
As \eqref{eqn:bi-level-SPC} and \eqref{eqn:DD-SPC} are equivalent directly from construction, we here mainly prove the equivalence of \eqref{eqn:DD-SPC} and \eqref{eqn:reg-proj}. It is obvious that \eqref{eqn:DD-SPC} and \eqref{eqn:reg-proj} have the same objective function and it only contains decision variables $u, y$ and $\sigma_y$. Thus, we show that \eqref{eqn:DD-SPC} and \eqref{eqn:reg-proj} provide the same unique optimal solution $u^*, y^*$ and $\sigma_y^*$ by proving:
\begin{enumerate}
    \item Feasible regions of $u,y, \sigma_y$ are the same for~\eqref{eqn:DD-SPC} and \eqref{eqn:reg-proj}: if $u, y, \sigma_y, g$ is feasible to  \eqref{eqn:reg-proj}, then the same $u, y, \sigma_y, g$ is also feasible for \eqref{eqn:DD-SPC}. Conversely, given any feasible solution $u, y, \sigma_y, g$ to \eqref{eqn:DD-SPC}, there exists a vector $\tilde{g}$ such that $u, y, \sigma_y, \tilde{g}$ is feasible to~\eqref{eqn:reg-proj}.
    \item The optimal solution of $u, y, \sigma_y$ is unique for~\eqref{eqn:DD-SPC}. 
\end{enumerate}

We assume $u_1, y_1, \sigma_{y_1}, g_1$ is a feasible solution for~\eqref{eqn:reg-proj}. Substituting the orthogonal decomposition \eqref{eq:ortho-decomposition-Yf} of $Y_\f$ into \eqref{eqn:reg-proj-1}, we have
\vspace{-2.5mm}
\[
\begin{aligned}
H_\D g_1
& = \col(U_\p, Y_\p, U_\f, M + Y_\f(I-\Pi_1)) g_1  \\
& = H_\s^* g_1  =  \col(
    u_\ini, y_\ini+\sigma_{y_1}, u_1, y_1), 
\end{aligned}
\vspace{-2.5mm}
\]
where we have applied the fact that $(I-\Pi_1)g_1=0$ from \eqref{eqn:reg-proj-2}. Thus, the set of variables $u_1, y_1, \sigma_{y_1}$ and $g_1$ is also a feasible solution for~$\eqref{eqn:DD-SPC}$.

We next assume $u_1, y_1, \sigma_{y_1}$ and $g_1$ is a feasible solution for \eqref{eqn:DD-SPC}. We define $\tilde{g}_1 = H_1^\dag 
\col(u_\ini, y_\ini + \sigma_{y_1}, u_1)$, which satisfies $y_1 = Y_\f \tilde{g}_1$ from \eqref{eqn:relate-y-YF}. We first verify that $\tilde{g}_1$, together with $u_1, y_1, \sigma_{y_1}$, satisfies \eqref{eqn:reg-proj-1}: 
\vspace{-2.5mm}
\[
\begin{aligned}
H_\D \tilde{g}_1&  = 
\begin{bmatrix}
    H_1   \\
    Y_\f
\end{bmatrix}
H_1^\dag \col(u_\ini, y_\ini + \sigma_{y_1},u_1) \\
        & 
        = \col(
            u_\ini, y_\ini + \sigma_{y_1}, u_1, y_1).
\end{aligned}
\vspace{-2.5mm}
\]
For the satisfaction of \eqref{eqn:reg-proj-2}, since $\tilde{g}_1$ is in the range space of $H_1^\dag$ and $\Pi_1$ is the orthogonal projector onto the row space of $H_1$, we have $\Pi_1 \tilde{g}_1 = \tilde{g}_1$ (the range space of $H_1^\dag$ and row space of $H_1$ are equivalent), which implies $\|(I-\Pi_1)\tilde{g}_1\|_2 = \|\tilde{g}_1 -\tilde{g}_1\|_2 = 0$. Thus, $u_1, y_1, \sigma_{y_1}$ and $\tilde{g}_1$ is a feasible solution for~\eqref{eqn:reg-proj}.

The uniqueness of the optimal solution $u^*, y^*$ and $\sigma_y^*$ for \eqref{eqn:DD-SPC} basically comes from strong convexity. For notational simplicity, we define $x_\D = \col(u, y, \sigma_{y}, g)$ and $f_\D(x_\D) = \|y\|_Q^2 + \|u\|_R^2 + \lambda_y \|\sigma_y\|_2^2$ as the decision variable and objective function of \eqref{eqn:DD-SPC}, respectively. Suppose that $x_{\D_1}$ and $x_{\D_2}$ are two optimal solutions with different $u, y$ or $\sigma_y$. Let the optimal value be $f_\D^*$. We then construct a convex combination $x_{\D_3} = \alpha x_{\D_1} + (1- \alpha) x_{\D_2}$ where $0 < \alpha < 1$. This new point $x_{\D_3}$ is also feasible as
\vspace{-2.5mm}
\[
\begin{aligned}
\begin{bmatrix}
    U_\p \\ Y_\p \\ U_\f \\ M
\end{bmatrix} g_3 &=
\begin{bmatrix}
    U_\p \\ Y_\p \\ U_\f \\ M
\end{bmatrix} (\alpha g_1 + (1-\alpha)g_2) \\
& = \begin{bmatrix}
    u_\ini \\ y_\ini + \alpha \sigma_{y_1} + (1-\alpha) \sigma_{y_2} \\ \alpha u_1 + (1-\alpha)u_2 \\ \alpha y_1 + (1-\alpha)y_2 
\end{bmatrix} = 
\begin{bmatrix}
    u_\ini \\ y_\ini + \sigma_{y_3} \\ u_3 \\ y_3   
\end{bmatrix}. 
\end{aligned}
\] 
It is obvious that $f_\D(\cdot)$ is a strongly convex function with respect to $u, y$ and $\sigma_y$ and its value is not affected by $g$. Thus, we have
$
f_\D(x_{\D_3}) = f_\D(\alpha x_{\D_1} + (1-\alpha) x_{\D_2}) < \alpha f_\D(x_{\D_1}) + (1-\alpha) f_\D(x_{\D_2}) = f_\D^*,
$ 
which contradicts our assumption. The optimal solution to \eqref{eqn:DD-SPC} is thus unique. This completes our proof.

\subsection{Proof of Proposition \ref{them:c-spc-equival}}
\label{appendix:c-spc-equival}
Similar to the proof in the Appendix \ref{appendix:spc-equival}, we here mainly prove feasible regions of $u, y, \sigma_y$ are equivalent for \eqref{eqn:c-gamma-DeePC} and \eqref{eqn:reg-causal}. Then, with the same objective function, their optimal solutions are also the same. The proof for the uniqueness of the optimal solution utilizes the strong convexity of the objective function of \eqref{eqn:c-gamma-DeePC} following the same procedure in Appendix \ref{appendix:spc-equival}. We omit the proof for the uniqueness here. We recall that we have $H_1 = \col(U_\p, Y_\p, U_\f) = \begin{bmatrix}
    L_{11} & \mathbb{0} \\ L_{21} & L_{22}
\end{bmatrix} \begin{bmatrix}
    Q_1 \\ Q_2
\end{bmatrix}$ from the LQ decomposition of $H_\D$. As we assume $H_\D$ has full row rank, the matrix $\begin{bmatrix}
    L_{11} & \mathbb{0} \\ L_{21} & L_{22}
\end{bmatrix}$ is invertible.

Let $u_1, y_1, \sigma_{y_1}$ and $g_1$ be a set of feasible solution for \eqref{eqn:reg-causal}. Substituting \eqref{eqn:reg-causal-2} into \eqref{eqn:reg-causal-1} leads to
\vspace{-3mm}
{\small\[
\begin{aligned}
\hat{H} g_1 \! =\! \begin{bmatrix}
        L_{11} & \mathbb{0} & \mathbb{0} & \mathbb{0}\\
        L_{21} & L_{22} & \mathbb{0} & \mathbb{0}\\
        L_{31} & L_{32}^* & L_{33} & L_{32}'
\end{bmatrix}\! 
    \begin{bmatrix}
        Q_1 g_1 \\ Q_2 g_1 \\ \mathbb{0} \\ \mathbb{0}
    \end{bmatrix} 
     \!= \! H_\sca^* g_1 \!= \! \begin{bmatrix}
            u_\ini \\ y_\ini + \sigma_{y_1} \\ u_1 \\ y_1
        \end{bmatrix},
\end{aligned}
\]}
thus, $u_1, y_1, \sigma_{y_1}$ and $g_1$ is also a feasible solution for~\eqref{eqn:c-gamma-DeePC}.

We next choose a feasible solution $u_1, y_1, \sigma_{y_1}$ and $g_1$ of~\eqref{eqn:c-gamma-DeePC}. We can construct $\tilde{g}_1$ as 
\vspace{-2mm}
\[
\begingroup
    \setlength\arraycolsep{5pt}
    \def\arraystretch{0.95} 
\label{eqn:g-construct}
\begin{aligned}
\tilde{g}_1 & = H_1^\dag \begin{bmatrix}
    u_\ini \\ y_\ini + \sigma_{y_1} \\ u_1
\end{bmatrix} \\
&= \begin{bmatrix}
    Q_1^\tr & Q_2^\tr 
\end{bmatrix}
\begin{bmatrix}
    L_{11} & \mathbb{0} \\ L_{21} & L_{22}
\end{bmatrix}^{-1} \begin{bmatrix}
    u_\ini \\ y_\ini + \sigma_{y_1} \\ u_1
\end{bmatrix},
\end{aligned}
\vspace{-2mm}
\endgroup
\]
so that $u_1, y_1, \sigma_{y_1}$ and $\tilde{g}_1$ is an optimal solution of \eqref{eqn:reg-causal}. Since $\col(Q_1, Q_2, Q_3, Q^*)$ has orthonormal rows, we have 
\[
Q_\cs \tilde{g}_1 =\underbrace{\begin{bmatrix}
    Q_3 \\ Q^*
\end{bmatrix} \begin{bmatrix}
    Q_1^\tr & Q_2^\tr 
\end{bmatrix}}_{=\mathbb{0}}
\begin{bmatrix}
    L_{11} & \mathbb{0} \\ L_{21} & L_{22}
\end{bmatrix}^{-1} \begin{bmatrix}
    u_\ini \\ y_\ini + \sigma_{y_1} \\ u_1
\end{bmatrix} = \mathbb{0}
\] 
and \eqref{eqn:reg-causal-2} is naturally satisfied. Further substituting $\tilde{g}_1$ to \eqref{eqn:reg-causal-1} leads to
\[
\begin{aligned}
&\begin{bmatrix}
    L_{11} & \mathbb{0} & \mathbb{0} & \mathbb{0} \\
    L_{21} & L_{22} & \mathbb{0} & \mathbb{0} \\
    L_{31} & L_{32}^* & L_{33} & L_{32}'
\end{bmatrix}
\begin{bmatrix}
    Q_1 \\ Q_2 \\ Q_3 \\ Q^*
\end{bmatrix} \tilde{g}_1 \\
= & \begin{bmatrix}
u_{\ini}\\
y_{\ini}+\sigma_{y_1}\\
u_1\\
\begin{bmatrix}
    L_{31} & L_{32}^*
\end{bmatrix}
\begin{bmatrix}
    Q_1 \\ Q_2
\end{bmatrix}\tilde{g}_1\\
\end{bmatrix} 
=\begin{bmatrix}
u_{\ini}\\
y_{\ini}+\sigma_{y_1}\\
u_1\\
y_1\\
\end{bmatrix},
\end{aligned}
\]
where the second equality comes from we can represent $g_1$ as $g_1 = \tilde{g}_1 + \hat{g}$ with $\hat{g} \in \textrm{null}(H_1) = \textrm{null}(\col(Q_1, Q_2))$ and then
\[
\begin{bmatrix}
    L_{31} & L_{32}^*
\end{bmatrix}
\begin{bmatrix}
    Q_1 \\ Q_2
\end{bmatrix}\tilde{g}_1
= \begin{bmatrix}
    L_{31} & L_{32}^*
\end{bmatrix}
\begin{bmatrix}
    Q_1 \\ Q_2
\end{bmatrix}g_1= y_1.
\vspace{-2mm}
\]
That shows the satisfaction of \eqref{eqn:reg-causal-1}.

\subsection{Proof of Corollaries \ref{them:DeeP-exact} and \ref{them:c-DeeP-exact}}
\label{appendix:exact-penalty}
We prove Corollaries \ref{them:DeeP-exact} and \ref{them:c-DeeP-exact} via showing both \eqref{eqn:reg-proj} and \eqref{eqn:reg-causal} can be represented in the form of \eqref{eqn:qp-c-main}. Thus, Corollaries \ref{them:DeeP-exact} and \ref{them:c-DeeP-exact} are actually special cases of Theorem~\ref{them:x-exact-main}.

For changing \eqref{eqn:reg-proj} to the same form as \eqref{eqn:qp-c-main}, we can let
\begin{gather}
    x_1 := \col(\sigma_y, u, y), \quad  x_2 := g,  \quad \mathcal{X} := \Gamma \times \mathcal{U} \times \mathcal{Y},  \nonumber\\
    A_1 := \col(\mathbb{0}, I), \quad  A_2 := -H_\D, \quad  b := -\col(u_\ini, y_\ini, \mathbb{0}), \nonumber \\
    M:=\mathrm{diag}(\lambda_y I, R, Q), \quad D:=(I-\Pi_1). \nonumber
\end{gather}
Similarly, we transform \eqref{eqn:reg-causal} to the form of \eqref{eqn:qp-c-main} by letting 
\begin{gather}
    x_1 := \col(\sigma_y, u, y), \quad  x_2 := g,  \quad \mathcal{X} := \Gamma \times \mathcal{U} \times \mathcal{Y},  \nonumber\\
    A_1 := \col(\mathbb{0}, I), \quad  A_2 := -\hat{H}, \quad  b := -\col(u_\ini, y_\ini, \mathbb{0}), \nonumber \\
    M:=\mathrm{diag}(\lambda_y I, R, Q), \quad D:=Q_\cs. \nonumber
\end{gather}
We note that the only differences between transformations of \eqref{eqn:reg-proj} and \eqref{eqn:reg-causal} are the choices of $D$ and $A_2$. 

As $\Gamma, \mathcal{U}, \mathcal{Y}$ are convex sets, $\mathcal{X}$ defined above for \eqref{eqn:reg-proj} and \eqref{eqn:reg-causal} are convex. Thus, all conditions in Theorem \ref{them:x-exact-main} are satisfied and that completes the proof.

\subsection{Proof of Proposition \ref{prop:opt-low-rank}}
\label{appendix:opt-low-rank}
The key idea for the proof of the Proposition \ref{prop:opt-low-rank} is that we can first partition both $H_y$ and $\tilde{H}_y$ into the part in the row space of $H_u$ and the part in the null space of $H_u$, \emph{i.e.}, $H_y = H_1 + H_2$, $\tilde{H}_y = \tilde{H}_1 + \tilde{H}_2$, where $\row(H_1), \row(\tilde{H}_1) \subseteq \row(H_u)$ and $\row(H_2), \row(\tilde{H}_2) \subseteq \nsp(H_u)$. Then, we illustrate that the optimal solution for $\tilde{H}_1$ is $H_1$ and thus the problem becomes approximately $H_2$ with a low rank matrix $\tilde{H}_2$ which admits an analytical solution via SVD.

We can equivalently formulate the optimization problem \eqref{eqn:low-rank-approx} as 
\begin{equation}
\label{eqn:low-rank-approx-sep}
    \begin{aligned} 
\min_{\tilde{H}_1, \tilde{H}_2} \quad & \|H_1 + H_2 - \tilde{H}_1 - \tilde{H}_2 \|_F \\
\mathrm{subject~to} \quad & \textrm{rank}(\tilde{H}_2) = n, \\ 
& \row(\tilde{H}_1) \subseteq \row(H_u), \\
&\row(\tilde{H}_2) \subseteq \nsp(H_u).
\end{aligned}
\end{equation}
The equivalence of \eqref{eqn:low-rank-approx} and \eqref{eqn:low-rank-approx-sep} comes from 
\begin{enumerate}
    \item For any feasible solution $\tilde{H}_y$ of \eqref{eqn:low-rank-approx}, we can let $\tilde{H}_1 = \tilde{H}_y \Pi_2$ and $\tilde{H}_w = \tilde{H}_y (I-\Pi_2)$, which satisfies constraints for \eqref{eqn:low-rank-approx-sep} and provides the same value for the objective function.
    \item Conversely, for any feasible solution $\tilde{H}_1$ nad $\tilde{H}_2$ of \eqref{eqn:low-rank-approx-sep}, we can let $\tilde{H}_y = \tilde{H}_1 + \tilde{H}_2$ which is feasible for \eqref{eqn:low-rank-approx} with the same objective function value.
\end{enumerate}
We can expand the objective function in \eqref{eqn:low-rank-approx-sep} as 
\[
\begin{aligned}
& \|H_1 + H_2 -\tilde{H}_1 - \tilde{H}_2\|_F  \\
= & \ \mathrm{tr}((H_1 + H_2 -\tilde{H}_1 - \tilde{H}_2)(H_1 + H_2 -\tilde{H}_1 - \tilde{H}_2)^\tr) \\
= & \ \mathrm{tr}((H_1 -\tilde{H}_1)(H_1 -\tilde{H}_1)^\tr) + \mathrm{tr}((H_2 -\tilde{H}_2)(H_2 -\tilde{H}_2)^\tr) \\
& \ + 2 \mathrm{tr}((H_1 -\tilde{H}_1)(H_2 -\tilde{H}_2)^\tr)) \\
= & \ \|H_1- \tilde{H}_1\|_F + \|H_2 -\tilde{H}_2\|_F,
\end{aligned}
\]
where the third equality is derived from $\row(H_1), \row(\tilde{H}_1) \subseteq \row(H_u)$ and $\row(H_2), \row(\tilde{H}_2) \subseteq \nsp(H_u)$. Thus, it is obvious that the optimal solution of $\tilde{H}_1$ is $H_1$ and we can simplify \eqref{eqn:low-rank-approx-sep} as 
\begin{equation}
\label{eqn:low-rank-approx-sim}
    \begin{aligned} 
\min_{\tilde{H}_2} \quad & \|H_2 - \tilde{H}_2 \|_F \\
\mathrm{subject~to} \quad & \textrm{rank}(\tilde{H}_2) = n, \\ 
&\row(\tilde{H}_2) \subseteq \nsp(H_u),
\end{aligned}
\end{equation}
which admits an analytical solution from the SVD that is $\tilde{H}_2^* = \sum_{i=1}^n \bar{\sigma}_i \bar{u}_i \bar{v}_i^\tr$ where $H_2 = \sum_{i=1}^{pL} \bar{\sigma}_i \bar{u}_i \bar{v}_i^\tr$. We note that $\row(\tilde{H}_2^*) \subseteq \row(H_2) \subseteq \nsp(H_u)$. That completes the proof.

\section{Additional example: Two-wheeled robot}
\label{appendix:two-wheeled-robot}
We here present the closed-loop control performance of \method{DDPC} variants and \method{Equality-form SPC} for a two-wheeled robot to illustrate the effectiveness of the framework in a more realistic application. The kinematic model of the robot is nonlinear~\cite{li2019online}
\begin{equation}
\label{eq:robot}
\begin{aligned}
    z_x(k+1) & = z_x(k) + \Delta t \cdot \cos(z_\delta(k))\cdot v(k), \\
    z_y(k+1) & = z_y(k) + \Delta t \cdot \sin(z_\delta(k))\cdot v(k), \\
    z_{\delta}(k+1) & = z_\delta(k) + \Delta t \cdot w_t,
\end{aligned}
\end{equation}
where $(z_x, z_y, z_\delta)$ denotes the position and heading angle of the robot (\emph{i.e.}, state), $v, w$ represent tangential and angular velocities (\emph{i.e.}, input) and $\Delta t=0.025\, \mathrm{s}$ is the sampling time. We aim to track the following heart-shaped reference trajectory $(r_x, r_y, r_\delta)$ while minimizing the control effort
\[
\begin{aligned}
    r_x(t) =&  \ 16\sin^3(t-6), \\ 
    r_y(t) =& \ 13\cos(t)-5\cos(2t-12) \\
    & \ -2\cos(3t-18)-\cos(4t-24), \\
    r_\delta(t) =& \ \arctan\left(\frac{r_y(t+1)-r_y(t)}{r_x(t+1)-r_x(t)}\right).
\end{aligned}
\]
We choose the cost matrices $Q=\mathrm{diag}(2,2,8), $ and $R=6.25\times 10^{-4}$, and set the weighting parameter $\lambda_y = 3\times 10^6$ for all approaches. The remaining parameters are tuned separately for each approach.

To improve the accuracy of the data-driven representation and the tracking performance, we partition the state space into different regions and construct a corresponding data-driven model for each of them. This can be viewed as a local linear approximation of the nonlinear system. Specifically, as the nonlinearity of the system mainly arises from the coupled terms $\sin(z_\delta)\cdot v$ and $\cos(z_\delta)\cdot v$ in states $z_x$ and $z_y$, we divide the state-space into $4$ regions according to the orientation, that are $[0, \frac{\pi}{2}), [\frac{\pi}{2}, \pi), [\pi, \frac{3\pi}{2})$ and $[\frac{3\pi}{2}, 2\pi)$. We collect $4$ trajectories of length $1200$ with initial states $(0, 0 ,\frac{\pi}{4}), (0, 0, \frac{3\pi}{4}), (0, 0, \frac{5\pi}{4})$ and $(0, 0, \frac{7\pi}{4})$ to capture the behavior of the nonlinear system within these regions. The input signals $v$ and $w$ are sampled uniformly from $[10, 20]$ and $[-\frac{\pi}{6}, \frac{\pi}{6}]$, respectively. Then, for each approach, we construct $4$ corresponding trajectory libraries over these regions with $T_\ini = 5$ and $N = 12$. During the online tracking, we switch the data library based on the robot's current orientation.

The closed-loop tracking performance of different approaches is shown in Fig. \ref{fig:robot-tracking}. All approaches are able to approximately track the prescribable reference trajectory. Among them, the \method{L-DDPC} and \method{SPC} exhibit noticeable deviation and oscillation behavior, while the \method{C-DDPC} and \method{O-DDPC} achieve superior tracking performance. 
\begin{figure}[t]
\setlength{\abovedisplayskip}{0pt}
\centering
\subfigure[\method{L-DDPC}]{\includegraphics[width=0.23\textwidth, trim={3mm 20mm 5mm 17mm},clip]{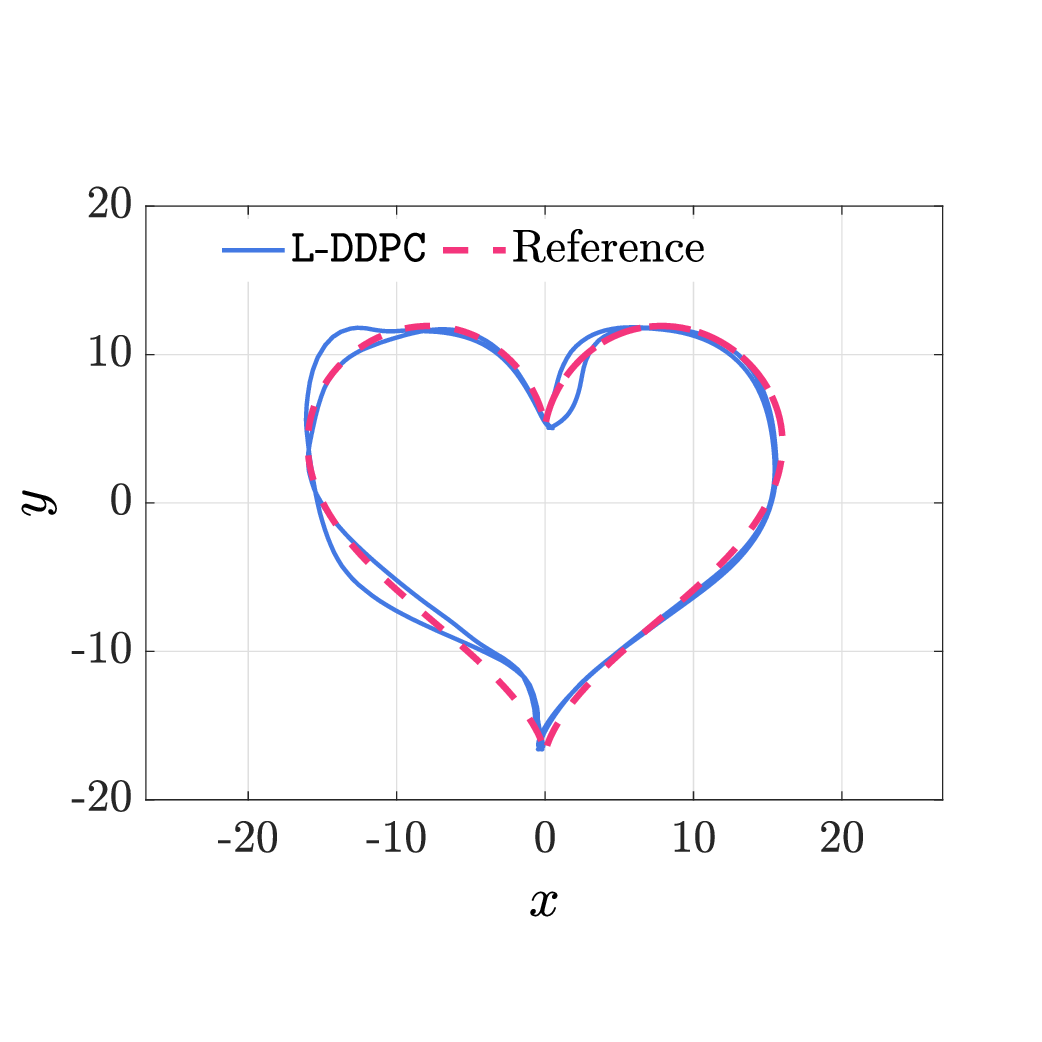}} \hspace{-2mm}
\subfigure[\method{SPC}]{\includegraphics[width=0.23\textwidth, trim={3mm 20mm 5mm 17mm},clip]{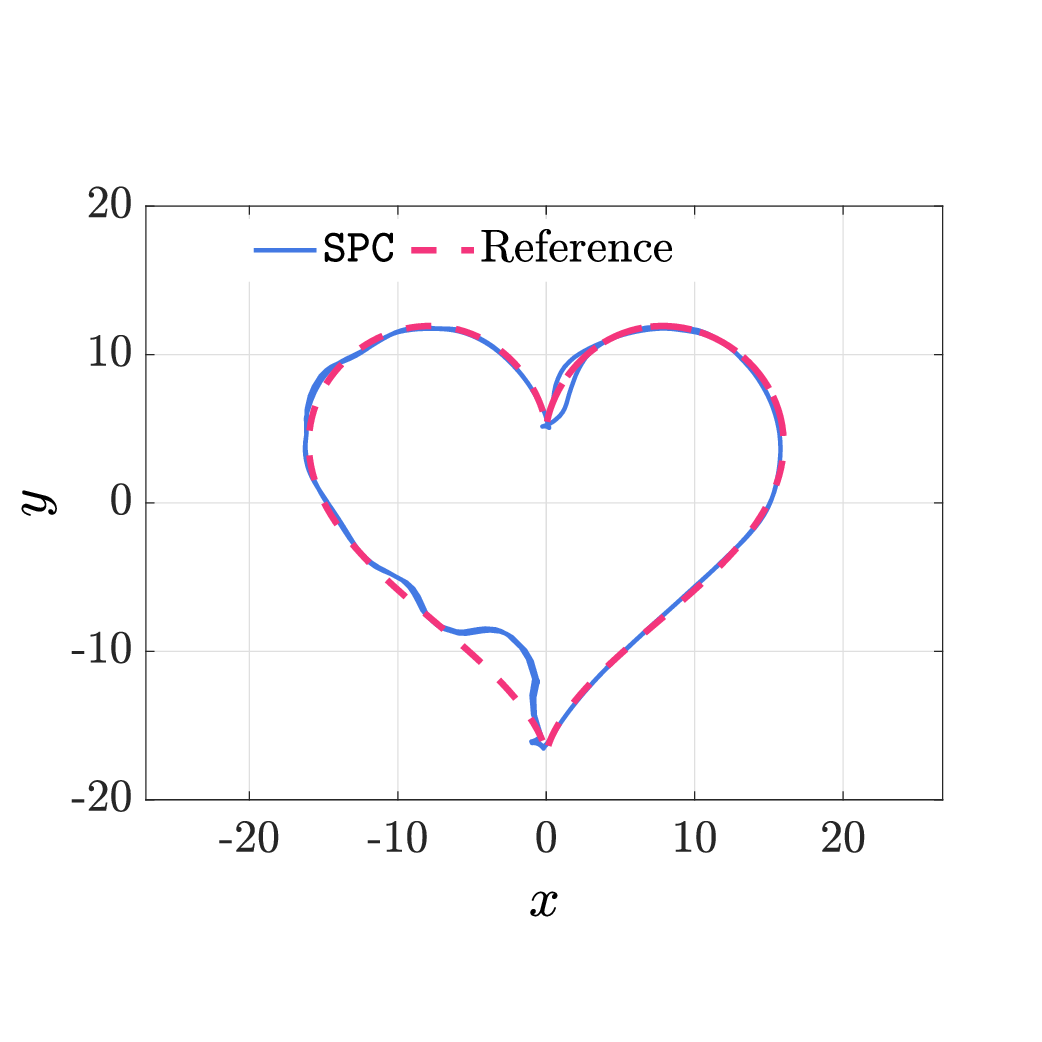}} \\
\vspace{-3mm}
\subfigure[\method{C-DDPC}]{\includegraphics[width=0.23\textwidth, trim={3mm 20mm 5mm 17mm},clip]{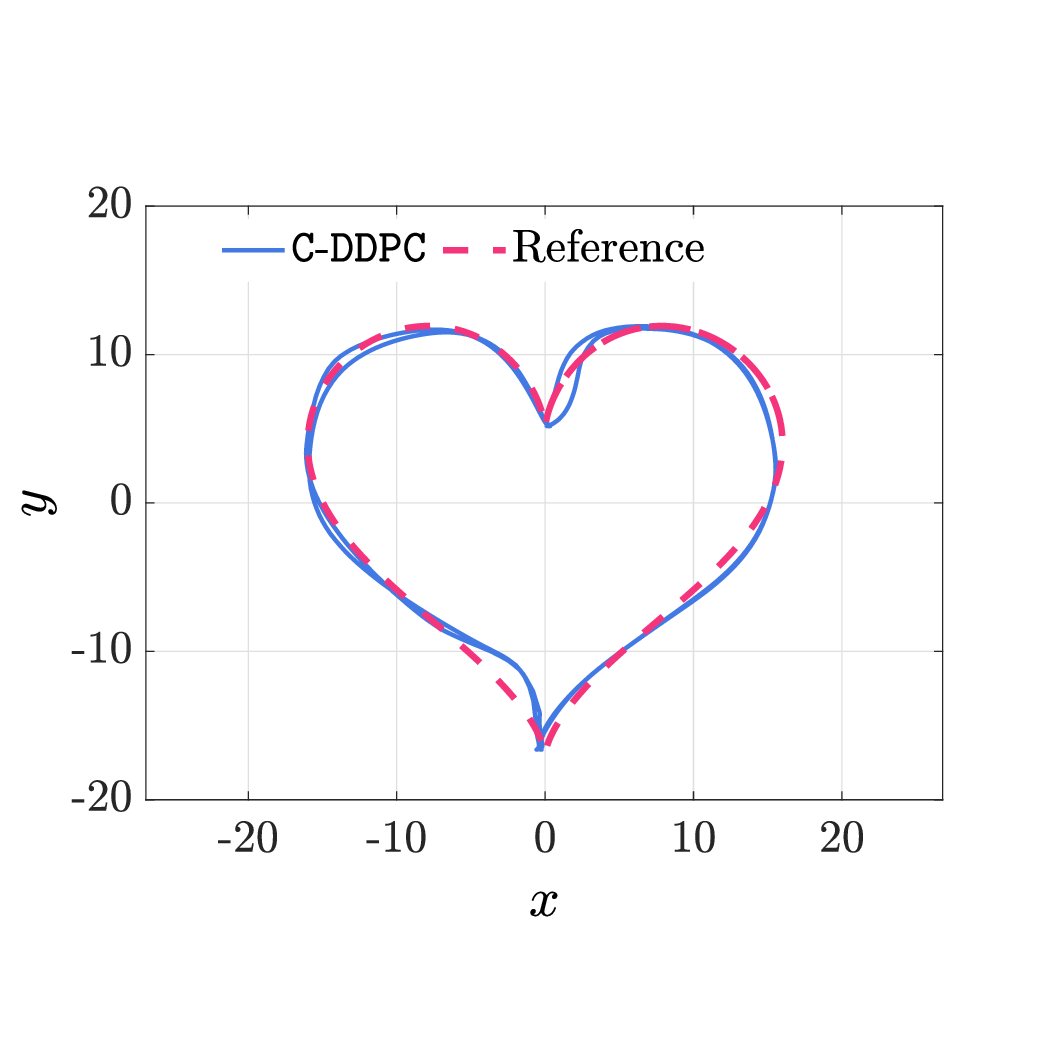}} \hspace{-2mm}
\subfigure[\method{A-DDPC}]{\includegraphics[width=0.23\textwidth, trim={3mm 20mm 5mm 17mm},clip]{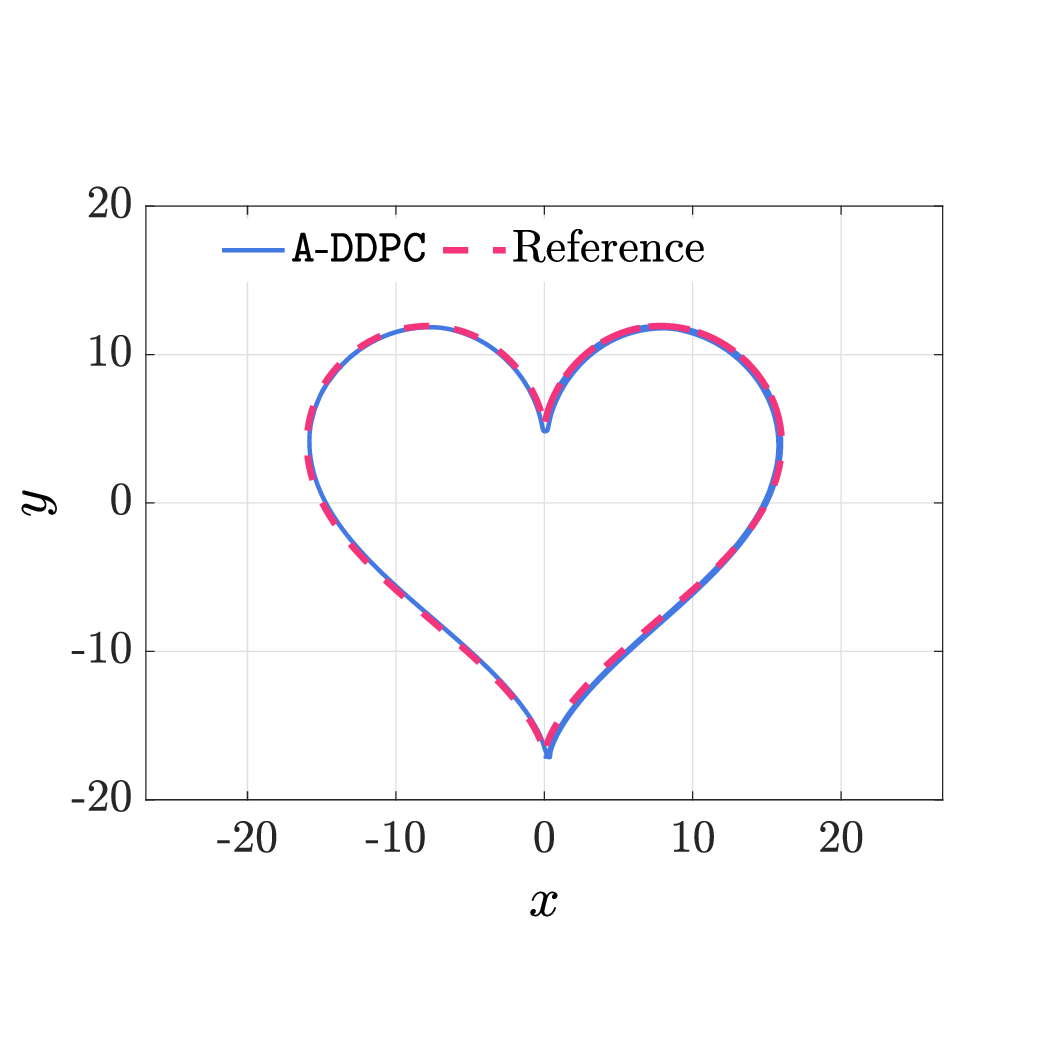}}  \\
\vspace{-2mm}
\caption{Closed-loop tracking performance of the two-wheel robot with different \method{DDPC} variants.}
\label{fig:robot-tracking}
\end{figure}
\begin{figure}[t]
\setlength{\abovedisplayskip}{0pt}
\centering
\subfigure[\method{L-DDPC}]{\includegraphics[width=0.23\textwidth, trim={3mm 20mm 5mm 17mm},clip]{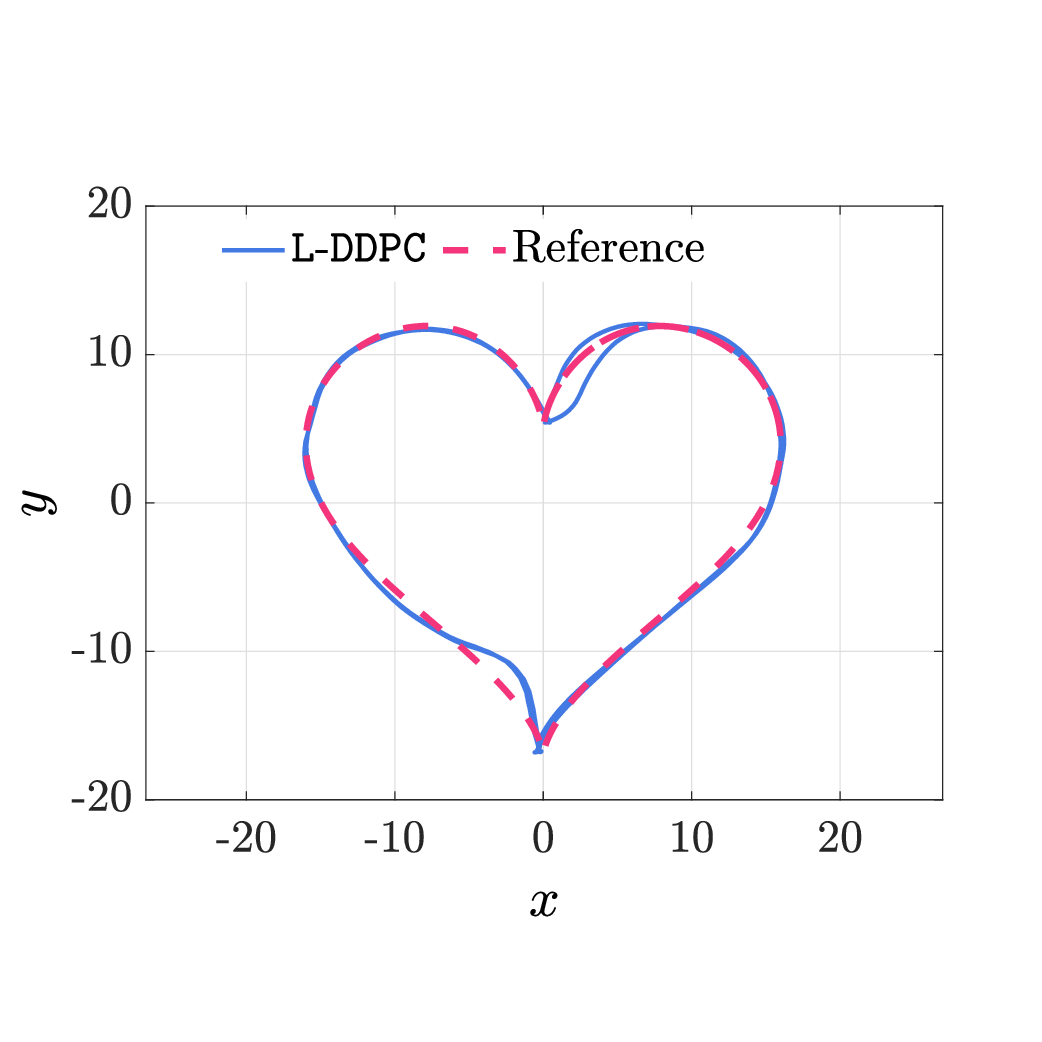}} \hspace{-2mm}
\subfigure[\method{SPC}]{\includegraphics[width=0.23\textwidth, trim={3mm 20mm 5mm 17mm},clip]{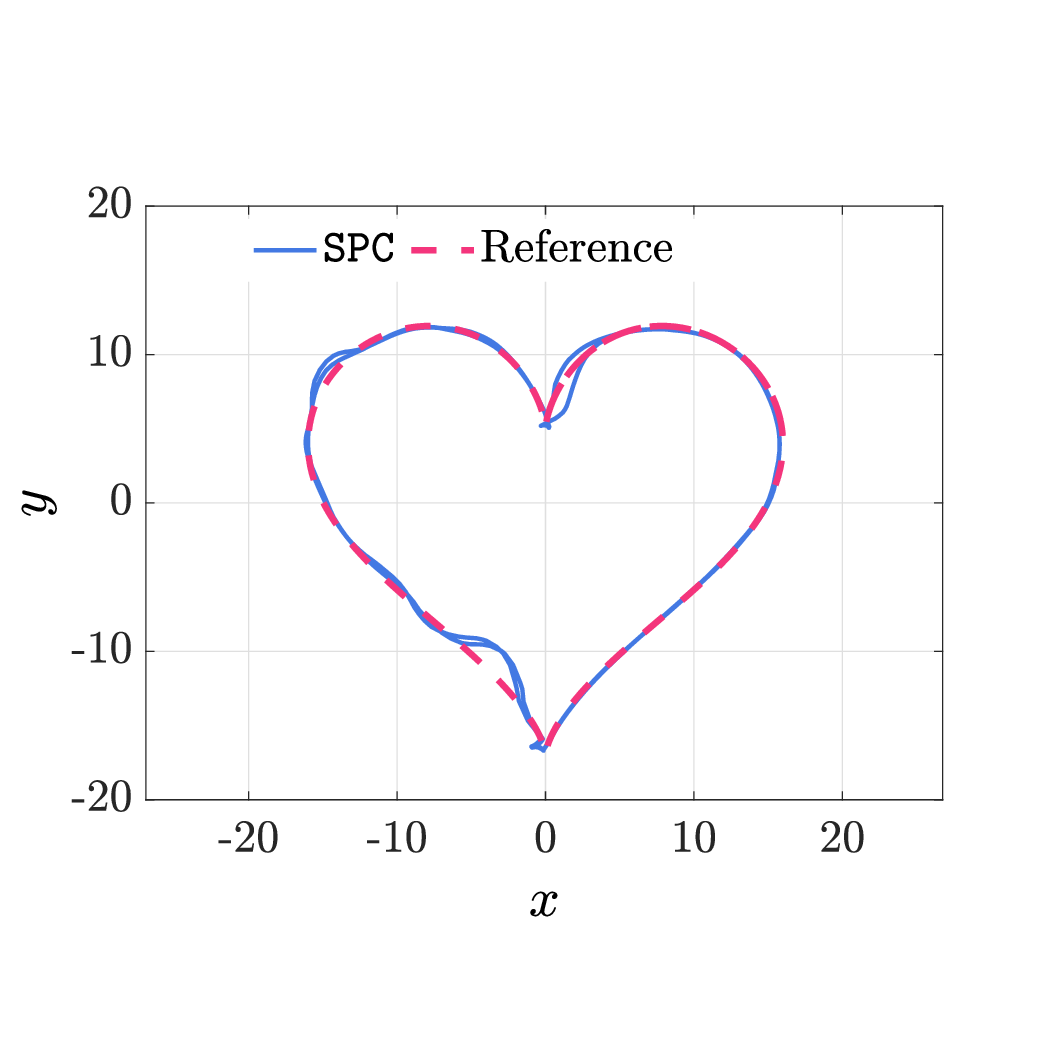}} \\
\vspace{-3mm}
\subfigure[\method{C-DDPC}]{\includegraphics[width=0.23\textwidth, trim={3mm 20mm 5mm 17mm},clip]{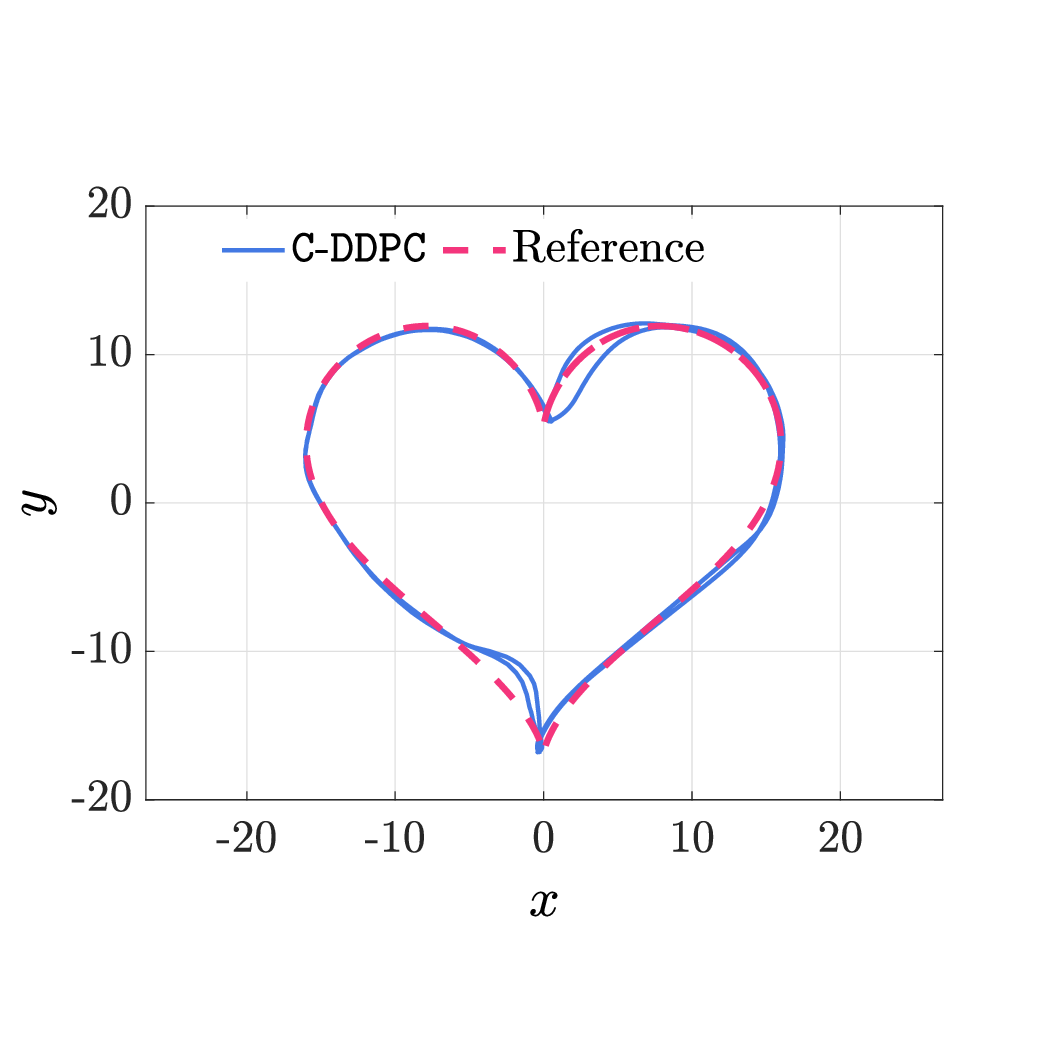}} \hspace{-2mm}
\subfigure[\method{A-DDPC}]{\includegraphics[width=0.23\textwidth, trim={3mm 20mm 5mm 17mm},clip]{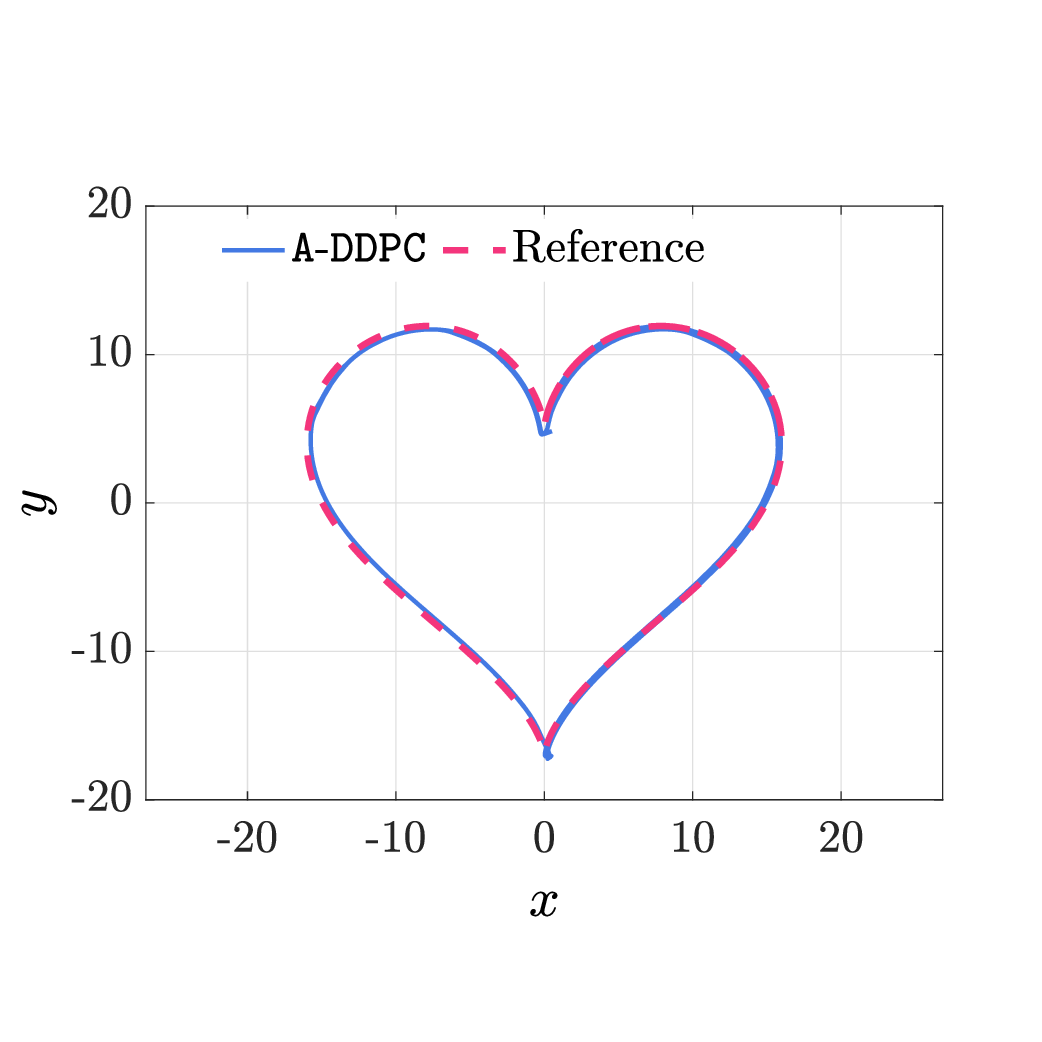}}  \\
\vspace{-2mm}
\caption{Closed-loop tracking performance of the two-wheel robot with different \method{DDPC} variants.}
\label{fig:robot-tracking-2}
\end{figure}

We further report the computation times of \method{DDPC} variants in Table \ref{table:computation-time-two-wheeled}. The optimization problems are solved using Mosek \cite{mosek} on an Intel Core i7-9750H CPU with 16GB RAM and each controller is executed $100$ times to obtain the mean computation time. The computation times of all approaches are around $3\times10^{-2}\,\mathrm{s}$, which are relatively small, although slightly larger than the discretization time $0.025\,\mathrm{s}$, with a gap of around $5 \times 10^{-3}\,\mathrm{s}$. We can compute the control input less frequently and apply multi-step inputs from the controller between updates. Fig. \ref{fig:robot-tracking-2} shows the control performance of the \method{DDPC} variants when the optimal control input is updated every $0.05\,\mathrm{s}$ and two control inputs are applied at each update. The results remain satisfactory and meet the real-time implementation requirement. Furthermore, this gap could potentially be reduced by developing a solver tailored to the specific structure of our problem.
\begin{table}[h]
\caption{Computation time of \method{DDPC} variants (Unit $\mathrm{s}$)}
\centering
\begin{threeparttable}
    \begin{tabular}{ccccc}
    \toprule
    \noalign{\vskip -2pt}
     &\method{L-DDPC} & \method{C-DDPC} &\method{O-DDPC} & \method{SPC}  \\
     \noalign{\vskip -2pt}
    \midrule
     Two-wheeled robot &0.030 &0.032 & 0.029 & 0.0316 \\
    \bottomrule
    \end{tabular}
    \end{threeparttable}
    \label{table:computation-time-two-wheeled}
\end{table}

%% file: reference.bib
@article{dorfler2022bridging,
  title={Bridging direct and indirect data-driven control formulations via regularizations and relaxations},
  author={D{\"o}rfler, Florian and Coulson, Jeremy and Markovsky, Ivan},
  journal={IEEE Transactions on Automatic Control},
  volume={68},
  number={2},
  pages={883--897},
  year={2022},
  publisher={IEEE}
}

@ARTICLE{Shang2025Decentralized,
  author={Shang, Xu and Wang, Jiawei and Zheng, Yang},
  journal={IEEE Transactions on Intelligent Transportation Systems}, 
  title={Decentralized Robust Data-Driven Predictive Control for Smoothing Mixed Traffic Flow}, 
  year={2025},
  volume={26},
  number={2},
  pages={2075-2090},
  doi={10.1109/TITS.2024.3514117}}

@article{wang2023implementation,
  title={Implementation and experimental validation of data-driven predictive control for dissipating stop-and-go waves in mixed traffic},
  author={Wang, Jiawei and Zheng, Yang and Dong, Jianghong and Chen, Chaoyi and Cai, Mengchi and Li, Keqiang and Xu, Qing},
  journal={IEEE Internet of Things Journal},
  volume={11},
  number={3},
  pages={4570--4585},
  year={2023},
  publisher={IEEE}
}

@inproceedings{zheng2021sample,
  title={Sample complexity of linear quadratic gaussian ({LQG}) control for output feedback systems},
  author={Zheng, Yang and Furieri, Luca and Kamgarpour, Maryam and Li, Na},
  booktitle={3rd Annual Learning for Dynamics \& Control Conference},
  pages={559--570},
  year={2021},
  organization={PMLR}
}

@article{talebi2024policy,
  title={Policy optimization in control: Geometry and algorithmic implications},
  author={Talebi, Shahriar and Zheng, Yang and Kraisler, Spencer and Li, Na and Mesbahi, Mehran},
  journal={arXiv preprint arXiv:2406.04243},
  year={2024}
}

@article{hu2023toward,
  title={Toward a theoretical foundation of policy optimization for learning control policies},
  author={Hu, Bin and Zhang, Kaiqing and Li, Na and Mesbahi, Mehran and Fazel, Maryam and Ba{\c{s}}ar, Tamer},
  journal={Annual Review of Control, Robotics, and Autonomous Systems},
  volume={6},
  number={1},
  pages={123--158},
  year={2023},
  publisher={Annual Reviews}
}

@article{dorfler2023data,
  title={Data-driven control: Part two of two: Hot take: Why not go with models?},
  author={D{\"o}rfler, Florian},
  journal={IEEE Control Systems Magazine},
  volume={43},
  number={6},
  pages={27--31},
  year={2023},
  publisher={IEEE}
}

@article{lian2023adaptive,
  title={Adaptive robust data-driven building control via bilevel reformulation: An experimental result},
  author={Lian, Yingzhao and Shi, Jicheng and Koch, Manuel and Jones, Colin Neil},
  journal={IEEE Transactions on Control Systems Technology},
  year={2023},
  publisher={IEEE}
}

@article{wang2023deep,
  title={{DeeP-LCC}: Data-enabled predictive leading cruise control in mixed traffic flow},
  author={Wang, Jiawei and Zheng, Yang and Li, Keqiang and Xu, Qing},
  journal={IEEE Transactions on Control Systems Technology},
  year={2023},
  volume={31},
  number={6},
  pages={2760-2776},
  publisher={IEEE}
}

@article{mattsson2024equivalence,
  title={On the equivalence of direct and indirect data-driven predictive control approaches},
  author={Mattsson, Per and Bonassi, Fabio and Breschi, Valentina and Sch{\"o}n, Thomas B},
  journal={IEEE Control Systems Letters},
  year={2024},
  publisher={IEEE}
}

@article{berberich2024overview,
  title={An overview of systems-theoretic guarantees in data-driven model predictive control},
  author={Berberich, Julian and Allg{\"o}wer, Frank},
  journal={Annual Review of Control, Robotics, and Autonomous Systems},
  volume={8},
  year={2024},
  publisher={Annual Reviews}
}

@article{willems2005note,
  title={A note on persistency of excitation},
  author={Willems, Jan C and Rapisarda, Paolo and Markovsky, Ivan and De Moor, Bart LM},
  journal={Systems \& Control Letters},
  volume={54},
  number={4},
  pages={325--329},
  year={2005},
  publisher={Elsevier}
}

@inproceedings{coulson2019data,
  title={Data-enabled predictive control: In the shallows of the {D}ee{PC}},
  author={Coulson, Jeremy and Lygeros, John and D{\"o}rfler, Florian},
  booktitle={2019 18th European Control Conference (ECC)},
  pages={307--312},
  year={2019},
  organization={IEEE}
}

@article{favoreel1999spc,
  title={{SPC}: Subspace predictive control},
  author={Favoreel, Wouter and De Moor, Bart and Gevers, Michel},
  journal={IFAC Proceedings Volumes},
  volume={32},
  number={2},
  pages={4004--4009},
  year={1999},
  publisher={Elsevier}
}

@article{markovsky2016missing,
  title={A missing data approach to data-driven filtering and control},
  author={Markovsky, Ivan},
  journal={IEEE Transactions on Automatic Control},
  volume={62},
  number={4},
  pages={1972--1978},
  year={2016},
  publisher={IEEE}
}

@inproceedings{fiedler2021relationship,
  title={On the relationship between data-enabled predictive control and subspace predictive control},
  author={Fiedler, Felix and Lucia, Sergio},
  booktitle={2021 European Control Conference (ECC)},
  pages={222--229},
  year={2021},
  organization={IEEE}
}

@article{alsalti2023data,
  title={Data-based system representations from irregularly measured data},
  author={Alsalti, Mohammad and Markovsky, Ivan and Lopez, Victor G and M{\"u}ller, Matthias A},
  journal={arXiv preprint arXiv:2307.11589},
  year={2023}
}

@article{zhang2023dimension,
  title={Dimension reduction for efficient data-enabled predictive control},
  author={Zhang, Kaixiang and Zheng, Yang and Shang, Chao and Li, Zhaojian},
  journal={IEEE Control Systems Letters},
    year={2023},
  volume={7},
  number={},
  pages={3277-3282},
  publisher={IEEE}
}

@article{markovsky2008structured,
  title={Structured low-rank approximation and its applications},
  author={Markovsky, Ivan},
  journal={Automatica},
  volume={44},
  number={4},
  pages={891--909},
  year={2008},
  publisher={Elsevier}
}

@article{yin2021low,
  title={On low-rank {H}ankel matrix denoising},
  author={Yin, Mingzhou and Smith, Roy S},
  journal={IFAC-PapersOnLine},
  volume={54},
  number={7},
  pages={198--203},
  year={2021},
  publisher={Elsevier}
}

@article{markovsky2021behavioral,
  title={Behavioral systems theory in data-driven analysis, signal processing, and control},
  author={Markovsky, Ivan and D{\"o}rfler, Florian},
  journal={Annual Reviews in Control},
  volume={52},
  pages={42--64},
  year={2021},
  publisher={Elsevier}
}

@incollection{ljung1998system,
  title={System identification},
  author={Ljung, Lennart},
  booktitle={Signal analysis and prediction},
  pages={163--173},
  year={1998},
  publisher={Springer}
}

@article{chiuso2019system,
  title={System identification: A machine learning perspective},
  author={Chiuso, Alessandro and Pillonetto, Gianluigi},
  journal={Annual Review of Control, Robotics, and Autonomous Systems},
  volume={2},
  pages={281--304},
  year={2019},
  publisher={Annual Reviews}
}

@article{elokda2021data,
  title={Data-enabled predictive control for quadcopters},
  author={Elokda, Ezzat and Coulson, Jeremy and Beuchat, Paul N and Lygeros, John and D{\"o}rfler, Florian},
  journal={International Journal of Robust and Nonlinear Control},
  volume={31},
  number={18},
  pages={8916--8936},
  year={2021},
  publisher={Wiley Online Library}
}

@article{berberich2020data,
  title={Data-driven model predictive control with stability and robustness guarantees},
  author={Berberich, Julian and K{\"o}hler, Johannes and M{\"u}ller, Matthias A and Allg{\"o}wer, Frank},
  journal={IEEE Transactions on Automatic Control},
  volume={66},
  number={4},
  pages={1702--1717},
  year={2020},
  publisher={IEEE}
}

@article{kouvaritakis2016model,
  title={Model predictive control},
  author={Kouvaritakis, Basil and Cannon, Mark},
  journal={Switzerland: Springer International Publishing},
  volume={38},
  year={2016},
  publisher={Springer}
}

@article{breschi2023data,
  title={Data-driven predictive control in a stochastic setting: A unified framework},
  author={Breschi, Valentina and Chiuso, Alessandro and Formentin, Simone},
  journal={Automatica},
  volume={152},
  pages={110961},
  year={2023},
  publisher={Elsevier}
}

@article{van1994n4sid,
  title={N4{SID}: Subspace algorithms for the identification of combined deterministic-stochastic systems},
  author={Van Overschee, Peter and De Moor, Bart},
  journal={Automatica},
  volume={30},
  number={1},
  pages={75--93},
  year={1994},
  publisher={Elsevier}
}

@article{yin2021maximum,
  title={Maximum likelihood estimation in data-driven modeling and control},
  author={Yin, Mingzhou and Iannelli, Andrea and Smith, Roy S},
  journal={IEEE Transactions on Automatic Control},
  volume={68},
  number={1},
  pages={317--328},
  year={2021},
  publisher={IEEE}
}

@article{berberich2021design,
  title={On the design of terminal ingredients for data-driven {MPC}},
  author={Berberich, Julian and K{\"o}hler, Johannes and M{\"u}ller, Matthias A and Allg{\"o}wer, Frank},
  journal={IFAC-PapersOnLine},
  volume={54},
  number={6},
  pages={257--263},
  year={2021},
  publisher={Elsevier}
}

@inproceedings{shang2024convex,
  title={Convex approximations for a bi-level formulation of data-enabled predictive control},
  author={Shang, Xu and Zheng, Yang},
  booktitle={6th Annual Learning for Dynamics \& Control Conference},
  pages={1071--1082},
  year={2024},
  organization={PMLR}
}

@article{kladtke2023implicit,
  title={Implicit predictors in regularized data-driven predictive control},
  author={Kl{\"a}dtke, Manuel and Darup, Moritz Schulze},
  journal={IEEE Control Systems Letters},
  volume={7},
  pages={2479--2484},
  year={2023},
  publisher={IEEE}
}

@article{sader2023causality,
  title={Causality-informed data-driven predictive control},
  author={Sader, Malika and Wang, Yibo and Huang, Dexian and Shang, Chao and Huang, Biao},
  journal={IEEE Transactions on Control Systems Technology},
  year={2025},
  publisher={IEEE}
}

@article{qin2005novel,
  title={A novel subspace identification approach with enforced causal models},
  author={Qin, S Joe and Lin, Weilu and Ljung, Lennart},
  journal={Automatica},
  volume={41},
  number={12},
  pages={2043--2053},
  year={2005},
  publisher={Elsevier}
}

@article{shang2024willems,
  title={Willems’ fundamental lemma for nonlinear systems with {K}oopman linear embedding},
  author={Shang, Xu and Cort{\'e}s, Jorge and Zheng, Yang},
  journal={IEEE Control Systems Letters},
  volume={8},
  pages={3135--3140},
  year={2024},
  publisher={IEEE}
}

@article{chiuso2023harnessing,
  title={Harnessing the final control error for optimal data-driven predictive control},
  author={Chiuso, Alessandro and Fabris, Marco and Breschi, Valentina and Formentin, Simone},
  journal={arXiv preprint arXiv:2312.14788},
  year={2023}
}

@article{berberich2022linear,
  title={Linear tracking {MPC} for nonlinear systems—{P}art {II}: The data-driven case},
  author={Berberich, Julian and K{\"o}hler, Johannes and M{\"u}ller, Matthias A and Allg{\"o}wer, Frank},
  journal={IEEE Transactions on Automatic Control},
  volume={67},
  number={9},
  pages={4406--4421},
  year={2022},
  publisher={IEEE}
}

@book{ruszczynski2011nonlinear,
  title={Nonlinear optimization},
  author={Ruszczynski, Andrzej},
  year={2011},
  publisher={Princeton university press}
}

@article{ye1997exact,
  title={Exact penalization and necessary optimality conditions for generalized bilevel programming problems},
  author={Ye, JJ and Zhu, DL and Zhu, Qiji Jim},
  journal={SIAM Journal on optimization},
  volume={7},
  number={2},
  pages={481--507},
  year={1997},
  publisher={SIAM}
}

@inproceedings{lawrence2024deep,
  title={Deep {H}ankel matrices with random elements},
  author={Lawrence, Nathan and Loewen, Philip and Wang, Shuyuan and Forbes, Michael and Gopaluni, Bhushan},
  booktitle={6th Annual Learning for Dynamics \& Control Conference},
  pages={1579--1591},
  year={2024},
  organization={PMLR}
}

@article{korda2018linear,
  title={Linear predictors for nonlinear dynamical systems: Koopman operator meets model predictive control},
  author={Korda, Milan and Mezi{\'c}, Igor},
  journal={Automatica},
  volume={93},
  pages={149--160},
  year={2018},
  publisher={Elsevier}
}

@article{haseli2023modeling,
  title={Modeling nonlinear control systems via {K}oopman control family: Universal forms and subspace invariance proximity},
  author={Haseli, Masih and Cort{\'e}s, Jorge},
  journal={Automatica},
  volume={185},
  pages={112722},
  year={2026},
  publisher={Elsevier}
}

@book{mauroy2020koopman,
  title={Koopman operator in systems and control},
  author={Mauroy, Alexandre and Susuki, Y and Mezic, Igor},
  year={2020},
  publisher={Springer}
}

@article{andersson2013alternating,
  title={Alternating projections on nontangential manifolds},
  author={Andersson, Fredrik and Carlsson, Marcus},
  journal={Constructive approximation},
  volume={38},
  number={3},
  pages={489--525},
  year={2013},
  publisher={Springer}
}

@article{huang2021quadratic,
  title={Quadratic regularization of data-enabled predictive control: Theory and application to power converter experiments},
  author={Huang, Linbin and Zhen, Jianzhe and Lygeros, John and D{\"o}rfler, Florian},
  journal={IFAC-PapersOnLine},
  volume={54},
  number={7},
  pages={192--197},
  year={2021},
  publisher={Elsevier}
}

@article{huang2021decentralized,
  title={Decentralized data-enabled predictive control for power system oscillation damping},
  author={Huang, Linbin and Coulson, Jeremy and Lygeros, John and D{\"o}rfler, Florian},
  journal={IEEE Transactions on Control Systems Technology},
  volume={30},
  number={3},
  pages={1065--1077},
  year={2021},
  publisher={IEEE}
}

@book{clarke1990optimization,
  title={Optimization and nonsmooth analysis},
  author={Clarke, Frank H},
  year={1990},
  publisher={SIAM}
}

@article{li2019online,
  title={Online optimal control with linear dynamics and predictions: Algorithms and regret analysis},
  author={Li, Yingying and Chen, Xin and Li, Na},
  journal={Advances in Neural Information Processing Systems},
  volume={32},
  year={2019}
}

@incollection{mosek,
  title={The MOSEK interior point optimizer for linear
programming: an implementation of the homogeneous
algorithm},
  author={Andersen, Erling D and Andersen, Knud D},
  booktitle={High Performance Optimization},
  pages={197--232},
  year={2000},
  publisher={Springer}
}

@inproceedings{shang2025dictionary,
  title={Dictionary-free {K}oopman predictive control for autonomous vehicles in mixed traffic},
  author={Shang, Xu and Li, Zhaojian and Zheng, Yang},
  booktitle={2025 IEEE Conference on Control Technology and Applications (CCTA)},
  pages={571--576},
  year={2025},
  organization={IEEE}
}

@inproceedings{yang2015data,
  title={A data-driven predictive controller design based on reduced {H}ankel matrix},
  author={Yang, Hua and Li, Shaoyuan},
  booktitle={2015 10th Asian Control Conference (ASCC)},
  pages={1--7},
  year={2015},
  organization={IEEE}
}

@article{hoffman1952approximate,
  title={On Approximate Solutions of Systems of Linear Inequalities},
  author={Hoffman, Alan J},
  journal={Journal of Research of the National Bureau of Standards},
  volume={49},
  number={4},
  year={1952}
}

@article{shang2026existence,
  title={On the existence of {K}oopman linear embeddings for controlled nonlinear systems},
  author={Shang, Xu and Haseli, Masih and Cort{\'e}s, Jorge and Zheng, Yang},
  journal={arXiv preprint arXiv:2602.14537},
  year={2026}
}

@article{fawcett2022distributed,
  title={Distributed data-driven predictive control for multi-agent collaborative legged locomotion},
  author={Fawcett, Randall T and Amanzadeh, Leila and Kim, Jeeseop and Ames, Aaron D and Hamed, Kaveh Akbari},
  journal={arXiv preprint arXiv:2211.06917},
  year={2022}
}

@article{verhoek2021data,
  title={Data-driven predictive control for linear parameter-varying systems},
  author={Verhoek, Chris and Abbas, Hossam S and T{\'o}th, Roland and Haesaert, Sofie},
  journal={IFAC-PapersOnLine},
  volume={54},
  number={8},
  pages={101--108},
  year={2021},
  publisher={Elsevier}
}

@inproceedings{lian2021nonlinear,
  title={Nonlinear data-enabled prediction and control},
  author={Lian, Yingzhao and Jones, Colin N},
  booktitle={3rd Annual Learning for Dynamics \& Control Conference},
  pages={523--534},
  year={2021},
  organization={PMLR}
}

@inproceedings{lazar2024basis,
  title={Basis-functions nonlinear data-enabled predictive control: Consistent and computationally efficient formulations},
  author={Lazar, Mircea},
  booktitle={2024 European Control Conference (ECC)},
  pages={888--893},
  year={2024},
  organization={IEEE}
}

@article{alsalti2024robust,
  title={Robust and efficient data-driven predictive control},
  author={Alsalti, Mohammad and Barkey, Manuel and Lopez, Victor G and M{\"u}ller, Matthias A},
  journal={arXiv preprint arXiv:2409.18867},
  year={2024}
}

@article{o2022data,
  title={Data-driven predictive control with improved performance using segmented trajectories},
  author={O’Dwyer, Edward and Kerrigan, Eric C and Falugi, Paola and Zagorowska, Marta and Shah, Nilay},
  journal={IEEE Transactions on Control Systems Technology},
  volume={31},
  number={3},
  pages={1355--1365},
  year={2022},
  publisher={IEEE}
}

@article{bongard2022robust,
  title={Robust stability analysis of a simple data-driven model predictive control approach},
  author={Bongard, Joscha and Berberich, Julian and K{\"o}hler, Johannes and Allg{\"o}wer, Frank},
  journal={IEEE Transactions on Automatic Control},
  volume={68},
  number={5},
  pages={2625--2637},
  year={2022},
  publisher={IEEE}
}

@article{naf2025choose,
  title={Choose wisely: Data-driven predictive control for nonlinear systems using online data selection},
  author={N{\"a}f, Joshua and Moffat, Keith and Eising, Jaap and D{\"o}rfler, Florian},
  journal={arXiv preprint arXiv:2503.18845},
  year={2025}
}
